\documentclass[12pt]{amsart}

\usepackage{frontmatter} 
\usepackage{todonotes}
\title{Scaling Relations for Auxin Waves}

\author[T.E. Faver]{Timothy E. Faver}
\address{Department of Mathematics, Kennesaw State University, 850 Polytechnic Lane, Marietta, GA 30060 USA, {\tt{tfaver1@kennesaw.edu}}}

\author[H.J. Hupkes]{Hermen Jan Hupkes}
\address{Mathematical Institute, Universiteit Leiden, P.O. Box 9512, 2300 RA Leiden, The Netherlands, {\tt{hhupkes@math.leidenuniv.nl}}}

\author[R.M.H. Merks]{Roeland M. H. Merks}
\address{Mathematical Institute and Institute of Biology Leiden, Universiteit Leiden, P.O. Box 9512, 2300 RA Leiden, The Netherlands, {\tt{merksrmh@math.leidenuniv.nl}}}

\author[J. van der Voort]{Jelle van der Voort}
\address{Mathematical Institute, Universiteit Leiden, P.O. Box 9512, 2300 RA Leiden, The Netherlands, {\tt{jelvoort@live.nl}}}



\date{\today}

\begin{document}

\begin{abstract}
We analyze an `up-the-gradient' model for the formation of transport channels of the phytohormone auxin, through auxin-mediated polarization of the PIN1 auxin transporter. We show that this model admits a family of travelling wave solutions that is parameterized by the height of the auxin-pulse. We uncover scaling relations for the speed and width of these waves and verify these rigorous results with numerical computations. In addition, we provide explicit expressions for the leading-order wave profiles, which allows the influence of the biological parameters in the problem to be readily identified. Our proofs are based on a generalization of the scaling principle developed by Friesecke and Pego to construct pulse solutions to the classic Fermi-Pasta-Ulam-Tsingou model, which describes a one-dimensional chain of coupled nonlinear springs.
\end{abstract}

\maketitle

%%-----------------------------------------------------------%%
%%-----------------------------------------------------------%%
%%-----------------------------------------------------------%%
%%-----------------------------------------------------------%%
%%-----------------------------------------------------------%%
\section{Introduction}

\subsection{Polar auxin transport}
The phytohormone auxin is a central player in practically all aspects of the development and growth of plants, for example in phyllotaxis, root development and the initiation of lateral roots, the formation of vascular tissues in stems, the patterning of leaf veins, and flower development \cite{Paque2016}. The pattern formation principles underlying these developmental mechanisms have been uncovered to a large part through an intensive cross-talk between experimental approaches and mathematical modeling  \cite{Shi2018,Autran2021,Cieslak2021}. Auxin is transported between cells and between cells and the cell walls both through diffusion and through transport proteins that are localized at the cell membrane of the cell. These transport proteins are distributed in a polarized manner inside the cells, and the polarization of adjacent cells is coordinated in plant tissue,  leading to a directed transport of auxin through plant tissues in a mechanism called polar auxin transport (PAT) \cite{Adamowski2015}. For example, in fully developed seed plants, auxin is synthesized in leaves, then is transported through the central tissues of the stem and the root towards the root tips, where it redirected along the superficial tissues of the root back to towards the stem and recycled towards the internal tissues of the root \cite{Adamowski2015}.

    Despite new details being uncovered incessantly (see e.g.~\cite{Hajny2020}), it is still incompletely understood what mechanisms drive the polarization of auxin transporters inside cells and the coordinated polarization among adjacent cells. In a series of classical experiments, Sachs applied artificial auxin to bean plants, and observed that these become the source new vascular tissue that then joins the existing vasculature; see e.g.~\cite{Sachs1975} and the review~\cite{Hajny2022}. These initial observations, together with the discovery of auxin transporters including PIN1 suggested that auxin drives the polarization of its own transporters, and hence the direction of its own transport (reviewed in~\cite{merks2007canalization,Hajny2022}). Initial models aimed to explain the formation of transport channels as observed in Sachs' experiments. These models therefore assumed that the rate of auxin flux from cell to cell further polarised auxin transport. This positive feedback led to the self-organised formation of auxin transport channels in a process called auxin canalisation. When it was realised that auxin accumulations mark the formation of a new leaves at the shoot apex, an alternative model was proposed, in which cells polarised towards the locally increased concentrations of auxin, thus forming self-organised accumulation of auxin \cite{Reinhardt:2003ww}. Mathematical models of the self-organisation of polar auxin transport therefore follow  these two broad 
categories. `With-the-gradient' models formalise the canalisation hypothesis and assume that the rate of cell polarisation depends
on the auxin \textit{flux} towards the relevant neighbour~\cite{Mitchison1980,mitchison1981polar,RollandLagan:2005ej,RollandLagan:2008db}. 'Up-the-gradient' models assume that PIN polarizes in the direction of neighbouring cells
at a rate that positively depends on the auxin \textit{concentration} in that neighbour~\cite{Jonsson.pnas06,Smith:2006du}. Attempts to reconcile these two seemingly contradicting ideas have followed two broad approaches. The first approach proposed that with-the-gradient and up-the-gradient models act at different positions of the plant or at different stages during development. For example Bayer et al.~\cite{Bayer:2009iq} proposed that the up-the-gradient model act at superficial tissue layers of the shoot apical meristem where it forms auxin accumulation points leading to the initial of new leaves. The deeper tissue layers could follow the with-the-gradient model channeling auxin away from the auxin accumulation point towards the vascular tissues \cite{Bayer:2009iq}. A similar approach was recently taken to explain the leaf venation patterning in combination with auxin convergence at the edge of the leaf primordium~\cite{Holloway2021}. The second approach looked for variants of the with-the-gradient or up-the-gradient models that could explain both auxin canalisation and auxin canalisation depending on the parameter settings. In this line of reasoning Walker et al. have proposed a with-the-gradient hypothesis for phyllotaxis \cite{Walker:2013wl}, whereas one of us has proposed an up-the-gradient hypothesis for canalisation~\cite{merks2007canalization}. 

\subsection{Mathematical motivation}
In order to distinguish between the available phenomenological models of auxin-driven pattern formation and the general developmental principles that they represent, mathematical insight into the models' structure and the models' solutions will be crucial. This will help pinpoint key differences between the model structures and may uncover potential structural instabilities in the models upon which evolution may have acted, so as to produce new developmental patterning modules~\cite{Benitez2018}.    
From the mathematical side, 
almost all previous studies have focused on the types of patterns that can be generated by different models once the transitory dynamics have died out. An important example is the study
by Van Berkel and coworkers \cite{van2013polar},
where a number of models for polar auxin transport are recast into a common mathematical framework that allows them to be compared. 
A steady state analysis for a general class of 
active transport models can be found in \cite{draelants2015localized},
using advanced tools such as snaking from the field of bifurcation theory. Both periodic and stationary patterns are examined in \cite{allen2020mathematical}, where the authors consider an extended
`with-the-flux' model. Haskovec and his coworkers derive local and global existence results together with an appropriate continuum limit
for their graph-based diffusion model in \cite{haskovec2019auxin}.

Important qualitative examples of the with-the-gradient model are the formation of regularly spaced auxin maximums that lead to the growth of new leaves, as well as the formation of auxin channels that precede the formation of veins. Our goal here is to move beyond the well-studied equilibrium settings above and focus instead on understanding the dynamical behavior that leads to these patterns. In particular, we  provide a rigorous framework to study a class of wave solutions that underpin the dynamical behaviour associated to with-the-gradient model. Ultimately, we hope that this analytic approach will provide an additional lens through which models of PAT can be examined and compared.

\subsection{The model}

 \begin{figure}[t]
\centering\includegraphics[width=\textwidth]{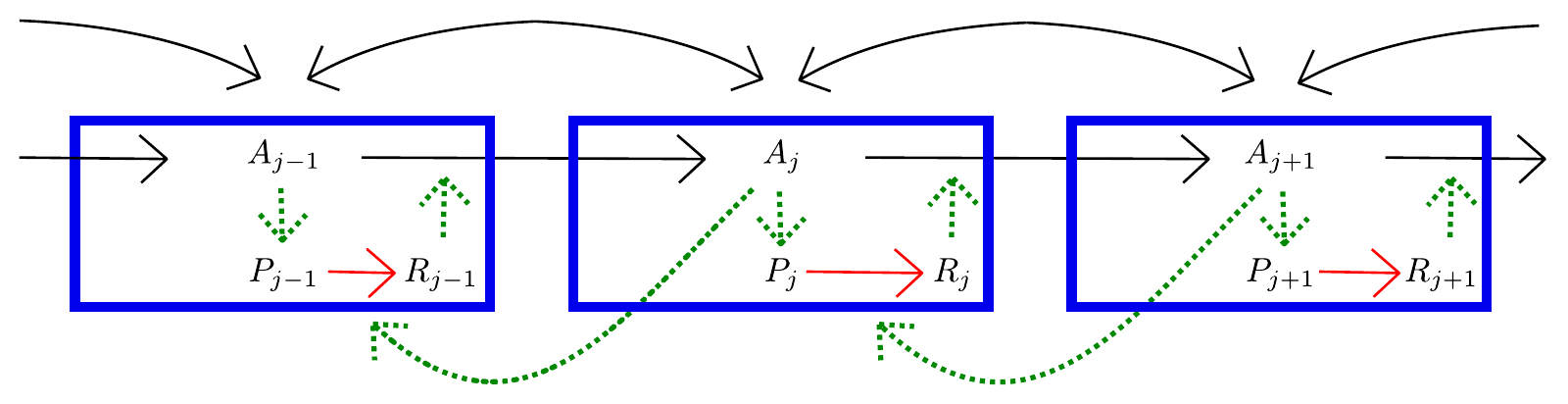}
\caption[]{Schematic representation of the model \eqref{eqn: main system}.
Black arrows represent transport, red arrows describe polarization
and the green dashed arrows indication promotion. In particular, the PIN1
polarization rate correlates positively with the neighbouring auxin concentration,
making this a model of `up-the-gradient' type.
}
\label{fig:int:model}
\end{figure} 

Inspired by \cite{Heisler2006,merks2007canalization}, the system we will study is given by 
\begin{equation}\label{eqn: main system}
\begin{cases}
\dot{A}_j = T_{\act}\left(R_{j-1}\frac{A_{j-1}}{k_a+A_{j-1}} - R_j\frac{A_j}{k_a+A_j}\right) + T_{\diff}(A_{j+1}-2A_j+A_{j-1}), \\
\\
\dot{P}_j = -k_1 \frac{A_{j+1}}{k_r + A_{j+1}}\left(\frac{P_j}{k_m+P_j}\right) + \alpha{A}_j , \\
\\
\dot{R}_j = k_1\frac{A_{j+1}}{k_r+A_{j+1}}\left(\frac{P_j}{k_m+P_j}\right),
\end{cases}
\end{equation}
posed on the one-dimensional lattice $j \in \mathbb{Z}$;
see Fig. \ref{fig:int:model}. The variable $A_j(t)$ denotes the auxin concentration in cell $j \in \mathbb{Z}$, while $P_j(t)$ and $R_j(t)$ represent the unpolarized respectively right-polarized PIN1 in this cell.
The PIN1 hormone is the PIN-variant that is believed to play the most important role in auxin transport \cite{heisler2010alignment}.

The parameters appearing in the problem are all strictly positive and labelled in the same manner as in \cite{merks2007canalization}\footnote{For presentation purposes, the parameters $L$ and $r$ appearing in \cite{merks2007canalization} have been set to unity.}. In particular, $T_{\mathrm{act}}$ and $T_{\mathrm{diff}}$ denote the strengths of the active PIN1-induced rightward auxin transport and its diffusive counterpart, respectively. Unpolarized PIN1 is formed in the presence of auxin at a rate $\alpha$, while $k_1$ denotes the polarization rate. Finally,  $k_a$, $k_r$, and $k_m$ are the Michaelis constants associated to the active transport of auxin and the polarization of PIN1, which depends on the auxin-concentration in the right-hand neighbouring cell. In particular, this model is of `up-the-gradient' type.

The main difference compared to \cite{merks2007canalization}
is that we are neglecting the presence of left-polarized PIN1
and have set the decay and depolarization rates of PIN1 to zero. Although this step of course imposes a pre-existing polarity on the system, we need to do this for technical reasons that we explain in the sequel. For now we simply point out that we wish to focus our attention on the dynamics of rightward auxin propagation, which takes place on timescales that are much faster than these decay and depolarization processes, and that the results will give novel insight into the full problem.

 \begin{figure}[t]
\centering\includegraphics[width=\textwidth]{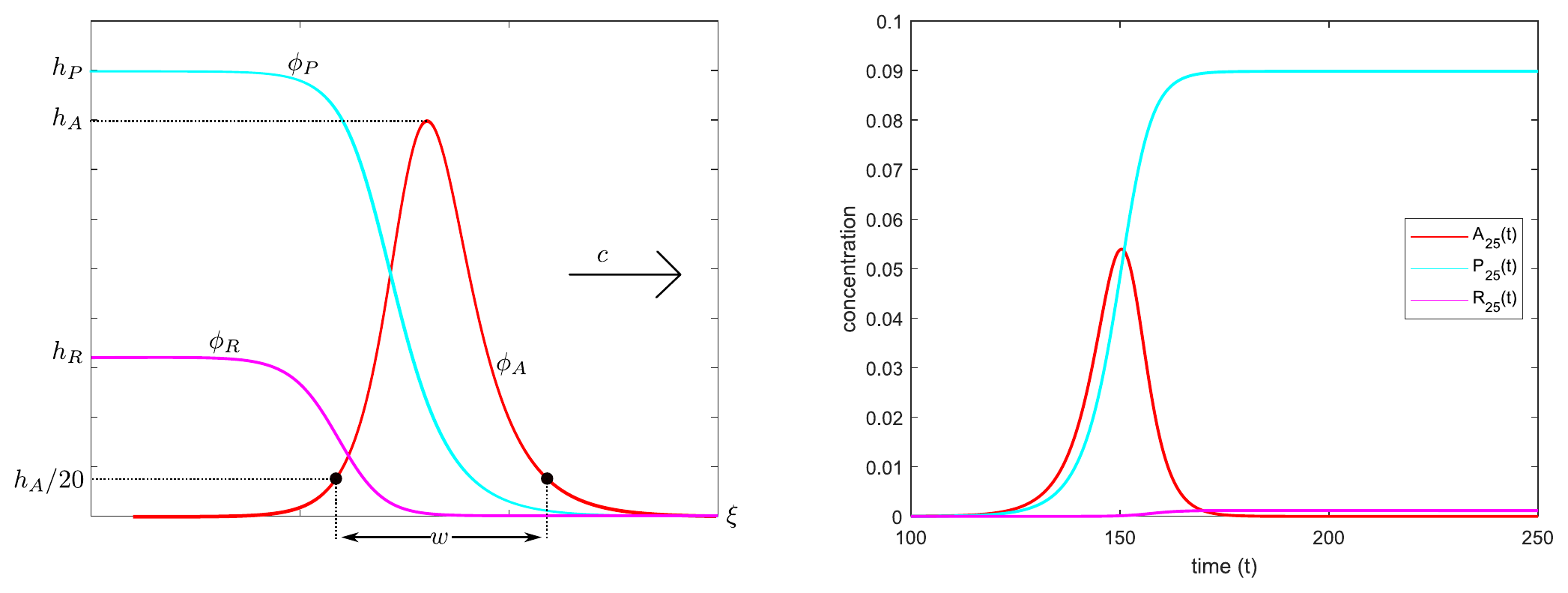}
\caption[]{Left: cartoon of the waveprofiles $(\phi_A, \phi_P, \phi_R)$, illustrating the definition of the width $w$ of the auxin-pulse
and the limits \eqref{eqn: profile limits}.
Right: numerical simulation of an auxin pulse passing through cell 25, leaving
a residue of (polarized) PIN1. We used the procedure described in {\S}\ref{sec:sub:int:mr}, with $A_1(0) = A_{\diamond} = 0.15$. The remaining parameters were fixed as $T_{\mathrm{act}} = 800$, $T_{\mathrm{diff}} = 0.15$,
$k_a = 1$, $k_m = k_r = 100$, $k_1 = 200$ and $\alpha = 0.1$.
}
\label{fig:int:profile}
\end{figure}

We will look for solutions of the special type
\begin{equation}
\label{eq:int:ansatz:a:p:r}
    (A_j, P_j, R_j)(t) = (\phi_A, \phi_P, \phi_R)(j - c t),
\end{equation}
with $c > 0$, in which we impose the limits
\begin{equation}\label{eqn: profile limits}
\lim_{\xi \to - \infty} \phi_A(\xi) = 0,
\qquad \qquad 
    \lim_{\xi \to \infty} (\phi_A, \phi_P, \phi_R)(\xi) = 0.
\end{equation}
From a modelling perspective, such solutions represent a pulse of auxin that moves to the right through a one-dimensional row of cells. Ahead of the wave the cells are clear of both polarized and unpolarized PIN, but behind the wavefront a residual amount of PIN is left in the cells, representing the coordinated polarisation of the tissue.

In reality these residues start to depolarize and decay, which can be included by 
adding linear decay terms to \eqref{eqn: main system}.
This leads to the expanded system
\begin{equation}\label{eqn:int:main:sys:expanded}
\begin{cases}
\dot{A}_j = T_{\act}\left(R_{j-1}\frac{A_{j-1}}{k_a+A_{j-1}} - R_j\frac{A_j}{k_a+A_j}\right) + T_{\diff}(A_{j+1}-2A_j+A_{j-1}), \\
\\
\dot{P}_j = -k_1 \frac{A_{j+1}}{k_r + A_{j+1}}\left(\frac{P_j}{k_m+P_j}\right) +\alpha{A}_j 
+ k_2 R_j
- \delta P_j , \\
\\
\dot{R}_j = k_1\frac{A_{j+1}}{k_r+A_{j+1}}\left(\frac{P_j}{k_m+P_j}\right) - k_2 R_j,
\end{cases}
\end{equation}
in which the positive parameters $\delta$ and $k_2$ represent the decay and depolarization rate of PIN1, respectively.
 Mathematically, these terms can be included into our framework
provided that the parameters $\delta$ and $k_2$ are small compared to the amplitude of the pulses, but we do not pursue this level of generality in the current paper for presentational clarity. Note in any case that in 
\cite{merks2007canalization} these parameters were chosen to be orders of magnitude smaller than $\alpha$ and $k_1$.

\begin{figure}[t]
\centering\includegraphics[width=\textwidth]{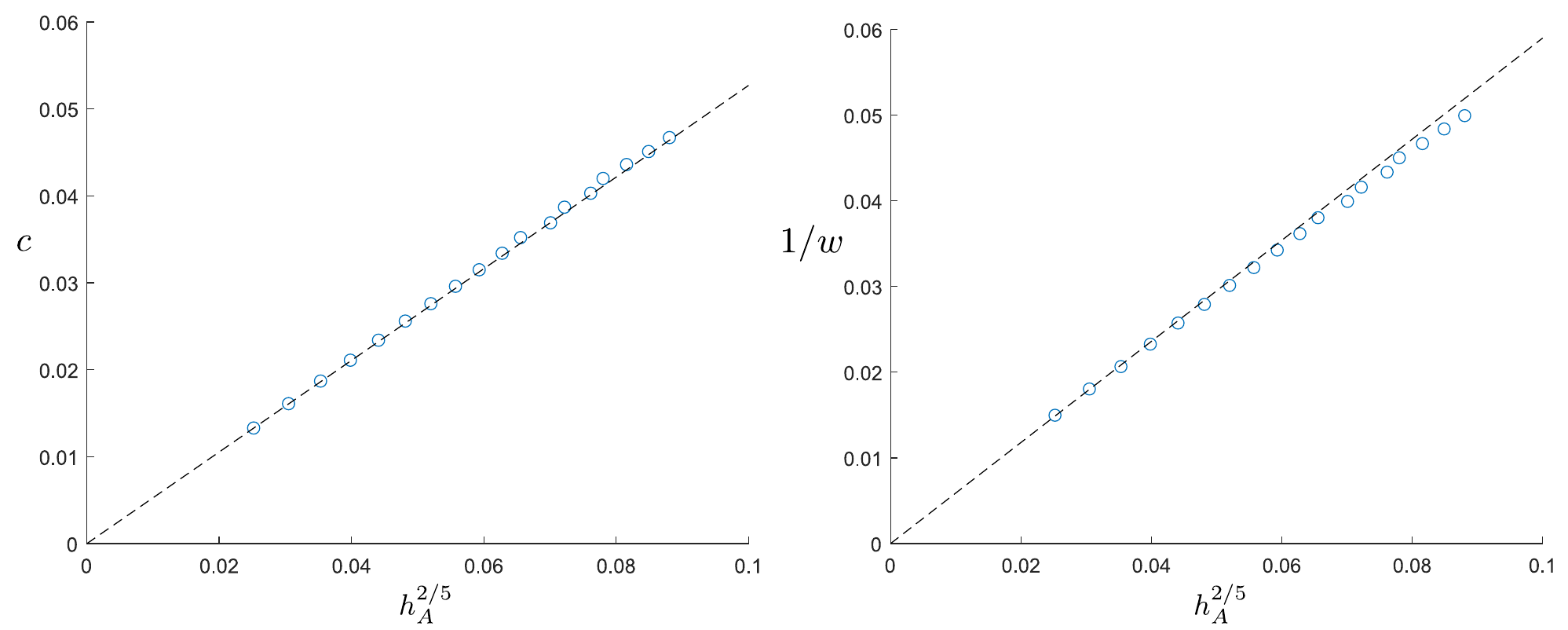}
\caption[]{Scaling behaviour of the wavespeed $c$ (left) and the auxin width $w$ (right) against the height $h_A$ of the auxin pulse. The dashed lines represent the explicit predictions \eqref{eq:int:scaling:relations}.
The circles arise from numerical simulations, following the procedure
described in {\S}\ref{sec:sub:int:mr} with several
different values for $A_\diamond$. The other parameters were chosen as in Fig. \ref{fig:int:profile}.
}
\label{fig:int:scalings}
\end{figure}

Travelling waves have played a fundamental role in the analysis of 
many spatially discrete systems \cite{kev,MPB,CGW08,HJHFZHNGM,VL28}.
They can be seen as a
lossless mechanism to transport matter or energy over arbitrary 
distances. As such, they are interesting in their own right,
but they can also be viewed as building blocks to describe
more complicated behaviour of nonlinear systems \cite{AW75,AW78}.
In the present case for example, one can construct
wavetrain solutions to \eqref{eqn:int:main:sys:expanded}
by adding a persistent auxin source;
see Fig.~\ref{fig:int:wave:trains} and Supplementary Video S1. 
Initially, these solutions
can be seen in an approximate sense as a concatenation of the
individual auxin pulses that we consider here \cite{moserBSc}. As a 
consequence of the amplitude variations, small speed differences occur between these pulses which leads to highly interesting collision processes.  
Due to this type of versatility, travelling waves play an 
important role in many applications and have been extensively studied
in a variety of settings
\cite{sandstede2002stability,kev,
hochstrasser1989energy,jones1991construction}.

\subsection{Main results}
\label{sec:sub:int:mr}

Our goal will be to obtain quantitative scaling information concerning
the speed and shape of these waves. In particular, we will show rigorously 
that \eqref{eqn: main system} admits a family  of travelling wave solutions
that are parameterized by the amplitude of the auxin-pulse. In addition, we show that the speed and width of these waves scale with this amplitude  
via a fractional power law. 
We state our results in full technical detail in Corollary \ref{cor: main corollary} below.

\begin{figure}[t]
\centering\includegraphics[width=\textwidth]{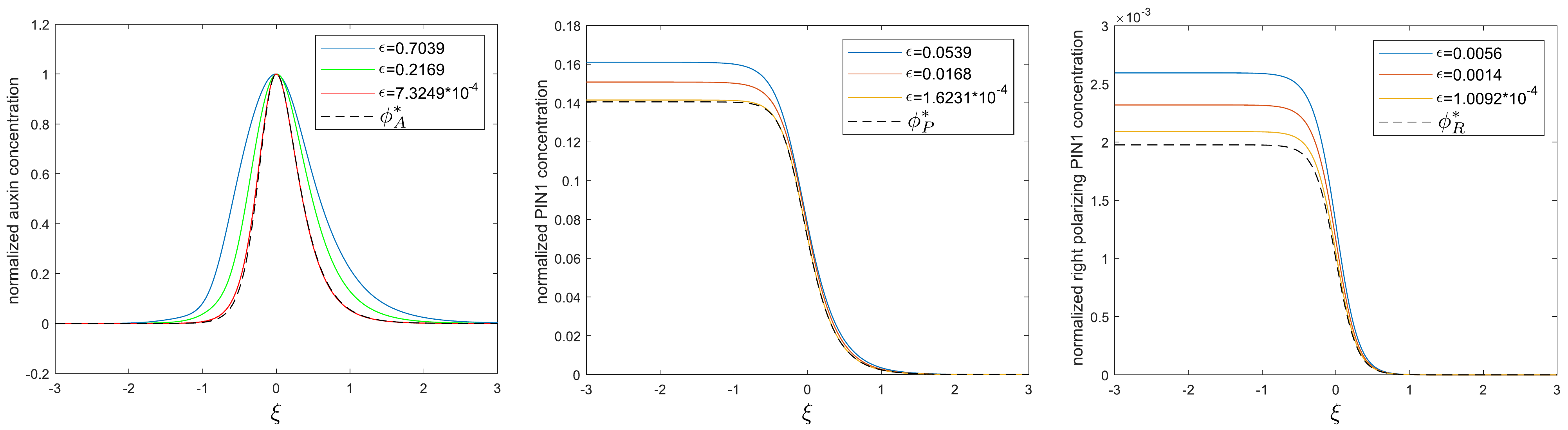}
\caption[]{Convergence of the (scaled) profiles $\phi_A$ (left), $\phi_P$ (center) and $\phi_R$ (right) to their limits $(\phi_A^*, \phi_P^*, \phi_R^*)$. To perform the scalings, we wrote $h_A = \norm{\phi_A}_{L^\infty}$, compressed space by a factor of $h_A^{2/5}$ and divided the three profiles by the respective factors $(h_A, h_A^{1/5}, h_A^{2/5})$,
in line with the relations \eqref{eq:int:scaling:relations}.
}
\label{fig:int:limits}
\end{figure}

More precisely, we provide an explicit triplet of functions
$(\phi_A^*, \phi_P^*, \phi_R^*)$ that satisfy the limits
\eqref{eqn: profile limits}
and construct solutions to
\eqref{eqn: main system} of the form
\begin{equation}
\label{eq:int:mr:a:p:r:profiles:scales}
\begin{array}{lcl}
\big( A_j, P_j, R_j \big)(t)
&=& \Big( \epsilon \phi_A^*, \epsilon^{1/5} \phi_P^*, \epsilon^{2/5} \phi_R^* \Big)
\Big( \epsilon^{2/5}( j - c_* \epsilon^{2/5} t ) \Big)
\\[0.2cm]
& & \qquad \qquad \qquad
+ \Big( \mathcal{O}( \epsilon^{17/15}),
  \mathcal{O}( \epsilon^{1/3} ),
  \mathcal{O}( \epsilon^{3/5} )
\Big),
\end{array}
 \end{equation}
for a constant $c_*$, which we state exactly in \eqref{eq:int:def:expl:constants}. Here the limiting profile $\phi_A^*$
 is scaled in such a way that $\norm{ \phi_A^* }_{L^{\infty}} = 1$.
Upon introducing the heights\footnote{
Here we use the abbreviation
$\norm{A}_{\infty} = \textstyle{\sup_{j,t}} |A_j(t)|$ and its analogues for $P$ and $R$.
}
 \begin{equation}
     (h_A, h_P, h_R) = \big( \norm{A}_\infty, \norm{P}_\infty, \norm{R}_\infty \big)
 \end{equation}
 associated to the three components of our waves, this choice ensures that
 the auxin-height $h_A$ is equal to the parameter $\epsilon>0$ at leading order.
  In particular, comparing
 this to \eqref{eq:int:ansatz:a:p:r} we uncover the
 leading order scaling relations
 \begin{equation}
 \label{eq:int:scaling:relations}
     c \sim c_*  h_A^{2/5},
     \qquad \qquad  w \sim w_* h_A^{-2/5},
     \qquad \qquad  h_P \sim h_P^* h_A^{1/5},
     \qquad \qquad h_R \sim h_R^* h_A^{2/5}
 \end{equation}
 for the speed $c$, width\footnote{We define the width of the auxin pulse as the distance between the two points where the pulse attains 5\% 
 of its maximum value.} $w$ and heights of the wave.
 Here
 the  constant $w_*$  denotes the width of the limiting profile $\phi_A^*$, while the other constants are given explicitly by
\begin{equation}
\label{eq:int:def:expl:constants}
\begin{array}{lcl}
c_* & = & \left(\frac{9\alpha{k}_1T_{\act}T_{\diff}^2}{8k_ak_mk_r}\right)^{1/5},
\\[0.4cm]
h_P^* & = & \sqrt{6} \left(\frac{9\alpha^6 {k}_a^4 k_m^4 k_r^4 T_{\diff}^2}{8 k_1^4T_{\act}^4}\right)^{1/10}, 
\\[0.4cm]
h_R^* &= & 3\left(\frac{9\alpha k_a^4 {k}_1T_{\diff}^2}{8k_r k_m T_{\act}^4}\right)^{1/5}.
\end{array}
\end{equation}
In particular, for a fixed height of the auxin-pulse our results state that the speed and residual PIN1 will increase as the PIN1-production parameter $\alpha > 0$ is increased. 

\begin{figure}[t]
\centering\includegraphics[width=\textwidth]{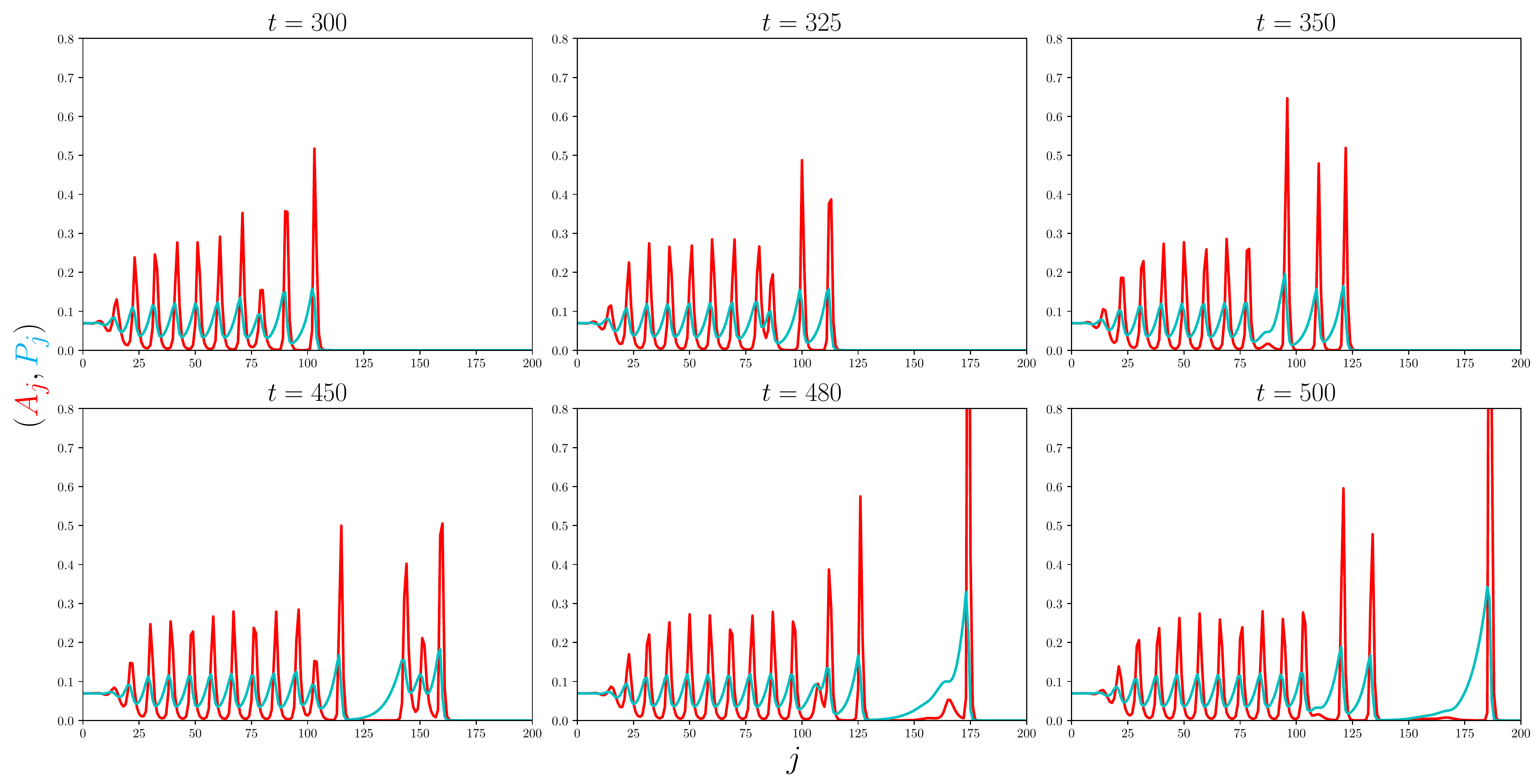}
\caption[]{Six snapshots of a wavetrain simulation for the expanded system
\eqref{eqn:int:main:sys:expanded}. Higher pulses travel faster than lower pulses, in correspondence with the scaling relations \eqref{eq:int:scaling:relations}. These speed differences lead to merge events where even higher pulses are formed, which detach from the bulk.
We used the procedure described in {\S}\ref{sec:sub:int:mr}, taking $A_1(0) = A_{\diamond} = 0.0$ but adding $0.025$ to $\dot{A}_1(t)$ to simulate a constant auxin influx at the left boundary. 
We picked $\delta = 0.1$ and $k_2 = 0.2$,
leaving the remaining parameters
from Fig. \ref{fig:int:profile} unchanged. The full simulation can be found in supplementary video S1.
}
\label{fig:int:wave:trains}
\end{figure} 
 
 Although our proof requires the parameter $\epsilon > 0$ and hence the amplitude of the auxin-pulses to be small, 
 this branch of solutions continues to exist well beyond this asymptotic regime.
 Indeed, we numerically confirmed the existence (and stability) of these waves by a direct simulation of \eqref{eqn: main system}
 on a row of cells $j \in \{1 , \ldots 500 \}$,
 initialized with $A_j(0) = P_j(0) = R_j(0) =0$ for  $2 \le j\le 500$,
 together with $P_1(0) = R_1(0)=0$ and $A_1(0) = A_{\diamond}$ for some $A_{\diamond} > 0$ that we varied between simulations.
 In order to close the system, we used the Neumann-type condition $A_0(t) = A_1(t)$ on the left-boundary, together with $R_0(t) = 0$ and a sink condition $A_{501}(t) = 0$
 on the right. An example of such a simulation can be found
 in Fig.~\ref{fig:int:profile} (right). By varying the initial auxin concentration $A_{\diamond}$, we were able to generate waves with a range of amplitudes. We subsequently numerically computed the speed and width of these waves, which allowed us to confirm the leading order behaviour \eqref{eq:int:scaling:relations}; see Fig.~\ref{fig:int:scalings}. In addition, we verified the
 convergence to the limiting profiles $(\phi_A^*,\phi_P^*, \phi_R^*)$ by comparing the appropriately rescaled numerical waveprofiles; see 
 Fig.~\ref{fig:int:limits}.

\subsection{Cross-diffusion}

From a mathematical perspective, the problem
\eqref{eqn: main system} is interesting due to its interpretation
as a so-called cross-diffusion problem, where the transport coefficient
of one component is influenced by one of the other components. Work in this
area was stimulated by developments in the modeling of bacterial cell membranes \cite{SHIH20191469} and biofilms \cite{EmereniniBlessingO2015AMMo},
where self-organization of biological molecules plays an important role.
In the continuum regime, such problems are tough to analyze on account
of potential degeneracies in the coefficients. The well-posedness of the underlying problem was analyzed in 
\cite{SonnerStefanie2011Otwo}, while a numerical method for such problems
was developed in  \cite{ghasemi_sonner_eberl_2018}.

The key phenomenological assumption behind such models is that particles
behave differently when they are isolated compared to when they are 
part of a cluster. A simplified agent-based approach to capture this 
mechanism can be found in
\cite{johnston2017co}, which reduces naturally to a scalar
PDE with nonlinear diffusion in the continuum limit. After adding
a small regularization term, it is possible to 
use geometric singular perturbation theory
to show that this PDE
admits travelling wave solutions \cite{LI2021132916}.
In this setting, the steepness of the wavefronts provides
the necessary scale-separation required for rigorous results.

Our approach in this paper proceeds along entirely different lines,
using the amplitude of the auxin pulse as a small continuation parameter
to construct a family of travelling wave solutions
to \eqref{eqn: main system}. The key insight is that one can extract an effective limiting system by scaling the width and speed of the wave 
in an appropriate fashion and sending the amplitude to zero.
By means of a fixed-point analysis one
can show in a rigorous fashion that solutions to this limiting system 
can be continued to form a family of solutions to the full system.

\subsection{Relation to FPUT pulses}

Our technique is a generalization of the approach developed
by Friesecke and Pego \cite{friesecke-pego1} to construct
small-amplitude travelling pulse solutions to the
Fermi-Pasta-Ulam-Tsingou (FPUT) problem
\cite{fput-original,dauxois}
\begin{equation}\label{newton}
 \ddot{x}_j
= F(x_{j+1} - x_j) -F(x_{j} - x_{j-1}),
\qquad \qquad j \in \mathbb{Z}.
\end{equation}
This models an infinite, one-dimensional chain of particles that
can only move horizontally and are connected to their nearest neighbours by springs. These springs transmit a force 
\begin{equation}
    F(r) = r + r^2
\end{equation}
that hence depends nonlinearly on the relative distance $r$
between neighbouring particles; see \cite{friesecke-pego1, herrmann-matthies-asymptotic, pankov} for the impact of other choices.
The FPUT system is well-established as a fundamental model
to study the propagation of disturbances
through spatially discrete systems,
such as granular media, artificial metamaterials, DNA strands, and electrical transmission lines \cite{brill, kev}. 

Looking for a travelling wave in the relative displacement coordinates, one introduces an Ansatz of the form
\begin{equation}
    x_{j+1}(t) - x_j(t) = \phi(j -  \sigma t ),
\end{equation}
which leads to the scalar functional differential equation of mixed type (MFDE)
\begin{equation}\label{mono:fput:tw:eqn}
\sigma^2 \phi''(\xi) 
= F\big(\phi(\xi+1)\big) - 2 F\big(\phi(\xi)\big) + F\big(\phi(\xi - 1) \big) .
\end{equation}
Following the classic papers  by Friesecke in combination with Wattis \cite{friesecke-wattis}
and Pego \cite{friesecke-pego1, friesecke-pego2, friesecke-pego3, friesecke-pego4},
we introduce the scaling
\begin{equation}
    \phi(\xi) = \ep^2 \varphi_{\ep}(\ep \xi )
\end{equation}
and write $\sigma = \sigma_{\epsilon}$,
which transforms \eqref{mono:fput:tw:eqn} into the MFDE
\begin{equation}
\label{eq:int:fput:trv:wave:scaled}
    \sigma_{\epsilon}^2 \epsilon^2   \varphi_{\epsilon}'' 
= \big( S^{\epsilon} + S^{-\epsilon} - 2 \big) \big[ \varphi_{\epsilon} + \epsilon^2 \varphi_{\epsilon}^2 \big] .
\end{equation}
Here the shift operator $S^d$ acts as
\begin{equation}\label{eqn: shift intro}
    (S^d f)(\xi) = f(\xi + d)
\end{equation}
for any $d \in \mathbb{R}$. Since the symbol $S^{\epsilon} + S^{-\epsilon}- 2$ represents a discrete Laplacian, we can interpret \eqref{eq:int:fput:trv:wave:scaled} as a wave equation with a nonlinear diffusion term. To some extent, this clarifies the link with our original problem \eqref{eqn: main system} and the discussion above.

Applying the Fourier transform to \eqref{eq:int:fput:trv:wave:scaled}
with $k$ as the frequency variable, we arrive at
\begin{equation}
-\sigma_{\epsilon}^2 \epsilon^2 k^2 \hat{\varphi}_{\epsilon}(k) =  2(\cos (\epsilon k) -1) \big[ \hat{\varphi_{\epsilon}} + \epsilon^2\hat{\varphi_{\epsilon}^2} \big](k)
= - 4 \sin^2(\epsilon k/2) \big[ \hat{\varphi_{\epsilon}} + \epsilon^2 \hat{\varphi_{\epsilon}^2} \big](k).
\end{equation}
Upon introducing the symbol
\begin{equation}
    \widetilde{\mathcal{M}}_{\mathrm{FPUT}}^{(\epsilon)}(k) = 
    \frac{4 \epsilon^2 \sin^2(\epsilon k/2)}{ \sigma_{\epsilon}^2 \epsilon^2 k^2 - 4 \sin^2(\epsilon k/2)},
\end{equation}
this can be recast into the compact form
\begin{equation}
    \hat{\varphi_{\epsilon}}(k) = \widetilde{\mathcal{M}}_{\mathrm{FPUT}}^{(\epsilon)}(k) \hat{\varphi_{\epsilon}^2}(k).
\end{equation}

Upon choosing the speed
\begin{equation}\label{eqn: FP wave speed}
\sigma_{\ep} = 1 + \frac{\ep^2}{3},
\end{equation}
we can exploit the expansion 
$\sin^2(z/2) = \frac{1}{4} z^2 - \frac{1}{48} z^4 + O(z^6)$
to obtain the pointwise limit
\begin{equation}
    \widetilde{\mathcal{M}}_{\mathrm{FPUT}}^{(\epsilon)}(k) \to \frac{12}{8 + k^2}, \qquad \qquad \epsilon \to 0 .
\end{equation}
Using the fact that $(8 + k^2)$ is the Fourier symbol for $8 - \partial^2_{\xi}$, this suggests that the
relevant
system for $\varphi_{\epsilon}$ in the formal $\epsilon \to 0$ limit is given by
\begin{equation}
8 \varphi_*  -  \varphi_*'' = 12 \varphi_*^2,
\end{equation}
which has the nontrivial even solution
\begin{equation}\label{eqn: FP sech}
    \varphi_*(\xi) = 
    \sech^2(\sqrt{2} \xi).
\end{equation}

By casting the problem in an appropriate
functional analytic framework, one can show
that this explicit solution $\varphi_*$ can be continued
to yield solutions $\varphi_{\epsilon}$
to \eqref{eq:int:fput:trv:wave:scaled} for small $\epsilon > 0$.
In this fashion, one establishes the existence of a family of pulse solutions \cite{friesecke-pego1}
\begin{equation}
    x_{j+1}(t) - x_j(t)
    = \epsilon^2 \sech^2\Big(\sqrt{2} \epsilon (j - \sigma_{\epsilon} t) \Big)
    + \mathcal{O} (\epsilon^4).
\end{equation}

Roughly speaking, the main mathematical contribution in this paper is that
we show how this analysis can be generalized to the setting of
\eqref{eqn: main system}. The first main obstacle is that this
is a multi-component system, which requires us to 
explicitly reduce the order before a tractable limit can be obtained.
The second main obstacle is that the analysis of our Fourier symbol is 
considerably more delicate, since in our setting the wavespeed $c$ 
converges to zero instead of one as $\ep \to 0$. 
Indeed, the denominator
of $\widetilde{\mathcal{M}}_{\mathrm{FPUT}}^{(\epsilon)}$ above depends
only on the product $\epsilon k$, while in our case
there is a separate dependence on $\epsilon^2 k$. 
This introduces a quasi-periodicity into the problem that requires
our convergence analysis to carefully distinguish between `small' values of $k$ and several separate regions of `large' $k$.

The third main difference is that we cannot use formal spectral arguments to analyze the limiting linear operator, which in our case is related to the Bernoulli equation.  Instead, we apply a direct solution technique using variation-of-constants formulas. On the one hand this is much more explicit, but on the other hand the resulting estimates are rather delicate
on account of the custom function spaces involved.

\subsection{Discussion}

Due to the important organizing role that wave solutions often play in complex systems,  scaling information such as \eqref{eq:int:scaling:relations} can be used as the starting point to uncover more general dynamical information concerning models
such as \eqref{eqn: main system} and related models of polar auxin trasnport. As such, we hope that the ideas
we present here will provide a robust analytical tool to analyze
different types of models as well. The resulting insights and predictions could
help to prioritize competing models on the basis of 
dynamical experimental observations. Indeed, scaling laws 
appear to play a role in
many aspects of biological systems,
such as the structural properties of vascular systems \cite{razavi2018scaling},
the mass dependence of metabolic rates \cite{west2004life}
and the functional constraints imposed by size
\cite{schmidt1984scaling}.

Although we have included only right-polarizing PIN in our system, 
we believe that our techniques can be adapted to cover the full
case where also left-polarizing PIN is included. However, the
computations rapidly become unwieldy and the limiting system
is expected to differ qualitatively. For this reason, we have 
not chosen to pursue this level of generality in the present paper,
as it would only obscure the main ideas behind our framework.
One of the main generalizations that we intend to pursue
in the future is to study the model in two spatial dimensions.
This is motivated by recent 
numerical observations concerning the formation of auxin channels
and their associated PIN walls
under the influence of travelling patterns that are localized in
both spatial dimensions \cite{althuisBSc}.

\subsection{Notation}
We summarize a few aspects of our (mostly standard) notation.

%%-----------------------------------------------------------%%
%%-----------------------------------------------------------%%
\begin{enumerate}[label=$\bullet$]

%%-----------------------------------------------------------%%
\item
If $f = f(X)$ is a differentiable function on $\R$, then we sometimes write $f' = \partial_X[f]$.

%%-----------------------------------------------------------%%
\item

If $\X$ and $\mathcal{Y}$ are normed spaces, then we denote the space of bounded linear operators from $\X$ to $\Y$ by $\b(\X,\Y)$.
We put $\b(\X) := \b(\X,\X)$.

\end{enumerate}

\subsection{Acknowledgments}
HJH and TEF acknowledge support from the Netherlands Organization for Scientific Research (NWO) (grant 639.032.612).

%%-----------------------------------------------------------%%
%%-----------------------------------------------------------%%
%%-----------------------------------------------------------%%
%%-----------------------------------------------------------%%
%%-----------------------------------------------------------%%
\section{The Travelling Wave Problem}
\label{sec:travelingwave}

%%-----------------------------------------------------------%%
%%-----------------------------------------------------------%%
%%-----------------------------------------------------------%%
%%-----------------------------------------------------------%%
\subsection{Rewriting the original problem \eqref{eqn: main system}}
We will reduce the problem \eqref{eqn: main system} to a system of equations involving only $A_j$ and $P_j$, and it will be this resulting system on which we make the long wave-scaled travelling wave Ansatz.

%%-----------------------------------------------------------%%
%%-----------------------------------------------------------%%
%%-----------------------------------------------------------%%
\subsubsection{Changes of notation}
We begin by rewriting \eqref{eqn: main system} in a slightly more compressed manner that also exposes more transparently the leading order terms in the nonlinearities.
Let $\delta^{\pm}$ be the left and right difference operators that act on sequences $(x_j)$ in $\R$ via
\[
\delta^+x_j := x_{j+1}-x_j
\quadword{and} 
\delta^-x_j := x_j-x_{j-1}.
\]
Next, for $k$, $x \in \R$ with $k+x \ne 0$ we have
\[
\frac{x}{k+x}
= \frac{x}{k} - \frac{x^2}{k(k+x)}.
\]
We put
\begin{equation}\label{eqn: Qsf1}
\Qsf_1(x,y)
:= \frac{x^2y}{k_a+x}
\end{equation}
and compress
\begin{equation}
\label{eq:twv:def:tau:1:2}
\tau_1 := \frac{T_{\act}}{k_a}
\quadword{and}
\tau_2 := T_{\diff}
\end{equation}
to see that our equation for $A_j$ now reads
\[
\dot{A}_j
= \tau_2\delta^+\delta^-A_j
- \tau_1\delta^-(R_jA_j)
+ \tau_1\delta^-\Qsf_1(A_j,R_j).
\]

Next, we abbreviate
\begin{equation}
\label{eq:twv:def:kappa}
\kappa
:= \frac{k_1}{k_rk_m}
\end{equation}
and put 
\begin{equation}\label{eqn: Qsf2}
\Qsf_2(x,y)
:= \kappa\left(\frac{k_ry + k_mx + xy}{(k_r+x)(k_m+y)}\right)xy
\end{equation}
to see that, the equation for $P_j$ is
\[
\dot{P}_j
= -\kappa{A}_{j+1}P_j
+ \alpha{A}_j
+ \Qsf_2(A_{j+1},P_j).
\]
The equation for $R_j$ is updated similarly, and so we have rewritten \eqref{eqn: main system} as
\begin{equation}\label{eqn: main system2}
\begin{cases}
\dot{A}_j = \tau_2\delta^+\delta^-A_j - \tau_1\delta^-(R_jA_j) + \tau_1\delta^-\Qsf_1(A_j,R_j), \\
\\
\dot{P}_j = -\kappa{A}_{j+1}P_j + \alpha{A}_j + \Qsf_2(A_{j+1},P_j), \\
\\
\dot{R}_j = \kappa{A}_{j+1}P_j - \Qsf_2(A_{j+1},P_j).
\end{cases}
\end{equation}

We observe that the equation for $R_j$ depends only on $A_{j+1}$ and $P_j$ and therefore can be solved by direct integration.
Before we do that, however, we rewrite the new equation for $P_j$ using Duhamel's formula.

%%-----------------------------------------------------------%%
%%-----------------------------------------------------------%%
%%-----------------------------------------------------------%%
\subsubsection{Rewriting the $P_j$ equation}\label{sec: Pj analysis}
We can view the equation for $P_j$ in \eqref{eqn: main system2} as a first-order linear differential equation forced by $\alpha{A}_j + \Qsf_2(A_{j+1},P_j)$, and so we can solve it via the integrating factor method.
For $f$, $g \in L^1$ and $h \in L^{\infty}$ we introduce the operators
\begin{equation}\label{eqn: Esf}
\Esf(f)(s,t)
:= \exp\left(-\kappa\int_s^t f(\xi) \dxi\right), \ s, t \in \R,
\end{equation}
\begin{equation}\label{eqn: Psf1}
\Psf_1(f,g)(t)
:= \alpha\int_{-\infty}^t \Esf(f)(s,t)g(s) \ds,
\end{equation}
and
\begin{equation}\label{eqn: Psf2}
\Psf_2(f,h)(t)
:= \int_{-\infty}^t \Esf(f)(s,t)\Qsf_2(f(s),h(s)) \ds.
\end{equation}
Recall from \eqref{eqn: profile limits} that we want $P_j$ to vanish at $-\infty$.
The unique solution for $P_j$ in \eqref{eqn: main system2} that does vanish at $-\infty$ must satisfy
\[
P_j(t)
= \Psf_1(A_{j+1},A_j)(t) + \Psf_2(A_{j+1},P_j)(t).
\]

%%-----------------------------------------------------------%%
%%-----------------------------------------------------------%%
%%-----------------------------------------------------------%%
\subsubsection{Solving the $R_j$ equation}\label{sec: Rj analysis}
Since, per \eqref{eqn: profile limits}, we want $R_j$ to vanish at $-\infty$, and since we are assuming that each $A_j$ vanishes sufficiently fast at both $\pm\infty$ and $P_j$ vanishes at $-\infty$ and remains bounded at $+\infty$, we may solve for $R_j$ by integrating the third equation in \eqref{eqn: main system2} from $-\infty$ to $t$.
For $f$, $g \in L^1$ and $h \in L^{\infty}$, we define more integral operators:
\begin{equation}\label{eqn: Rsf1}
\Rsf_1(f,g)(t) 
:= \kappa\tau_1\int_{-\infty}^t f(s)\Psf_1(f,g)(s) \ds, \ t \in \R,
\end{equation}
\begin{equation}\label{eqn: Rsf2}
\Rsf_2(f,g,h)(t)
:= \int_{-\infty}^t \big(\kappa{f}(s)\Psf_2(f,g)(s)-\Qsf_2(f(s),\Psf_1(f,g)(s)+\Psf_2(f,h)(s))\big) \ds,
\end{equation}
and
\begin{equation}\label{eqn: Rsf}
\Rsf(f,g,h)(t)
:= \Rsf_1(f,g)(t) + \Rsf_2(f,g,h)(t).
\end{equation}
We have defined $\Psf_1$ and $\Psf_2$ just above, respectively, in \eqref{eqn: Psf1} and \eqref{eqn: Psf2} and $\Qsf_2$ earlier in \eqref{eqn: Qsf2}.
Then the solution to the third equation in \eqref{eqn: main system2} that vanishes at $-\infty$ is
\begin{equation}\label{eqn: Rj}
R_j(t)
= \Rsf(A_{j+1},A_j,P_j)(t)
= \Rsf_1(A_{j+1},A_j)(t) + \Rsf_2(A_{j+1},A_j,P_j)(t).
\end{equation}

%%-----------------------------------------------------------%%
%%-----------------------------------------------------------%%
%%-----------------------------------------------------------%%
\subsubsection{The final system for $A_j$ and $P_j$}
We rewrite (part of) the $A_j$ equation once more to incorporate the new expression for $R_j$.
For $f$, $g \in L^1$ and $h \in L^{\infty}$ and $t \in \R$ put
\begin{equation}\label{eqn: Nsf}
\Nsf(f,g,h)(t)
:= \tau_1\Qsf_1(g(t),\Rsf(f,g,h)(t)) - \tau_1\Rsf_2(f,g,h)(t)g(t),
\end{equation}
where we defined $\Qsf_1$ in \eqref{eqn: Qsf1}.
Then $A_j$ must satisfy
\[
\dot{A}_j
= \tau_2\delta^+\delta^-A_j - \delta^-\big(\Rsf_1(A_{j+1},A_j)A_j\big) + \delta^-\Nsf(A_{j+1},A_j,P_j),
\]
and so our system for $A_j$ and $P_j$ is now
\begin{equation}\label{eqn: main system3}
\begin{cases}
\dot{A}_j = \tau_2\delta^+\delta^-A_j - \delta^-\big(\Rsf_1(A_{j+1},A_j)A_j\big) + \delta^-\Nsf(A_{j+1},A_j,P_j), \\
\\
P_j = \Psf_1(A_{j+1},A_j) + \Psf_2(A_{j+1},P_j).
\end{cases}
\end{equation}
That is, using the formula \eqref{eqn: Rj} for $R_j$ in terms of $A_j$ and $P_j$, we can solve \eqref{eqn: main system2} if we can solve \eqref{eqn: main system3}.

We will make two changes of variables on \eqref{eqn: main system3}.
First, in Section \ref{sec: tw}, we make a travelling wave Ansatz for $A_j$ and $P_j$.
We reformulate \eqref{eqn: main system3} for the travelling wave profiles as the system \eqref{eqn: tw syst fm} below.
Then, in Section \ref{sec: lw}, we introduce our long wave scaling on these travelling wave profiles.
After numerous adjustments, we arrive at the final system \eqref{eqn: lw syst-ep} for the scaled travelling wave profiles, which we solve in Section \ref{sec: lw syst-nu soln}.
The reader uninterested in these intermediate stages may wish to proceed directly to Theorem \ref{thm: lw tw syst final}, which discusses the equivalence of  the problem \eqref{eqn: main system3} for $A_j$ and $P_j$ and the ultimate long wave system \eqref{eqn: lw syst-nu}.
Of course, our notation must keep up with these changes of variables, and we summarize in Table \ref{table: notation} the evolution of a typical operator's typesetting across these different problems.

%%-----------------------------------------------------------%%
%%-----------------------------------------------------------%%
\begin{table}
\[
\begin{tabular}{c|l}
\hline
Symbol &Use \\
\hline\\[-10pt]
$\Rsf$ &The original problem \eqref{eqn: main system3} \\
\hline\\[-10pt]
$\tRsf^c$ &The travelling wave problem \eqref{eqn: tw syst fm} \\[3pt]
\hline\\[-10pt]
$\bRsf^{\ep}$ &The preliminary long wave problem \eqref{eqn: lw syst-ep} \\[3pt]
\hline\\[-10pt]
$\Rcal^{\nu}$ &The final long wave problem \eqref{eqn: lw syst-nu} \\[3pt]
\hline
\end{tabular}
\]

\caption{Summary of notational evolution.}
\label{table: notation}
\end{table}

%%-----------------------------------------------------------%%
%%-----------------------------------------------------------%%
\begin{remark}\label{rem: speed of sound}
The linearization of \eqref{eqn: main system3} at 0 yields
\[
\dot{A}_j
= \tau_2\delta^+\delta^-A_j,
\qquad
P_j = R_j = 0.
\]
If we follow the discussion after \cite[Thm.\@ 1.1]{friesecke-pego1}, as well as \cite[Rem.\@ 2.2]{faver-wright}, and look for plane wave solutions $A_j(t) = e^{ikj-i\omega{t}}$ with $\omega$, $k \in \R$, we find the dispersion relation
\begin{equation}\label{eqn: disp rel}
-i\omega
= 2\tau_2(\cos(k)-1).
\end{equation}
The only real solutions are $\omega=0$ and $k \in 2\pi\Z$. 
Previously, in \cite{friesecke-pego1, faver-wright} a nontrivial dispersion relation $\omega = \omega(k)$ was found by making the same kind of plane wave Ansatz, and the result `phase speed' $k \mapsto \omega(k)/k$ had a nonzero maximum $c_s$, which was called the `speed of sound.'
These articles then proceeded to look for travelling waves with speed slightly above their respective values of $c_s$; these were `supersonic' waves.
For us, $\omega(k)$ is identically zero, which suggests that the speed of sound for our auxin problem is 0.
Our long wave scaling in Section \ref{sec: lw} analytically justifies this intuition.
\end{remark}

%%-----------------------------------------------------------%%
%%-----------------------------------------------------------%%
%%-----------------------------------------------------------%%
%%-----------------------------------------------------------%%
\subsection{The travelling wave Ansatz}\label{sec: tw}
We now look for solutions $A_j$ and $P_j$ to \eqref{eqn: main system3} of the form
\begin{equation}\label{eqn: tw ansatz}
A_j = \phi_1(j-ct)
\quadword{and} 
P_j = \phi_2(j-ct).
\end{equation}
The profiles $\phi_1$ and $\phi_2$ are real-valued functions of a single real variable and $c \in \R$.
The following manipulations will be justified if we assume $\phi_1 \in H_q^1$ and $\phi_2 \in W^{1,\infty}$; we discuss the exponentially localized Sobolev space $H_q^1$ in Appendix \ref{app: expn loc Sob space}.
Furthermore, since we want $P_j$ to vanish at $-\infty$ and be asymptotically constant at $+\infty$, per the limits \eqref{eqn: profile limits} and the numerical predictions of Fig.~\ref{fig:int:profile}, we expect that $\phi_2$ should vanish at $+\infty$ and be asymptotically constant at $-\infty$.

We will convert the problem \eqref{eqn: main system3} for $A_j$ and $P_j$ into a nonlocal system for $\phi_1$ and $\phi_2$, with $c$ as a parameter.
Doing so amounts to little more than changing variables {\it{many}} times in the integral operators defined in Sections \ref{sec: Pj analysis} and \ref{sec: Rj analysis} and gives us a host of new integral operators that will constitute the problem for $\phi_1$ and $\phi_2$.

In what follows we assume $f \in L^1$ and $g \in L^{\infty}$, so that the operators below are defined in the special cases of $f = \phi_1 \in H_q^1$ and $g = \phi_2 \in W^{1,\infty}$.
First, for $x$, $v \in \R$, put
\begin{equation}\label{eqn: tEsf-c}
\tEsf^c(f)(v,x)
:= \exp\left(\frac{\kappa}{c}\int_v^x f(u+1) \du\right)
\end{equation}
and
\begin{equation}\label{eqn: tPsf1-c}
\tPsf_1^c(f)(x)
:= \frac{\alpha}{c}\int_x^{\infty} \tEsf^c(f)(v,x)f(v) \dv.
\end{equation}
Then we use the Ansatz \eqref{eqn: tw ansatz} and the definition of $\Psf_1$ in \eqref{eqn: Psf1} to find
\[
\Psf_1(A_{j+1},A_j)(t)
= \alpha\int_{-\infty}^t \exp\left(-\kappa\int_s^t \phi_1(j-c\xi+1)\dxi\right)\phi_1(j-cs) \ds
= \tPsf_1^c(\phi_1)(j-ct).
\]
Here we have substituted $u = j-c\xi$ in the exponential's integral and then $v = j-cs$ throughout.

Similar substitutions, which we do not discuss, yield the following identities.
Put
\begin{equation}\label{eqn: tPsf2-c}
\tPsf_2^c(f,g)(x)
:= \frac{1}{c}\int_x^{\infty} \tEsf^c(f)(v,x)\Qsf_2(f(v+1),g(v)) \dv,
\end{equation}
so that with $\Psf_2$ defined in \eqref{eqn: Psf2} we have
\[
\Psf_2(A_{j+1},P_j)(t)
= \tPsf_2^c(\phi_1,\phi_2)(j-ct).
\]
Thus $\phi_2$ must satisfy
\begin{equation}\label{eqn: tw eqn for phi2}
\phi_2
= \tPsf_1^c(\phi_1) + \tPsf_2^c(\phi_1,\phi_2),
\end{equation}
which indicates that, as expected, $\phi_2$ should vanish at $+\infty$ and be asymptotically constant at $-\infty$.

Now we reformulate the equation for $A_j$, equivalently, for $\phi_1$.
Put
\begin{equation}\label{eqn: tRsf1-c}
\tRsf_1^c(f)(x)
:= \frac{\kappa\tau_1}{c}\int_x^{\infty} f(u+1)\tPsf_1^c(f)(u) \du,
\end{equation}
so that with $\Rsf_1$ defined in \eqref{eqn: Rsf1} we have
\[
\Rsf_1(A_{j+1},A_j)(t)
= \tRsf_1^c(\phi_1)(j-ct).
\]
Put
\begin{equation}\label{eqn: Rsf2-c}
\tRsf_2^c(f,g)(x)
:= \frac{1}{c}\int_x^{\infty} \big(\kappa{f}(u+1)\tPsf_2^c(f,g)(u) - \Qsf_2(f(u+1),g(u))\big) \du
\end{equation}
and 
\begin{equation}\label{eqn: Rsf-c}
\tRsf^c(f,g) 
:= \tRsf_1^c(f) + \tRsf_2^c(f,g),
\end{equation}
so that with $\Rsf_2$ defined in \eqref{eqn: Rsf2} and $\Rsf$ in \eqref{eqn: Rsf} we have
\[
\Rsf_2(A_{j+1},A_j,P_j)(t) = \tRsf_2^c(\phi_1,\phi_2)(j-ct)
\quadword{and}
\Rsf(A_{j+1},A_j,P_j)(t) = \tRsf^c(\phi_1,\phi_2)(j-ct).
\]
Last, put
\begin{equation}\label{eqn: Nsf-c}
\tNsf^c(f,g)(x)
:= \tau_1\tRsf_2^c(f,g)(x)f(x)-\tau_1\Qsf_1(f(x),\tRsf^c(f,g)(x)),
\end{equation}
so that with $\Nsf$ defined in \eqref{eqn: Nsf} we have
\[
\Nsf(A_{j+1},A_j,P_j)(t)
= \tNsf^c(\phi_1,\phi_2)(j-ct).
\]

For a function $f \colon \R \to \R$ and $d \in \R$, define, as in \eqref{eqn: shift intro}, the shift operator $S^d$ by
\begin{equation}\label{eqn: shift operator}
(S^df)(x)
:= f(x+d).
\end{equation}
This final piece of notation, along with the equation \eqref{eqn: tw eqn for phi2}, allows us to convert the problem \eqref{eqn: main system3} for $A_j$ and $P_j$ into the following nonlocal system for $\phi_1$ and $\phi_2$:
\begin{equation}\label{eqn: tw syst}
\begin{cases}
-c\phi_1'
= \tau_2(S^1-2+S^{-1})\phi_1
+ (S^{-1}-1)\big(\tRsf_1^c(\phi_1)\phi_1+ \tNsf^c(\phi_1,\phi_2)\big), \\
\\
\phi_2
= \tPsf_1^c(\phi_1) + \tPsf_2^c(\phi_1,\phi_2).
\end{cases}
\end{equation}

%%-----------------------------------------------------------%%
%%-----------------------------------------------------------%%
%%-----------------------------------------------------------%%
\subsection{The Fourier multiplier structure}
We summarize our conventions and definitions for Fourier transforms and Fourier multipliers in Appendix \ref{app: fourier}.
If we take the Fourier transform of the equation for $\phi_1$ in \eqref{eqn: tw syst}, we find
\[
\big(ick + 2\tau_2(\cos(k)-1)\big)\hat{\phi}_1(k)
= (1-e^{-ik})\ft\big[\tRsf_1^c(\phi_1)\phi_1+ \tNsf^c(\phi_1,\phi_2)\big](k).
\]
For $k \in \R$, we have $ick+2\tau_2(\cos(k)-1) = 0$ if and only if $k=0$.
Consequently, the function
\begin{equation}\label{eqn: tMc}
\tMsf_c(k)
:= \frac{1-e^{-ik}}{ick+2\tau_2(\cos(k)-1)}
\end{equation}
has a removable singularity at 0 and is in fact analytic on $\R$.
We therefore define $\Msf_c$ to be the Fourier multiplier with symbol $\tMsf_c$, i.e., $\Msf_c$ satisfies
\[
\hat{\Msf_cf}(k)
= \tMsf_c(k)\hat{f}(k).
\]
We discuss some further properties of Fourier multipliers in Appendix \ref{app: fm}. 
Now the problem \eqref{eqn: tw syst} is equivalent to
\begin{equation}\label{eqn: tw syst fm}
\begin{cases}
\phi_1 = \Msf_c\big(\tRsf_1^c(\phi_1)\phi_1 + \tNsf^c(\phi_1,\phi_2)\big) \\
\phi_2 = \tPsf_1^c(\phi_1) + \tPsf_2^c(\phi_1,\phi_2).
\end{cases}
\end{equation}

%%-----------------------------------------------------------%%
%%-----------------------------------------------------------%%
%%-----------------------------------------------------------%%
%%-----------------------------------------------------------%%
%%-----------------------------------------------------------%%
\section{The Long Wave Problem}

%%-----------------------------------------------------------%%
%%-----------------------------------------------------------%%
%%-----------------------------------------------------------%%
%%-----------------------------------------------------------%%
\subsection{The long wave scaling}\label{sec: lw}
We now make the long wave Ansatz
\begin{equation}\label{eqn: lw ansatz}
\phi_1(x) = \ep\psi_1(\ep^{\mu}x),
\qquad
\phi_2(x) = \ep^{\beta}\psi_2(\ep^{\mu}x),
\quadword{and}
c = \ep^{\gamma}c_0.
\end{equation}
We assume, as with $\phi_1$ and $\phi_2$, that the scaled profiles satisfy $\psi_1 \in H_q^1$ and $\psi_2 \in W^{1,\infty}$.
We think of $\ep > 0$ as small and keep the exponents $\beta$, $\gamma$, $\mu > 0$ arbitrary for now; eventually we will pick 
\[
\gamma = \mu = \frac{2}{5}
\quadword{and}
\beta = \frac{1}{5}.
\]
The reasoning behind this choice is by no means obvious at this point and will not be for some time; leaving $\mu$, $\beta$, and $\gamma$ arbitrary will allow this choice to appear more naturally (at the cost of temporarily more cumbersome notation).

As we intuited in Remark \ref{rem: speed of sound}, our wave speed is now close to 0, which is the auxin problem's natural `speed of sound.'
The parameter $c_0$ affords us some additional flexibility in choosing the wave speed.
A properly chosen value of $c_0$ will cause the maximum of the leading-order term of $\phi_1$ to be $\ep$, which will fulfill our promise in Section \ref{sec:sub:int:mr} that the auxin-height is, to leading order, $\ep$.
Friesecke and Pego introduce a similar auxiliary parameter into their $\ep$-dependent wave speed, see \cite[Eq.\@ (2.5), (2.13)]{friesecke-pego1}.
This parameter allows them to prove that the dependence of their travelling wave profile on wave speed is sufficiently regular in different function spaces, a result needed for their subsequent stability arguments in \cite{friesecke-pego2, friesecke-pego3, friesecke-pego4}.
We did not provide this extra parameter in our version \eqref{eqn: FP wave speed} of the Friesecke-Pego wave speed, but rather we selected it so that the amplitude of the leading order $\sech^2$-profile term in \eqref{eqn: FP sech} is 1.
Similarly, we will not pursue their depth of wave-speed analysis on our profiles' dependence on $c_0$.

We convert \eqref{eqn: tw syst fm} to another nonlocal system for $\psi_1$ and $\psi_2$, which now depends heavily on the parameter $\ep$.
As before, this process mostly amounts to changing variables in many integrals.
For example, we use the definition of $\tPsf_1^c$ in \eqref{eqn: tPsf1-c} and the Ansatz \eqref{eqn: lw ansatz} to find
\begin{equation}\label{eqn: tPsf1-c lw intermed}
\tPsf_1^c(\phi_1)(x)
= \frac{\alpha}{\ep^{\gamma}c_0}\int_x^{\infty} \tEsf^c(\phi_1)(v,x)\ep\psi_1(\ep^{\mu}v) \dv,
\end{equation}
where, using the definition of $\tEsf^c$ in \eqref{eqn: tEsf-c}, we have
\[
\tEsf^c(\phi_1)(v,x)
= \exp\left(\frac{\kappa}{\ep^{\gamma}c_0}\int_v^x \ep\psi_1(\ep^{\mu}u + \ep^{\mu}) \du\right) \\
= \exp\left(\frac{\kappa}{c_0}\ep^{1-(\gamma+\mu)}\int_{\ep^{\mu}v}^{\ep^{\mu}x} \psi_1(U + \ep^{\mu}) \dU\right).
\]
Here we have substituted $U = \ep^{\mu}u$.

Now for $f \in L^1$ we put
\begin{equation}\label{eqn: E}
\E(f)(V,X)
:= \exp\left(\frac{\kappa}{c_0}\int_V^X f(U) \dU\right), \ V, X \in \R,
\end{equation}
so that \eqref{eqn: tPsf1-c lw intermed} becomes
\[
\tPsf_1^c(\phi_1)(x)
= \frac{\alpha}{c_0}\ep^{1-\gamma}\int_x^{\infty} \E(\ep^{1-(\gamma+\mu)}S^{\ep^{\mu}}\psi_1)(\ep^{\mu}v,\ep^{\mu}x)\psi_1(\ep^{\mu}v) \dv.
\]
Here $S^{\ep^{\mu}}$ is the shift operator defined in \eqref{eqn: shift operator} with $d = \ep^{\mu}$.
We substitute again with $V = \ep^{\mu}v$ and define
\begin{equation}\label{eqn: bPsf1-ep}
\bPsf_1^{\ep}(f)(X)
:= \frac{\alpha}{c_0}\int_X^{\infty} \E(\ep^{1-(\gamma+\mu)}S^{\ep^{\mu}}f)(V,X)f(V)\dV
\end{equation}
to conclude that
\[
\tPsf_1^c(\phi_1)(x)
= \ep^{1-(\gamma+\mu)}\bPsf_1^{\ep}(\psi_1)(\ep^{\mu}x).
\]

Similar careful substitutions will allow us to reformulate the integral operators from Section \ref{sec: tw} in terms of the long wave Ansatz.
First, however, we define
\begin{equation}\label{eqn: bQsf-ep}
\bQsf_1^{\ep}(X,Y) := \frac{X^2Y}{k_a+\ep{X}}
\quadword{and}
\bQsf_2^{\ep}(X,Y) := \kappa\frac{k_rY+k_m\ep^{1-\beta}X+\ep{XY}}{(k_r+\ep{X})(k_m+\ep^{\beta}Y)}XY .
\end{equation}
When $\ep \ne 0$, this definition permits the very convenient factorizations
\[
\Qsf_1(\ep{X},\ep^{1-(\gamma+\mu)}Y)
= \ep^{3-(\gamma+\mu)}\bQsf_1^{\ep}(X,Y)
\quadword{and}
\Qsf_2(\ep{X},\ep^{\beta}Y)
= \ep^{1+2\beta}\bQsf_2^{\ep}(X,Y),
\]
where $\Qsf_1$ was defined in \eqref{eqn: Qsf1} and $\Qsf_2$ in \eqref{eqn: Qsf2}.

Now we work on the travelling wave integral operators.
Below we will assume $f \in L^1$ and $g \in L^{\infty}$.
Put
\begin{equation}\label{eqn: bPsf2-ep}
\bPsf_2^{\ep}(f,g)(X)
:= \frac{1}{c_0}\int_X^{\infty} \E(\ep^{1-(\gamma+\mu)}S^{\ep^{\mu}}f)(V,X)\bQsf_2^{\ep}(f(V+\ep^{\mu}),g(V)) \dV,
\end{equation}
so that with $\tPsf_2^c$ defined in \eqref{eqn: tPsf2-c} we have
\[
\tPsf_2^c(\phi_1,\phi_2)(x)
= \ep^{1-(\gamma+\mu)+2\beta}\bPsf_2^{\ep}(\psi_1,\psi_2)(\ep^{\mu}x).
\]
This converts the second equation in \eqref{eqn: tw syst fm} for $\phi_2$ to
\[
\ep^{\beta}\psi_2(\ep^{\mu}x)
= \ep^{1-(\gamma+\mu)}\bPsf_1^{\ep}(\psi_1)(\ep^{\mu}x) + \ep^{1-(\gamma+\mu)+2\beta}\bPsf_2^{\ep}(\psi_1,\psi_2)(\ep^{\mu}x).
\]
Passing to $X = \ep^{\mu}x$, we find that $\psi_2$ must satisfy
\begin{equation}\label{eqn: psi2 lw eqn intermed}
\psi_2(X)
= \ep^{1-(\gamma+\mu)-\beta}\bPsf_1^{\ep}(\psi_1)(X) + \ep^{1-(\gamma+\mu)+\beta}\bPsf_2^{\ep}(\psi_1,\psi_2)(X).
\end{equation}

Now put
\begin{equation}\label{eqn: bRsf1-ep}
\bRsf_1^{\ep}(f)(X)
:= \frac{\kappa\tau_1}{c_0}\int_X^{\infty} \bPsf_1^{\ep}(f)(V)f(V + \ep^{\mu}) \dV,
\end{equation}
so that with $\tRsf_1^c$ defined in \eqref{eqn: tRsf1-c} we have
\[
\tRsf_1^c(\phi_1)(x)
= \ep^{2(1-(\gamma+\mu))}\bRsf_1^{\ep}(\psi_1)(\ep^{\mu}x).
\]
Put
\begin{equation}\label{eqn: bRsf2-ep}
\bRsf_2^{\ep}(f,g)(X)
:= \frac{1}{c_0}\int_X^{\infty} \big(\ep^{1-(\gamma+\mu)}\kappa{f}(V+\ep^{\mu})\bPsf_2^{\ep}(f,g)(V)-\bQsf_2^{\ep}(f(V+\ep^{\mu}),g(V))\big) \dV
\end{equation}
and 
\begin{equation}\label{eqn: bRsf-ep}
\bRsf^{\ep}(f,g)(X)
:= \ep^{1-(\gamma+\mu)}\bRsf_1^{\ep}(f)(X) + \ep^{2\beta}\Rsf_2^{\ep}(f,g)(X),
\end{equation}
so that with $\tRsf_2^c$ defined in \eqref{eqn: Rsf2-c} and $\tRsf^c$ defined in \eqref{eqn: Rsf-c} we have
\[
\tRsf_2^c(\phi_1,\phi_2)(x)
= \ep^{1-(\gamma+\mu)+2\beta}\bRsf_2^{\ep}(\psi_1,\psi_2)(\ep^{\mu}x)
\quadword{and}
\tRsf^c(\phi_1,\phi_2)(x)
= \ep^{1-(\gamma+\mu)}\bRsf^{\ep}(\psi_1,\psi_2)(\ep^{\mu}x).
\]
Finally, put
\begin{equation}\label{eqn: bNsf-ep}
\bNsf^{\ep}(f,g)(X)
:= \tau_1\bRsf_2^{\ep}(f,g)(X)f(X)
- \ep^{1-2\beta}\tau_1\bQsf_1^{\ep}(f(X),\bRsf^{\ep}(f,g)(X)),
\end{equation}
so that with $\tNsf^c$ defined in \eqref{eqn: Nsf-c} we have
\[
\tNsf^c(\phi_1,\phi_2)(x)
= \ep^{2-(\gamma+\mu)+2\beta}\bNsf^{\ep}(\psi_1,\psi_2)(\ep^{\mu}x).
\]

The definition of scaled Fourier multipliers from \eqref{eqn: scaled fm} tells us that, for $\ep > 0$, $\Msf_{\ep^{\gamma}c_0}^{(\ep^{\mu})}$ is the Fourier multiplier satisfying
\[
\hat{\Msf_{\ep^{\gamma}c_0}^{(\ep^{\mu})}f}(k)
= \tMsf_{\ep^{\gamma}c_0}(\ep^{\mu}k)\hat{f}(k),
\]
where $\tMsf_{\ep^{\gamma}c_0}$ is defined by taking $c = \ep^{\gamma}c_0$ in \eqref{eqn: tMc}.
This converts the first equation in \eqref{eqn: tw syst fm} for $\phi_1$ to 
\[
\ep\psi_1(\ep^{\mu}x)
= \Msf_{\ep^{\gamma}c_0}^{(\ep^{\mu})}[\ep^{2(1-(\gamma+\mu))}\bRsf_1^{\ep}(\psi_1)\ep\psi_1 + \ep^{2-(\gamma+\mu)+2\beta}\bNsf^{\ep}(\psi_1,\psi_2)](\ep^{\mu}x).
\]
We factor this to reveal
\begin{equation}\label{eqn: psi1 lw eqn intermed}
\psi_1(X)
= \ep^{2(1-(\gamma+\mu))}\Msf_{\ep^{\gamma}c_0}^{(\ep^{\mu})}\big[\bRsf_1^{\ep}(\psi_1)\psi_1 + \ep^{-1+\gamma+\mu+2\beta}\Nsf^{\ep}(\psi_1,\psi_2)\big](X).
\end{equation}
We abbreviate 
\begin{equation}\label{eqn: bMu-ep}
\bMsf_{\ep}
:= \ep^{2(1-(\gamma+\mu))}\Msf_{\ep^{\gamma}c_0}^{(\ep^{\mu})}
\end{equation}
to conclude from \eqref{eqn: psi1 lw eqn intermed} and the prior equation \eqref{eqn: psi2 lw eqn intermed} for $\psi_2$ that the long wave profiles must satisfy
\begin{equation}\label{eqn: lw syst-ep}
\begin{cases}
\psi_1 = \bMsf_{\ep}\big[\bRsf_1^{\ep}(\psi_1)\psi_1 + \ep^{-1+\gamma+\mu+2\beta}\bNsf^{\ep}(\psi_1,\psi_2)\big] \\
\\
\psi_2
= \ep^{1-(\gamma+\mu+\beta)}\bPsf_1^{\ep}(\psi_1) + \ep^{1-(\gamma+\mu)+\beta}\bPsf_2^{\ep}(\psi_1,\psi_2).
\end{cases}
\end{equation}

We have been tacitly assuming that all of the exponents on powers of $\ep$ above are nonnegative so that the various $\ep$-dependent operators and prefactors are actually defined at $\ep=0$.
In particular, this demands 
\begin{equation}\label{eqn: expn bare minimum}
1-2\beta \ge 0,
\qquad
-1+\gamma+\mu+2\beta \ge 0,
\quadword{and}
1-(\gamma+\mu+\beta) \ge 0.
\end{equation}

%%-----------------------------------------------------------%%
%%-----------------------------------------------------------%%
%%-----------------------------------------------------------%%
%%-----------------------------------------------------------%%
\subsection{The formal long wave limit and exponent selection}\label{sec: formal lw limit}
Our intention is now to take the limit $\ep \to 0$ in the equations \eqref{eqn: lw syst-ep} for $\psi_1$ and $\psi_2$.
Doing so in a way that the limit is both meaningful (i.e., defined and nontrivial) and reflective of what the numerics predict at $\ep=0$ will teach us what the exponents $\mu$, $\gamma$, and $\beta$ should be, beyond the requirements of \eqref{eqn: expn bare minimum}. 

%%-----------------------------------------------------------%%
%%-----------------------------------------------------------%%
%%-----------------------------------------------------------%%
\subsubsection{The formal limit on $\bMsf_{\ep}$ and the selection of the exponents $\gamma$ and $\mu$}\label{sec: bMu-ep formal}
We want to assign a `natural' definition to $\bMsf_0$, where $\bMsf_{\ep}$ was defined, for $\ep > 0$, in \eqref{eqn: bMu-ep}.
However, we relied above on having $\ep > 0$ to invoke the scaled Fourier multiplier identity \eqref{eqn: scaled fm} that gave us $\bMsf_{\ep}$, and naively setting $\ep = 0$ in that identity is meaningless.
Additionally, we should be careful that the prefactor $\ep^{2(1-(\gamma+\mu))}$ in \eqref{eqn: bMu-ep} does not lead us to define $\bMsf_0 = 0$; otherwise, we would have $\psi_1 = 0$ when $\ep = 0$, and that is not what the numerics in Fig.~\ref{fig:int:profile} predict.

A natural starting point, then, is to study $\bMsf_{\ep}$ in the limit $\ep \to 0^+$, and this amounts to considering the limit of its symbol, whose definition we extract from the definition of $\bMsf_{\ep}$ in \eqref{eqn: bMu-ep} and the definition of the scaled Fourier multiplier in \eqref{eqn: scaled fm}.  
Thus, for each $k \in \R$, we want the limit
\begin{equation}\label{eqn: Mc delta limit}
\lim_{\ep \to 0^+} \ep^{2(1-(\gamma+\mu))}\tMsf_{\ep^{\gamma}c_0}(\ep^{\mu}k)
\end{equation}
to exist without being identically zero.
The function $\tMsf_{\ep^{\gamma}c_0}$ was defined in \eqref{eqn: tMc}.

To calculate this limit, we first state the Taylor expansions
\begin{equation}\label{eqn: N}
1-e^{-iz} = iz+iz^2N_1(z)
\quadword{and}
\cos(z)-1 = -\frac{z^2}{2}+ \frac{iz^4N_2(z)}{2\tau_2}
\end{equation}
for $z \in \C$.
The functions $N_1$ and $N_2$ are analytic and uniformly bounded on strips in the sense that
\begin{equation}\label{eqn: Cq}
C_q
:= \sup_{x \in \R} |N_1(x\pm{iq})| + |N_2(x\pm{iq})|
< \infty
\end{equation}
for any $q > 0$.
The choice of constants on $N_1$ and $N_2$ will permit some useful cancellations later.
Then
\[
\tMsf_c(k)
= \frac{ik+ik^2N_1(k)}{ick -\tau_2k^2 + ik^4N_2(k)} 
= \frac{1+kN_1(k)}{c+i\tau_2k + k^3N_2(k)},
\]
and so
\begin{equation}\label{eqn: Mc-ep}
\ep^{2(1-(\gamma+\mu))}\tMsf_{\ep^{\gamma}c_0}(\ep^{\mu}k)
= \ep^{2(1-(\gamma+\mu))}\frac{1+\ep^{\mu}kN_1(\ep^{\mu}k)}{\ep^{\gamma}c_0 + i\tau_2\ep^{\mu}k + \ep^{3\mu}k^3N_2(\ep^{\mu}k)}.
\end{equation}
At this point it does not make sense to set $\ep = 0$, as then the denominator would be identically zero.
So, we would like to factor some power of $\ep$ out of the denominator.
Since the first term in the denominator has a factor of $\ep^{\gamma}$ and the second a factor of $\ep^{\mu}$, we assume $\gamma = \mu$ and remove the power of $\ep$ from both the first and the second terms.
We discuss the choice of $\gamma=\mu$ further in Remark \ref{rem: why gamma=nu}.

Then
\begin{equation}\label{eqn: Mc-ep2}
\ep^{2(1-(\gamma+\mu))}\tMsf_{\ep^{\gamma}c_0}(\ep^{\mu}k)
= \ep^{2(1-2\gamma)}\tMsf_{\ep^{\gamma}c_0}(\ep^{\gamma}k)
= \ep^{2(1-2\gamma)-\gamma}\frac{1+\ep^{\gamma}kN_1(\ep^{\gamma}k)}{c_0+i\tau_2k+\ep^{2\gamma}k^3N_2(\ep^{\gamma}k)}.
\end{equation}
Pointwise in $k$ we have
\[
\lim_{\ep \to 0^+} \frac{1+\ep^{\gamma}kN_1(\ep^{\gamma}k)}{c_0+i\tau_2k+\ep^{2\gamma}k^3N_2(\ep^{\gamma}k)}
= \frac{1}{c_0+i\tau_2k},
\]
and so we want 
\[
2(1-2\gamma)-\gamma
= 0
\]
so that the prefactor of $\ep^{2(1-2\gamma)-\gamma}$ in \eqref{eqn: Mc-ep2} does not induce a trivial or undefined limit.
Thus we take
\[
\gamma
= \mu
= \frac{2}{5}.
\]
Certainly doing so does not contradict any of the inequalities in \eqref{eqn: expn bare minimum}, provided that $\beta$ is chosen appropriately.
Moreover, the power of $2/5$ agrees with the height-speed-width relations suggested in Fig.\@ \ref{fig:int:scalings}.
And so
\[
\lim_{\ep \to 0^+} \ep^{2(1-(\gamma+\mu))}\tMsf_{\ep^{\gamma}c_0}(\ep^{\mu}k)
= \lim_{\ep \to 0^+} \ep^{4/5}\tMsf_{\ep^{2/5}c_0}(\ep^{2/5}k)
= \frac{1}{c_0+\tau_2ik}.
\]

Put
\begin{equation}\label{eqn: tvarpi-0}
\tM^{(0)}(z)
:= \frac{1}{c_0+i\tau_2z},
\end{equation}
so $\tM^{(0)}$ is analytic on any strip $\set{z \in \C}{|\im(z)| < q}$ for $q \in (0,\tau_2/c_0)$.
Let $\M^{(0)}$ be the Fourier multiplier with symbol $\tM^{(0)}$.

Lemma \ref{lem: Beale fm lemma} then gives the following properties of $\M^{(0)}$; the identities \eqref{eqn: varpi-0 ids} are direct calculations with the Fourier transform.

%%-----------------------------------------------------------%%
%%-----------------------------------------------------------%%
\begin{lemma}\label{lem: varpi0}
Fix $q \in (0,\tau_2/c_0)$.
Then $\M^{(0)} \in \b(H_q^r,H_q^{r+1})$ for all $r$.
More generally, if $f \in H^1$ and $g \in L^2$, then
\begin{equation}\label{eqn: varpi-0 ids}
\M^{(0)}(c_0+ \tau_2\partial_X)f = f
\quadword{and}
(c_0+\tau_2\partial_X)\M^{(0)}g = g.
\end{equation}
\end{lemma}

Because of the identities \eqref{eqn: varpi-0 ids}, we write $\M^{(0)} = (c_0+\tau_2\partial_X)^{-1}$.
The formal analysis above then leads us to expect
\begin{equation}\label{eqn: bMu-ep to 0}
\lim_{\ep \to 0^+} \bMsf_{\ep}
= \M^{(0)}
= (c_0+\tau_2\partial_X)^{-1}.
\end{equation}
However, we have not yet proved this rigorously by any means.

%%-----------------------------------------------------------%%
%%-----------------------------------------------------------%%
\begin{remark}\label{rem: why gamma=nu}
Here is why we take $\gamma=\mu$ when factoring the power of $\ep$ out of the denominator in \eqref{eqn: Mc-ep}.
First, taking $\gamma > \mu$ produces
\[
\ep^{2(1-(\gamma+\mu))}\tMsf_{\ep^{\gamma}c_0}(\ep^{\mu}k)
= \ep^{2(1-(\gamma+\mu))-\mu}\frac{1+\ep^{\mu}kN_1(\ep^{\mu}k)}{\ep^{\gamma-\mu}c_0 + i\tau_2k + \ep^{2\mu}k^3N_2(\ep^{\mu}k)}
\]
instead of \eqref{eqn: Mc-ep2}.
If $2(1-(\gamma+\mu))-\mu > 0$, then the right side above is identically zero at $\ep = 0$, and so we demand $2(1-(\gamma+\mu))-\mu = 0$; there are many pairs of $\gamma$ and $\mu$ that work here.
But then
\[
\lim_{\ep \to 0^+} \frac{1+\ep^{\mu}kN_1(\ep^{\mu}k)}{\ep^{\gamma-\mu}c_0 + i\tau_2k + \ep^{2\mu}k^3N_2(\ep^{\mu}k)}
= \frac{1}{i\tau_2k}.
\]
This suggests that instead of \eqref{eqn: bMu-ep to 0}, we have
\[
\lim_{\ep \to 0^+} \bMsf_{\ep}
= (\tau_2\partial_X)^{-1}.
\]
However, this is meaningless: differentiation is not invertible from $H_q^r$ to $H_q^{r+1}$. 

Taking $\gamma < \mu$ also does not work. 
In that case, instead of \eqref{eqn: Mc-ep2} we would have found
\[
\ep^{2(1-(\gamma+\mu))}\tMsf_{\ep^{\gamma}c_0}(\ep^{\mu}k)
= \ep^{2(1-(\gamma+\mu))-\gamma}\frac{1+\ep^{\mu}kN_1(\ep^{\mu}k)}{c_0 + i\tau_2\ep^{\mu-\gamma}k + \ep^{3\mu-\gamma}k^3N_2(\ep^{\mu}k)}.
\]
Since $\gamma < \mu$ we find
\[
\lim_{\ep \to 0^+} \frac{1+\ep^{\mu}kN_1(\ep^{\mu}k)}{c_0 + i\tau_2\ep^{\mu-\gamma}k + \ep^{3\mu-\gamma}k^3N_2(\ep^{\mu}k)}
= \frac{1}{c_0}.
\]
We would then want $2-3\gamma-2\mu = 0$ to prevent a nontrivial limit.

Choosing $\gamma$ and $\mu$ appropriately, we conclude that at $\ep = 0$ the equation for $\psi_1$ from \eqref{eqn: lw syst-ep} formally reduces to
\[
\psi_1
= \frac{1}{c_0}\bRsf_1^0(\psi_1)\psi_1.
\]
Numerically we expect $\psi_1(X) > 0$ for all $X$ when $\ep = 0$, and so, using the definition of $\bRsf_1^0$ from \eqref{eqn: bRsf1-ep}, we have
\[
c_0
= \bRsf_1^0(\psi_1)(X)
= \frac{\alpha\kappa\tau_1}{c_0^2}\int_X^{\infty} \left(\int_V^{\infty} \psi_1(W)\dW\right)\psi_1(V) \dV.
\]
Differentiating, we find
\[
\left(\int_X^{\infty}\psi_1(W) \dW\right)\psi_1(X)
= 0.
\]
But since $\psi_1(W) > 0$ for all $W$, we cancel the integral factor to find $\psi_1(X) = 0$, a contradiction to our numerical predictions.
\end{remark}

%%-----------------------------------------------------------%%
%%-----------------------------------------------------------%%
%%-----------------------------------------------------------%%
\subsubsection{The formal leading order equation for $\psi_1$}\label{sec: formal leading order1}
At $\ep = 0$ the equation for $\psi_1$ in \eqref{eqn: lw syst-ep} becomes (again, formally)
\[
\psi_1
= \M^{(0)}\big(\bRsf^0(\psi_1)\psi_1\big)
= (c_0+\tau_2\partial_X)^{-1}\big(\bRsf^0(\psi_1)\psi_1\big).
\]
This is equivalent to
\begin{equation}\label{eqn: psi1 nu=0}
c_0\psi_1+\tau_2\psi_1'
= \bRsf^0(\psi_1)\psi_1.
\end{equation}
We will rewrite this equation so that each term is a perfect derivative.

The definition of $\bRsf_1^{\ep}$ in \eqref{eqn: bRsf1-ep}, valid for all $\ep$, gives
\begin{equation}\label{eqn: bRsf0}
\bRsf^0(\psi_1)(X)
= \frac{\alpha\kappa\tau_1}{c_0^2}\int_X^{\infty} \left(\int_V^{\infty} \psi_1(W) \dW\right)\psi_1(V) \dV.
\end{equation}
Write
\[
\Psi_1(X)
:= \int_X^{\infty} \psi_1(W) \dW,
\]
so that $\Psi_1' = -\psi_1$.
The double integral from \eqref{eqn: bRsf0} is
\begin{multline*}
\int_X^{\infty} \left(\int_V^{\infty} \psi_1(W) \dW\right)\psi_1(V) \dV
= -\int_X^{\infty} \Psi_1(V)\Psi_1'(V) \dV 
= -\int_X^{\infty} \partial_V\left[\frac{\Psi_1^2(V)}{2}\right] \dV \\
= \frac{\Psi_1(X)^2}{2}.
\end{multline*}
Here we are using the requirement that $\psi_1 \in H_q^1$, which implies $\Psi_1(X) \to 0$ as $X \to \infty$.
Thus
\[
\bRsf^0(\psi_1)\psi_1
= -\left(\frac{\alpha\kappa\tau_1}{2c_0^2}\right)\Psi_1^2\Psi_1'
= -\left(\frac{\alpha\kappa\tau_1}{6c_0^2}\right)\partial_X[\Psi_1^3] .
\]

Then \eqref{eqn: psi1 nu=0} is equivalent to 
\[
\tau_2\Psi_1'' + c_0\Psi_1' - \left(\frac{\alpha\kappa\tau_1}{6c_0^2}\right)\partial_X[\Psi_1^3]
= 0.
\]
We integrate both sides from 0 to $\infty$ and use the aforementioned fact that $\Psi_1$ and its derivatives are required to vanish at $\infty$ to find
\begin{equation}\label{eqn: Bernoulli}
\tau_2\Psi_1' + c_0\Psi_1 - \left(\frac{\alpha\kappa\tau_1}{6c_0^2}\right)\Psi_1^3
= 0.
\end{equation}
This is a Bernoulli equation, and it has the solution
\begin{equation}\label{eqn: Sigma}
\Psi_1(X)
= \Sigma(X)
:= \left(\frac{6c_0^3}{\alpha\kappa\tau_1+6c_0^2\exp\big(2c_0X/\tau_2+\theta\big)}\right)^{1/2}. 
\end{equation}

Here $\theta \in \R$ is an arbitrary phase shift.
It follows that putting
\begin{equation}\label{eqn: sigma}
\psi_1(X)
= \sigma(X)
:= -\Sigma'(X)
= \frac{(6c_0^3)^{3/2}\exp\big(2c_0X/\tau_2+\theta\big)}{\tau_2\big[\alpha\kappa\tau_1+6c_0^2\exp\big(2c_0X/\tau_2+\theta\big)\big]^{3/2}}
\end{equation}
solves \eqref{eqn: psi1 nu=0}.

Friesecke and Pego \cite{friesecke-pego1} do not incorporate a phase shift like $\theta$ into their leading order $\sech^2$-type KdV solution, since their broader existence result relies on working in spaces of even functions.
We will not need such symmetry in our subsequent arguments (nor could we achieve it, since no translation of $\sigma$ is even or odd), and so we will leave $\theta$ as an arbitrary free parameter and not specify its value.

%%-----------------------------------------------------------%%
%%-----------------------------------------------------------%%
%%-----------------------------------------------------------%%
\subsubsection{The formal leading order equation for $\psi_2$ and the selection of the exponent $\beta$}\label{sec: formal leading order2}
From our choice of $\gamma=\mu=2/5$ and the inequalities in \eqref{eqn: expn bare minimum}, we need, at the very least,
\[
\frac{1}{10}
\le \beta
\le \frac{1}{5}.
\]
If the strict inequality $\beta < 1/5$ holds, then at $\ep = 0$ the equation for $\psi_2$ in \eqref{eqn: lw syst-ep} reduces to the trivial result $\psi_2 = 0$.
This is not at all what we expect numerically from Figure \ref{fig:int:profile}; rather, we anticipate that $\psi_2$ will asymptote to some nonzero constant at $\infty$.

However, if we instead take $\beta$ so that 
\[
0
= 1-(\gamma+\mu+\beta)
= \frac{1}{5}-\beta,
\]
which is to say, 
\[
\beta
= \frac{1}{5},
\]
then the equation for $\psi_2$ in \eqref{eqn: lw syst-ep} at $\ep = 0$ becomes
\[
\psi_2
= \bPsf_1^0(\psi_1).
\]
Putting
\begin{equation}\label{eqn: zeta}
\psi_2(X)
= \zeta(X)
:= \bPsf_1^0(\sigma)(X)
= \frac{\alpha}{c_0}\int_X^{\infty} \sigma(V) \dV
\end{equation}
therefore solves the leading order equation for $\psi_2$.
We really have
\[
\zeta(X)
= \frac{\alpha}{c_0}\Sigma(X)
= \frac{\alpha}{c_0}\left(\frac{6c_0^3}{\alpha\kappa\tau_1+6c_0^2e^{2c_0X/\tau_2+\theta}}\right)^{1/2},
\]
where $\Sigma$ was defined in \eqref{eqn: Sigma}.

%%-----------------------------------------------------------%%
%%-----------------------------------------------------------%%
%%-----------------------------------------------------------%%
%%-----------------------------------------------------------%%
\subsection{The final long wave system}
With the choices of exponents $\gamma=\mu=2/5$ and $\beta = 1/5$, it becomes convenient to introduce the new small parameter
\begin{equation}\label{eqn: nu-ep}
\nu
:= \ep^{2/5}
\end{equation}
into the problem \eqref{eqn: lw syst-ep} and then recast that problem more cleanly in terms of $\nu$.
First, the long wave Ansatz \eqref{eqn: lw ansatz} becomes
\begin{equation}\label{eqn: nu-ep lw}
\phi_1(x) = \nu^{5/2}\psi_1(\nu{x}),
\qquad
\phi_2(x) = \nu^{1/2}\psi_2(\nu{x}),
\quadword{and}
c = \nu{c}_0.
\end{equation}

Proceeding very much as in Section \ref{sec: lw}, we then define
\begin{equation}\label{eqn: nl-nu}
\nl_1^{\nu}(X,Y	) := \frac{X^2Y}{k_a(k_a+\nu^{5/2}{X})}
\quadword{and}
\nl_2^{\nu}(X,Y) := \kappa\frac{k_rY+k_m\nu^2X+\nu^{5/2}{XY}}{(k_r+\nu^{5/2}{X})(k_m+\nu^{1/2}Y)}XY
\end{equation}
for $X$, $Y \in \R$, while for $f \in L^1$ and $g \in L^{\infty}$, we put
\begin{equation}\label{eqn: P1-nu}
\P_1^{\nu}(f)(X)
:= \frac{\alpha}{c_0}\int_X^{\infty} \E(\nu^{1/2}S^{\nu}f)(V,X)f(V) \dV,
\end{equation}
where $\E$ was defined in \eqref{eqn: E}, and
\begin{equation}\label{eqn: P2-nu}
\P_2^{\nu}(f,g)(X)
:= \frac{1}{c_0}\int_X^{\infty} \E(\nu^{1/2}S^{\nu}f)(V,X)\nl_2^{\nu}(f(V+\nu),g(V)) \dV,
\end{equation}
\begin{equation}\label{eqn: Rcal1-nu}
\Rcal_1^{\nu}(f)(X)
:= \frac{\kappa\tau_1}{c_0}\int_X^{\infty} \P_1^{\nu}(f)(V)f(V + \nu) \dV,
\end{equation}
\begin{equation}\label{eqn: Rcal2-nu}
\Rcal_2^{\nu}(f,g)(X)
:= \frac{1}{c_0}\int_X^{\infty} \big(\nu^{1/2}\kappa{f}(V+\nu)\P_2^{\nu}(f,g)(V)-\nl_2^{\nu}(f(V+\nu),g(V))\big) \dV,
\end{equation}
\begin{equation}\label{eqn: Rcal-nu}
\Rcal^{\nu}(f,g)(X)
:= \Rcal_1^{\nu}(f)(X) + \nu^{1/2}\Rcal_2^{\nu}(f,g)(X),
\end{equation}
and
\begin{equation}\label{eqn: Ncal-nu}
\Ncal^{\nu}(f,g)(X)
:= \tau_1\Rcal_2^{\nu}(f,g)(X)f(X)
- \nu^{3/2}\tau_1\nl_1^{\nu}(f(X),\Rcal^{\nu}(f,g)(X)).
\end{equation}

%%-----------------------------------------------------------%%
%%-----------------------------------------------------------%%
\begin{remark}\label{rem: ops rem}
The operators $\P_1^{\nu}$ and $\Rcal_1^{\nu}$ map $L^1$ into $L^{\infty}$, while $\P_2^{\nu}$ and $\Rcal_2^{\nu}$ map $L^1 \times L^{\infty}$ into $L^{\infty}$, and $\Ncal^{\nu}$ maps $L^1 \times L^{\infty}$ into $L^1$.
More precisely, we could replace $L^1$ with $H_q^1$ and $L^{\infty}$ with $W^{1,\infty}$ and the preceding statement would still be true; see the estimates in Appendix \ref{app: aux ests}.

The operator $\Rcal_1^0$ has the especially simple form
\begin{equation}\label{eqn: Rcal1-0}
\Rcal_1^0(f)(X)
= \left(\frac{\alpha\kappa\tau_1}{6c_0^2}\right)\int_X^{\infty} \left(\int_V^{\infty} f(W) \dW\right) f(V) \dV
\end{equation}
and therefore is differentiable from $L^1$ to $L^{\infty}$.
\end{remark}

Last, for $\nu > 0$, let $\M^{(\nu)}$ be the Fourier multiplier with symbol
\begin{equation}\label{eqn: M-nu}
\tM^{(\nu)}(z)
:= \nu\frac{1-e^{-i\nu{z}}}{ic_0\nu^2z+2\tau_2(\cos(\nu{z})-1)}.
\end{equation}
When $\nu = 0$ we have already defined $\M^{(0)}$ as the Fourier multiplier whose symbol $\tM^{(0)}$ is given in \eqref{eqn: tvarpi-0}.
Now we can state precisely the convergence result that we formally anticipated in Section \ref{sec: bMu-ep formal}, specifically in the limit \eqref{eqn: bMu-ep to 0}.

%%-----------------------------------------------------------%%
%%-----------------------------------------------------------%%
\begin{proposition}\label{prop: varpi-nu conv}
Fix $q \in (0,c_0/\tau_2)$.
There exist $\nu_{\M}$, $C_{\M} > 0$ such that if $0 < \nu < \nu_{\M}$, then 
\[
\norm{\M^{(\nu)}-\M^{(0)}}_{\b(H_q^1)}
\le C_{\M}\nu^{1/3}.
\]
\end{proposition}

We prove this proposition in Appendix \ref{app: proof of prop varpi-nu conv}.
More broadly, we can summarize the work above on the travelling wave Ansatz and subsequent long wave scaling and exponent selection for the system \eqref{eqn: main system3} in the following theorem.

%%-----------------------------------------------------------%%
%%-----------------------------------------------------------%%
\begin{theorem}\label{thm: lw tw syst final}
Suppose
\[
\begin{cases}
A_j(t) = \nu^{5/2}\psi_1(\nu(j-\nu{c}_0t)), \\
P_j(t) = \nu^{1/2}\psi_2(\nu(j-\nu{c}_0t))
\end{cases}
\]
for some $\psi_1 \in H_q^1$ and $\psi_2 \in L^{\infty}$, where $c_0$, $\nu > 0$ and $q \in (0,c_0/\tau_2)$.
Then $A_j$ and $P_j$ satisfy \eqref{eqn: main system3} if and only if $\psi_1$ and $\psi_2$ satisfy
\begin{equation}\label{eqn: lw syst-nu}
\begin{cases}
\psi_1 = \M^{(\nu)}\big(\Rcal_1^{\nu}(\psi_1)\psi_1 + \nu^{1/2}\Ncal^{\nu}(\psi_1,\psi_2)\big), \\
\psi_2 = \P_1^{\nu}(\psi_1) + \nu\P_2^{\nu}(\psi_1,\psi_2).
\end{cases}
\end{equation}
Moreover, taking 
\[
\psi_1 = \sigma
\quadword{and}
\psi_2 = \zeta = \P_1^0(\sigma),
\]
where $\sigma$ is defined in \eqref{eqn: sigma} and $\zeta$ is given explicitly in \eqref{eqn: zeta}, solves \eqref{eqn: lw syst-nu} when $\nu = 0$.
\end{theorem}

We proceed to analyze the system \eqref{eqn: lw syst-nu} with a quantitative contraction mapping argument that tracks its dependence on $\nu$.

%%-----------------------------------------------------------%%
%%-----------------------------------------------------------%%
%%-----------------------------------------------------------%%
%%-----------------------------------------------------------%%
%%-----------------------------------------------------------%%
\section{Analysis of the Long Wave Problem}\label{sec: lw syst-nu soln}

%%-----------------------------------------------------------%%
%%-----------------------------------------------------------%%
%%-----------------------------------------------------------%%
%%-----------------------------------------------------------%%
\subsection{The perturbation Ansatz for the long wave problem \eqref{eqn: lw syst-nu}}
Throughout this section we keep $q \in (0,c_0/\tau_2)$ fixed.
We make the perturbation Ansatz
\begin{equation}\label{eqn: perturb ansatz}
\psi_1 = \sigma + \eta_1
\quadword{and}
\psi_2 = \zeta + \eta_2
\end{equation}
for the long wave problem \eqref{eqn: lw syst-nu}.
Here $\eta_1 \in H_q^1$ and $\eta_2 \in W^{1,\infty}$ are unknown.
We abbreviate 
\[
\etab
= (\eta_1,\eta_2)
\in \X := H_q^1 \times W^{1,\infty},
\]
where $\X$ has the norm
\[
\norm{\etab}_{\X}
:= \norm{\eta_1}_{H_q^1} + \norm{\eta_2}_{W^{1,\infty}}.
\]

The Ansatz \eqref{eqn: perturb ansatz} solves the system \eqref{eqn: lw syst-nu} if and only if $\eta_1$ and $\eta_2$ solve
\begin{equation}\label{eqn: eta1 eta2}
\begin{cases}
\T\eta_1 = \sum_{k=1}^5 \V_{1k}^{\nu}(\etab), \\
\eta_2 = \sum_{k=1}^3 \V_{2k}^{\nu}(\etab),
\end{cases}
\end{equation}
where 
\begin{equation}\label{eqn: T}
\T\eta_1 
:= \eta_1-\M^{(0)}\big[\Rcal_1^0(\sigma)\eta_1+\big(D\Rcal_1^0(\sigma)\eta_1\big)\sigma\big],
\end{equation}
\begin{equation}\label{eqn: V1k}
\begin{aligned}
\V_{11}^{\nu}(\etab) &:= \big(\M^{(\nu)}-\M^{(0)}\big)\big[\Rcal_1^{\nu}(\sigma+\eta_1)(\sigma+\eta_1)\big], \\
\V_{12}^{\nu}(\etab) &:= \M^{(0)}\big[\big(\Rcal_1^{\nu}(\sigma+\eta_1)-\Rcal_1^0(\sigma+\eta_1)\big)(\sigma+\eta_1)\big], \\
\V_{13}^{\nu}(\etab) &:= \M^{(0)}\big[\big(\Rcal_1^0(\sigma+\eta_1)-\Rcal_1^0(\sigma)-D\Rcal_1^0(\sigma)\eta_1\big)\sigma\big], \\
\V_{14}^{\nu}(\etab) &:= \M^{(0)}\big[\big(\Rcal_1^0(\sigma+\eta_1)-\Rcal_1^0(\sigma)\big)\eta_1\big], \\
\V_{15}^{\nu}(\etab) &:= \nu^{1/2}\M^{(0)}\Ncal^{\nu}(\sigma+\eta_1,\zeta+\eta_2) 
\end{aligned}
\end{equation}
and
\begin{equation}\label{eqn: V2k}
\begin{aligned}
\V_{21}^{\nu}(\etab) &:= \P_1^{\nu}(\sigma+\eta_1)-\P_1^0(\sigma+\eta_1), \\
\V_{22}^{\nu}(\etab) &:= \P_1^0(\sigma+\eta_1)-\zeta, \\
\V_{23}^{\nu}(\etab) &:= \nu\P_2^{\nu}(\sigma+\eta_1,\zeta+\eta_2).
\end{aligned}
\end{equation}

We claim that $\T$ is `right-invertible' in the following sense, which we make rigorous in Appendix \ref{app: proof of prop T invert}.

%%-----------------------------------------------------------%%
%%-----------------------------------------------------------%%
\begin{proposition}\label{prop: T invert}
Let $q \in (0,c_0/\tau_2)$.
There exists $\Scal \in \b(H_q^1)$ such that $\T\Scal{g} = g$ for all $g \in H_q^1$.
\end{proposition}

The operator $\T$ is really the linearization at $\psi_1 = \sigma$ and $\nu=0$ of the first equation in \eqref{eqn: lw syst-nu}.
Such a linearization at the limiting localized solution appears as a key operator in numerous FPUT problems, including \cite{friesecke-pego1, faver-wright, hoffman-wright}, and the invertibility of this operator is a property essential to the development of the right fixed point formula for the given problem.   
Our treatment of the invertibility of $\T$ in Appendix \ref{app: proof of prop T invert} is rather different from the analogous inversions in those papers, as here $\T{f}=g$ is really a linearized Bernoulli equation in disguise, rather than the linearized KdV travelling wave profile equation.
In particular, solving $\T{f} = g$ turns into a first-order linear problem, which we can solve explicitly with an integrating factor.
While this requires a fair amount of calculus, we do avoid the more abstract spectral theory that manages the second-order KdV linearizations (see, e.g., \cite[Lem.\@ 4.2]{friesecke-pego1}).

Due to Proposition \ref{prop: T invert}, for $\eta_1 \in H_q^1$ and $\eta_2 \in W^{1,\infty}$ to solve \eqref{eqn: eta1 eta2}, it suffices to take
\begin{equation}\label{eqn: Nfrak1-nu}
\eta_1
= \Scal\sum_{k=1}^5 \V_{1k}^{\nu}(\etab)
=: \Nfrak_1^{\nu}(\etab).
\end{equation}
Subsequently, $\eta_1$ and $\eta_2$ solve \eqref{eqn: eta1 eta2} if and only if
\begin{equation}\label{eqn: Nfrak2-nu}
\eta_2
= \V_{21}^{\nu}(\etab)
+ \V_{22}^{\nu}(\Nfrak_1^{\nu}(\etab))
+ \V_{23}^{\nu}(\etab)
= : \Nfrak_2^{\nu}(\etab).
\end{equation}
We have replaced $\eta_1$ with its fixed point expression \eqref{eqn: Nfrak1-nu} in $\V_{22}^{\nu}$ for the sake of better estimates later; see Appendix \ref{app: V22-nu Lip} for a more precise discussion.
Finally, set
\begin{equation}\label{eqn: Nfrakb-nu}
\Nfrakb^{\nu}(\etab)
:= (\Nfrak_1^{\nu}(\etab),\Nfrak_2^{\nu}(\etab)),
\end{equation}
so $\Nfrakb^{\nu}$ maps $\X$ to $\X$.
More precisely, this follows from the mapping estimates in Appendix \ref{app: map}.
We conclude that the problem \eqref{eqn: eta1 eta2} is equivalent to the fixed point problem
\begin{equation}\label{eqn: final fixed point}
\etab 
= \Nfrakb^{\nu}(\etab),
\end{equation}
which we now solve.

%%-----------------------------------------------------------%%
%%-----------------------------------------------------------%%
%%-----------------------------------------------------------%%
%%-----------------------------------------------------------%%
\subsection{The solution of the fixed point problem \eqref{eqn: final fixed point}}
For $\rho > 0$, we define the ball
\[
\Bfrak(\rho)
:= \set{\etab \in \X}{\norm{\etab}_{\X} \le \rho}.
\]
We prove the following estimates in Appendix \ref{app: proof of main contraction ests}.

%%-----------------------------------------------------------%%
%%-----------------------------------------------------------%%
\begin{proposition}\label{prop: main contraction ests}
There exist $C_{\star}$, $\nu_{\star} > 0$ such that if $0 < \nu < \nu_{\star}$ then the following hold.

%%-----------------------------------------------------------%%
\begin{enumerate}[label={\bf(\roman*)}, ref={(\roman*)}]

%%-----------------------------------------------------------%%
\item\label{part: main mapping}
If $\etab \in \Bfrak(C_{\star}\nu^{1/3})$, then $\Nfrakb^{\nu}(\etab) \in \Bfrak(C_{\star}\nu^{1/3})$.

%%-----------------------------------------------------------%%
\item\label{part: main Lipschitz}
If $\etab$, $\grave{\etab} \in \Bfrak(C_{\star}\nu^{1/3})$, then 
\[
\norm{\Nfrakb^{\nu}(\etab)-\Nfrakb^{\nu}(\grave{\etab})}_{\X}
\le \frac{1}{2}\norm{\etab-\grave{\etab}}_{\X}.
\]
\end{enumerate}
\end{proposition}

Proposition \ref{prop: main contraction ests} guarantees that $\Nfrakb^{\nu}$ is a contraction on $\Bfrak(C_{\star}\nu^{1/3})$ for each $0 < \nu < \nu_{\star}$, and so Banach's fixed point theorem gives the following solution to \eqref{eqn: eta1 eta2}.

%%-----------------------------------------------------------%%
%%-----------------------------------------------------------%%
\begin{theorem}\label{thm: eta1 eta2}
Let $C_{\star}$, $\nu_{\star} > 0$ be as in Proposition \ref{prop: main contraction ests}.
For each $0 < \nu < \nu_{\star}$, there exists $\etab^{\nu} \in \Bfrak(C_{\star}\nu^{1/3})$ such that $\etab^{\nu} = \Nfrakb^{\nu}(\etab^{\nu})$.
\end{theorem}

Theorems \ref{thm: lw tw syst final} and \ref{thm: eta1 eta2}, along with the integral formulations of Section \ref{sec: Rj analysis} and the relation $\nu = \epsilon^{2/5}$, per \eqref{eqn: nu-ep}, together yield the following solutions to our original problem \eqref{eqn: main system} for $A_j$, $P_j$, and $R_j$.
These results are paraphrased nontechnically in \eqref{eq:int:mr:a:p:r:profiles:scales}.

%%-----------------------------------------------------------%%
%%-----------------------------------------------------------%%
\begin{corollary}\label{cor: main corollary}
Let $\alpha$, $\kappa$, $\tau_1$, $\tau_2$, $c_0 > 0$, $q \in (0,c_0/\tau_2)$, and $\theta \in \R$.
Define the leading-order profile terms
\[
\phi_A^*(X)
:= \left(\frac{6\sqrt{6}c_0^{9/2}}{\tau_2}\right)\left(\frac{\exp(2c_0X/\tau_2+\theta)}{\big[\alpha\kappa\tau_1+6c_0^2\exp(2c_0X/\tau_2+\theta)\big]^{3/2}}\right),
\]
\[
\phi_P^*(X)
:= \big((6c_0)^{1/2}\alpha\big)\left(\frac{1}{\big[\alpha\kappa\tau_1+6c_0^2\exp\big(2c_0X/\tau_2+\theta)/\tau_2\big)\big]^{1/2}}\right),
\]
and
\[
\phi_R^*(X)
:= (3\alpha\kappa{c}_0)\left(\frac{1}{\alpha\kappa\tau_1+6c_0^2\exp(2c_0X/\tau_2+\theta)/\tau_2)}\right).
\]
There exists $\ep_{\star} > 0$ such that for each $0 < \ep < \ep_{\star}$, there are $\phi_A^{\ep} \in H_q^1 \cap \Cal^{\infty}$ and $\phi_P^{\ep}$, $\phi_R^{\ep} \in W^{1,\infty} \cap \Cal^{\infty}$ with the following properties.

%%-----------------------------------------------------------%%
\begin{enumerate}[label={\bf(\roman*)}]

%%-----------------------------------------------------------%%
\item
Let
\[
A_j(t)
= \ep\phi_A^*(\ep^{2/5}(j-\ep^{2/5}{c}_0t)) + \ep^{17/15}\phi_A^{\ep}(\ep^{2/5}(j-\ep^{2/5}c_0t)),
\]
\[
P_j(t) 
= \ep^{1/5}\phi_P^*(\ep^{2/5}(j-\ep^{2/5}c_0t))+\ep^{1/3}\phi_P^{\ep}(\ep^{2/5}(j-\ep^{2/5}c_0t)),
\]
and
\[
R_j(t)
= \ep^{2/5}\phi_R^*(\ep^{2/5}(j-\ep^{2/5}c_0t)) + \ep^{3/5}\phi_R^{\ep}(\ep^{2/5}(j-\ep^{2/5}c_0t)).
\]
Then the triple $(A_j,P_j,R_j)$ solves \eqref{eqn: main system}.

%%-----------------------------------------------------------%%
\item
The remainder terms $\phi_A^{\ep}$, $\phi_P^{\ep}$, and $\phi_R^{\ep}$ satisfy
\[
\sup_{0 < \ep < \ep_{\star}}
\norm{\phi_A^{\ep}}_{H_q^1} 
+ \norm{\phi_P^{\ep}}_{W^{1,\infty}}
+ \norm{\phi_R^{\ep}}_{W^{1,\infty}}
< \infty.
\]

%%-----------------------------------------------------------%%
\item
The functions $\phi_P^{\ep}$ and $\phi_R^{\ep}$ vanish exponentially fast at $+\infty$ and are asymptotically constant at $-\infty$ in the following sense: there exist $\ell_P^{\ep}$, $\ell_R^{\ep} \in \R$ such that 
\[
\sup_{0 < \ep < \ep_{\star}} \left(|\ell_P^{\ep}| + \sup_{X \ge 0} e^{qX}|\phi_P^{\ep}(X)| 
+ \sup_{X \le 0} e^{-qX}|\phi_P^{\ep}(X)-\ell_P^{\ep}|\right)
< \infty
\]
and
\[
\sup_{0 < \ep < \ep_{\star}} \left(|\ell_R^{\ep}| + \sup_{X \ge 0} e^{qX}|\phi_R^{\ep}(X)| 
+ \sup_{X \le 0} e^{-qX}|\phi_R^{\ep}(X)-\ell_R^{\ep}|\right)
< \infty.
\]
\end{enumerate}
\end{corollary}

In order to achieve the normalization $\norm{\phi_A^*}_{L^{\infty}} = 1$, as discussed in Section \ref{sec:sub:int:mr}, we need to use the explicit choice
\begin{equation}\label{eqn: cstar}
c_0
= \left(\frac{9\alpha\kappa\tau_1\tau_2^2}{8}\right)^{1/5}
=: c_*,
\end{equation}
as used in \eqref{eq:int:mr:a:p:r:profiles:scales}.
Furthemore, we then get 
\begin{equation}\label{eqn: P R Linfty}
\norm{\phi_P^*}_{L^{\infty}}
= \left(\frac{6\alpha}{\kappa\tau_1}\right)^{1/2}\left(\frac{9\alpha\kappa\tau_1\tau_2^2}{8}\right)^{1/10}
\quadword{and}
\norm{\phi_R^*}_{L^{\infty}}
= \frac{3}{\tau_1}\left(\frac{9\alpha\kappa\tau_1\tau_2^2}{8}\right)^{1/5}.
\end{equation}
Substituting the abbreviations \eqref{eq:twv:def:tau:1:2} and \eqref{eq:twv:def:kappa} into the quantities in \eqref{eqn: cstar} and \eqref{eqn: P R Linfty} then leads to the identities \eqref{eq:int:def:expl:constants}.

%%-----------------------------------------------------------%%
%%-----------------------------------------------------------%%
%%-----------------------------------------------------------%%
%%-----------------------------------------------------------%%
%%-----------------------------------------------------------%%
\appendix

%%-----------------------------------------------------------%%
%%-----------------------------------------------------------%%
%%-----------------------------------------------------------%%
%%-----------------------------------------------------------%%
%%-----------------------------------------------------------%%
\section{Fourier Analysis}\label{app: fourier}

%%-----------------------------------------------------------%%
%%-----------------------------------------------------------%%
%%-----------------------------------------------------------%%
%%-----------------------------------------------------------%%
\subsection{The Fourier transform}
We use the following conventions for Fourier transforms.
If $f \in L^1$, then its Fourier transform is
\[
\ft[f](k)
= \hat{f}(k)
:= \frac{1}{\sqrt{2\pi}}\int_{-\infty}^{\infty} f(x)e^{-ikx} \dx,
\]
and its inverse Fourier transform is
\[
\ft^{-1}[f](x)
= \check{f}(x)
:= \frac{1}{\sqrt{2\pi}}\int_{-\infty}^{\infty} f(k)e^{ikx} \dk.
\]

%%-----------------------------------------------------------%%
%%-----------------------------------------------------------%%
%%-----------------------------------------------------------%%
%%-----------------------------------------------------------%%
\subsection{Fourier multipliers on Sobolev spaces}\label{app: fm}
For integers $r \ge 0$, we denote by $H^r = H^r(\R)$ the usual Sobolev space of all $r$-times weakly differentiable functions whose weak derivatives are square-integrable.

Our fundamental operator on Sobolev spaces is the Fourier multiplier.
The following result above is standard; see, e.g., \cite[Lem.\@ D.2.1]{faver-dissertation}.

%%-----------------------------------------------------------%%
%%-----------------------------------------------------------%%
\begin{lemma}\label{lem: fm}
Let $\tM \colon \R \to \C$ be measurable and suppose
\[
N_{\tM}(r,s)
:= \sup_{k \in \R} \frac{|\tM(k)|}{(1+k^2)^{(r-s)/2}}
< \infty.
\]
Then the {\bf{Fourier multiplier $\M$ with symbol $\tM$}} defined by
\begin{equation}\label{eqn: fm defn}
\M{f}
:= \ft^{-1}\big[\tM\hat{f}\big],
\end{equation}
i.e., by $\hat{\M{f}}(k) = \tM(k)\hat{f}(k)$, is a bounded operator from $H^r$ to $H^s$, and 
\begin{equation}\label{eqn: fm norm eq}
\norm{\M}_{\b(H^r,H^s)}
= N_{\tM}(r,s).
\end{equation}
\end{lemma}

We also need a convenient expression for `scaled' Fourier multipliers.
If $f$ is a function on $\R$ and $\nu \in \R\setminus\{0\}$, let $f(\nu\cdot)$ be the `scaled' map $x \mapsto f(\nu{x})$.
Now let $\M$ be the Fourier multiplier with symbol $\tM$ and define $\tM^{(\nu)}(k) := \tM(\nu{k})$.
Let $\M^{(\nu)}$ be the Fourier multiplier with symbol $\tM^{(\nu)}$. 
Then standard scaling properties of the Fourier transform imply that 
\begin{equation}\label{eqn: scaled fm}
\M[f(\nu\cdot)](x)
= (\M^{(\nu)}f)(\nu{x}).
\end{equation}

%%-----------------------------------------------------------%%
%%-----------------------------------------------------------%%
%%-----------------------------------------------------------%%
%%-----------------------------------------------------------%%
\subsection{Fourier multipliers on weighted Sobolev spaces}\label{app: expn loc Sob space}
We frequently work with weighted Sobolev spaces.
For $q \in \R$, let
\[
H_q^r
:= \set{f \in H_q^r}{e^{q|\cdot|}f \in H^r}.
\]
We norm this space by
\[
\norm{f}_{H_q^r}
:= \norm{e^{q|\cdot|}f}_{H^r},
\]
and, see \cite[App.@ C]{faver-dissertation}, this norm is equivalent to
\[
f \mapsto \sum_{j=0}^r \norm{e^{q|\cdot|}\partial_x^j[f]}_{L^2}.
\]
We put $L_q^2 := H_q^0$.
The Cauchy-Schwarz inequality guarantees that $L_q^2$ embeds into $L^1$:
\[
\norm{f}_{L^1}
= \norm{e^{-q|\cdot|}(e^{q|\cdot|}f)}_{L^1}
\le \norm{e^{-|q\cdot|}}_{L^2}\norm{e^{q|\cdot|}f}_{L^2}
\le C_q\norm{f}_{L_q^2}.
\]
Finally, if $I \subseteq \R$ is an interval, we sometimes denote by $L_q^2(I)$ the set of all measurable functions $f \colon I \to \C$ such that 
\[
\int_I e^{2q|X|}|f(X)|^2 \dX
< \infty.
\]

Since $H_q^r \subseteq H^r$, any Fourier multiplier defined on $H^r$ is defined on $H_q^r$.
A variation on a result of Beale \cite[Lem.\@ 5.1]{beale1} gives sufficient conditions for a Fourier multiplier on $H_q^r$ to map into another weighted space $H_q^s$.

%%-----------------------------------------------------------%%
%%-----------------------------------------------------------%%
\begin{lemma}[Beale]\label{lem: Beale fm lemma}
Fix $q > 0$ and let 
\[
\Ubar_q
:= \set{z \in \C}{|\im(z)| \le q}.
\]
Let $\tM$ be analytic on $\Ubar_q$.
Suppose there  exist $s \ge 0$ and $C$, $r_0 > 0$ such that if $z \in \Ubar_q$ and $r_0 < |z|$, then
\[
|\tM(z)|
\le \frac{C}{|\re(z)|^{(s-r)}}.
\]
Then the Fourier multiplier $\M$ with symbol $\tM$, defined by \eqref{eqn: fm defn}, is a bounded operator from $H_q^r$ to $H_q^s$ and 
\[
\norm{\M}_{\b(H_q^r,H_q^s)}
\le \sup_{k \in \R} \big|(1+k^2)^{(s-r)/2}\tM(k\pm{i}q)\big|.
\]
\end{lemma}

%%-----------------------------------------------------------%%
%%-----------------------------------------------------------%%
%%-----------------------------------------------------------%%
%%-----------------------------------------------------------%%
%%-----------------------------------------------------------%%
\section{The Proof of Proposition \ref{prop: varpi-nu conv}}\label{app: proof of prop varpi-nu conv}

We prove the following lemma in this appendix.

%%-----------------------------------------------------------%%
%%-----------------------------------------------------------%%
\begin{lemma}
Let $q \in (0,c_0/\tau_2)$.
There exist $C_{\M}$, $\nu_{\M} > 0$ such that if $0 < \nu < \nu_{\M}$, then the map $\tM^{(\nu)}$ defined in \eqref{eqn: M-nu} is analytic on the strip $\Ubar_q := \set{z \in \mathbb{C}}{|\im(z)| \le q}$ and satisfies
\begin{equation}\label{eqn: varpi symbol conv}
\sup_{k \in \R} \big|\tM^{(\nu)}(k \pm iq) - \tM^{(0)}(k\pm{iq})\big|
< C_{\M}\nu^{1/3},
\end{equation}
where $\tM^{(0)}$ was defined in \eqref{eqn: tvarpi-0}.
\end{lemma}

This lemma allows us to invoke Beale's result in Lemma \ref{lem: Beale fm lemma} (with $r=1$ and $s=0$) to prove Proposition \ref{prop: varpi-nu conv}.
We will estimate the difference $\big|\tM^{(\nu)}(z)-\tM^{(0)}(z)\big|$ over two regimes, one in which $z = k\pm{iq}$ is `close' to 0, and the other in which $z$ is `far from' 0.
Part of these estimates will involve bounding the denominator of $\tM^{(\nu)}$ away from zero; this will ensure the analyticity of $\tM^{(\nu)}$, since it is the quotient of two analytic functions.

To quantify these regimes, we introduce two positive constants $p$ and $m$; we say that $z$ is `close' to 0 if $|z| \le \nu^{-p}$ and `far from' 0 if $|z| > \nu^{-p}$.
The constant $m$ will later control how close the real part of $\nu{z}$ is to an integer multiple of $2\pi$, a bound that will be very useful in certain estimates to come.
All constants $C$ in the work below are allowed to depend on $m$, $p$, and $q$, but they are always independent of $\nu$ and $z$.

Our estimates will depend on the parameters $p$ and $m$; once we have all the estimates together, we will choose useful values for $p$ and $m$.
We feel that this approach allows the otherwise nonobvious final values for $p$ and $m$ to emerge very naturally.
This strategy of splitting the estimates over regions close to and far from 0 is modeled on the proofs of \cite[Lem.\@ A.13]{faver-wright} and \cite[Lem.\@ 3]{stefanov-wright} and the strategy in \cite[App.\@ A.3]{johnson-wright}.
Friesecke and Pego \cite[Sec.\@ 3]{friesecke-pego1} give a rather different proof of symbol convergence that relies on more knowledge of the poles of $\tM^{(\nu)}$ than we care to discover.

%%-----------------------------------------------------------%%
%%-----------------------------------------------------------%%
%%-----------------------------------------------------------%%
%%-----------------------------------------------------------%%
\subsection{Estimates for $z$ `close to' 0} 
In this regime we fix $|z| \le \nu^{-p}$.
We recall the Taylor expansions 
\[
1-e^{-iz} = iz+iz^2N_1(z)
\quadword{and}
\cos(z)-1 = -\frac{z^2}{2}+ \frac{iz^4N_2(z)}{2\tau_2}
\]
from \eqref{eqn: N}, as well as the estimate 
\[
C_q
:= \sup_{x \in \R} |N_1(x\pm{iq})| + |N_2(x\pm{iq})|
< \infty.
\]

Now we can write
\[
\tM^{(\nu)}(z)
= \frac{1+\nu{z}N_1(\nu{z})}{c_0+\tau_2iz+ \nu^2z^3N_2(\nu{z})}.
\]
With this expression we find the following equality
\[
\tM^{(\nu)}(z) - \tM^0(z)
= I_{\nu}(z) + \II_{\nu}(z),
\]
where
\begin{equation}\label{eqn: I-nu-z}
I_{\nu}(z) 
:= \frac{c_0\nu{z}N_1(\nu{z}) - \nu^2z^3N_2(\nu{z})}{(c_0 + \tau_2 i z + \nu^2 z^3 N_2(\nu z))(c_0 +i \tau_2 z)}
\end{equation}
and
\begin{equation}\label{eqn: II-nu-z}
\II_{\nu}(z) 
:= \frac{i\tau_2\nu{z}^2N_1(\nu{z})}{(c_0 + \tau_2 i z + \nu^2 z^3 N_2(\nu z))(c_0 +i \tau_2 z)}.
\end{equation}

We work on the denominators.
We use the reverse triangle inequality to find
\[
\big|\big(c_0+\tau_2iz- 2\tau_2 i \nu^2z^3N_2(\nu{z})\big)(c_0+\tau_2iz)\big|
\ge \big|c_0-\tau_2q-2 \tau_2\nu^2|z|^3|N_2(\nu{z})|\big||c_0-\tau_2q|.
\]
As $q \in (0, c_0/\tau_2)$, we have $|c_0-\tau_2 q| > 0$. 
Also, since $|z| \le \nu^{-p}$, we have $\nu^2 |z|^3 \le \nu^{2-3p}$. 
If we take 
\begin{equation}\label{eqn: p1}
0 
< p
< \frac{2}{3}
\end{equation}
and assume $\nu \in (0,\nu_1)$, where
\begin{equation}\label{eqn: ep1}
\nu_1
:= \min\left\{1,\left(\frac{|c_0-\tau_2q|}{4C_q\tau_2}\right)^{1/(2-3p)}\right\},
\end{equation}
then
\begin{equation}\label{eqn: partial denom}
\big|c_0-\tau_2q-\nu^2|z|^3|N_2(\nu{z})|\big|
\ge \frac{|c_0-\tau_2q|}{2}.
\end{equation}
In particular,
\begin{equation}\label{eqn: full denom}
\big|c_0-\tau_2q-2 \tau_2\nu^2|z|^3|N_2(\nu{z})|\big||c_0-\tau_2q|
\ge \frac{|c_0-\tau_2q|^2}{2},
\end{equation}
and this inequality guarantees that $\tM^{(\nu)}$ is defined (and analytic) for $|z| \le \nu^{-p}$ and $|\im(z)| < q$.
Then we use \eqref{eqn: full denom} to estimate $I_{\nu}(z)$ from \eqref{eqn: I-nu-z} as
\[
|I_{\nu}(z)|
\le C\nu^{1-p} + C\nu^{2-3p}.
\]

Next, we use \eqref{eqn: partial denom} to estimate $\II_{\nu}(z)$ from \eqref{eqn: II-nu-z} as
\[
|\II_{\nu}(z)|
\le C\nu^{1-p}\frac{|z|}{|c_0+i\tau_2z|}.
\]
Setting $z=x \pm iq$ we note
\[
\frac{|z|^2}{|c_0+i\tau_2z|^2}
= \frac{x^2+q^2}{(c_0\pm\tau_2q)^2 + \tau_2^2x^2}
\le \frac{x^2+q^2}{(c_0-\tau_2q)^2 + \tau_2^2x^2}.
\]
We know
\[
D:=
\sup_{x \in \R} \frac{x^2+q^2}{(c_0-\tau_2q)^2 + c_0^2x^2}  
< \infty,
\]
and thus
\[
|\II_{\nu}(z)|
\le C\nu^{1-p}.
\]

We conclude
\begin{equation}\label{eqn: z small}
\big|\tM^{(\nu)}(z) - \tM^{(0)}(z)\big|
\le |I_{\nu}(z)| + |\II_{\nu}(z)|
\le C\big(\nu^{1-p} + \nu^{2-3p}\big).
\end{equation}
As we required $p \in (0,2/3)$, the final estimate contains only positive powers of $\nu$.
Since we will always consider $0 < \nu < \nu_1$ in the future, the definition of $\nu_1$ in \eqref{eqn: ep1} ensures $0 < \nu < 1$ in the following regimes.

%%-----------------------------------------------------------%%
%%-----------------------------------------------------------%%
%%-----------------------------------------------------------%%
%%-----------------------------------------------------------%%
\subsection{Estimates for $z$ `far from' 0} 
In this regime we assume $|z| > \nu^p$.
Take
\begin{equation}\label{eqn: ep2}
\nu_2
< \min\left\{\nu_1,\left(\frac{\tau_2}{2c_0}\right)^{1/p}\right\},
\end{equation}
with $\nu_1$ defined in \eqref{eqn: ep1}, so that if $0 < \nu < \nu_2$, then $|z| > c_0/\tau_2$.
With the reverse triangle inequality we find
\begin{equation}\label{eqn: varpi0 z large}
|\tM^{(0)}(z)|  
\le \frac{1}{\big||c_0| - |\tau_2 z|\big|} 
< \frac{1}{\tau_2 \nu^{-p}-c_0} 
< \frac{2}{\tau_2} \nu^{p}.
\end{equation}
Consequently, it suffices in this regime to show that $\tM^{(\nu)}$ is bounded by a multiple of some power of $\nu$.
It will be convenient now to rewrite $\tM^{(\nu)}$ as
\[
\tM^{(\nu)}(z)
= \frac{\nu\tM_1^{(\nu)}(z)}{\tM_2^{(\nu)}(z)},
\]
where
\begin{equation}\label{eqn: varpi1 varpi2}
\tM_1^{(\nu)}(z) := 1-e^{-i\nu{z}}
\quadword{and}
\tM_2^{(\nu)}(z) := ic_0\nu^2z + 2\tau_2(\cos(\nu{z})-1).
\end{equation}
The analyticity of $\tM^{(\nu)}$ for $|z| > \nu^p$ will follow if we bound $\tM_2^{(\nu)}$ away from zero here.

The presence of the factor $\cos(\nu{z})-1$ in the denominator of $\tM^{(\nu)}$ suggests that the behavior of this function may be different when $\re(\nu{z})$ is `close' to an integer multiple of $2\pi$ and when it is not.
For this reason, we expand $z=x \pm iq$ and let $n \in \Z$ be the unique integer such that $|\nu{x}-2\pi{n}| \le \pi$.
We consider three cases on the behavior of $\nu{x}$ and $n$.

%%-----------------------------------------------------------%%
%%-----------------------------------------------------------%%
%%-----------------------------------------------------------%%
\subsubsection{Estimates for $\re(\nu{z})$ `close to' a nonzero integer multiple of $2\pi$} 
In this regime we assume $|\nu{x}-2\pi{n}| \le \nu^m$ with $n \ne 0$.

We first rewrite the numerator as
\[
\tM_1^{(\nu)}(z)
= 1-e^{-i(\nu{x}-2\pi{n)}}+e^{-i\nu{x}}(1-e^{\pm\nu{q}}).
\]
Since the map $y \mapsto e^{-iy}$ is uniformly Lipschitz on $\R$ we have
\[
|1-e^{-i(\nu{x}-2\pi{n})}|
\le |\nu{x}-2\pi{n}|
\le \nu^m.
\]
Since the map $y \mapsto e^{-y}$ is locally Lipschitz on $\R$ we have, if we take $0 < \nu < \nu_3$ with
\begin{equation}\label{eqn: ep3}
\nu_3 
:= \min\left\{\nu_2,\frac{1}{q}\right\}
\end{equation}
and $\nu_2$ defined in \eqref{eqn: ep2},  the estimate
\[
|1-e^{\pm\nu{q}}|
\le \nu{q}.
\]
Then
\begin{equation}\label{eqn: num close to 0}
\big|\tM_1^{(\nu)}(z)\big|
\le \nu^m + \nu{q}
\le C(\nu^m+\nu).
\end{equation}
We remark that we did not need $n \ne 0$ here, although we will momentarily.

We now turn to the denominator, $\M_2^{(\nu)}(z)$. 
Using the identity
\begin{equation}\label{eqn: cos cosh sinh}
\cos(a+bi) = \cos(a)\cosh(b)-i\sin(a)\sinh(b)
\end{equation}
for $a,b \in \R$ we find
\[
\im\big(\tM_1^{(\nu)}(z)\big) 
= c_0 \nu^2 x - 2 \tau_2 \sin(\nu x) \sinh(\nu q).
\]
We estimate
\[
\big|\im\big(\tM_1^{(\nu)}(z)\big) \big|
\ge C\big(\nu|n| - \nu|\nu{x}-2n\pi| - |\sin(\nu{x}-2n\pi)||\sinh(\nu{q})|\big)
\]

We control the three terms on the right as follows.
First, $|n| \ge 1$.
Next, we are in the regime $|\nu{x}-2n\pi| \le \nu^m$.
Finally, we have 
\[
|\sin(\nu{x}-2n\pi)| \le |\nu{x}-2n\pi| \le \nu^m
\quadword{and}
|\sinh(\nu{q})| \le 2|\nu{q}|,
\]
since $|\nu{q}| \le 1$. We thus find
\[
\big|\im\big(\tM_2^{(\nu)}(z)\big)\big|
\ge C\big(\nu-\nu^{m+1}).
\]
Now that we have the numerator and the denominator bounded, we can conclude 
\begin{equation}\label{eqn: varpi-nu z large 1}
\big|\tM^{(\nu)}(z)\big|
\le C\nu\frac{\nu^m+\nu}{\nu-\nu^{m+1}}
= C\frac{\nu^m+\nu}{1-\nu^m}
\le C\nu^m.
\end{equation}
Here we need to assume
\begin{equation}\label{eqn: m1}
0
< m
< 1.
\end{equation}

%%-----------------------------------------------------------%%
%%-----------------------------------------------------------%%
%%-----------------------------------------------------------%%
\subsubsection{Estimates for $\re(\nu{z})$ `close to' 0} 
In this regime we assume $|\nu{x}| \le \nu^m$; in particular, we are taking $n = 0$.
We will need the following bound on the cosine, which is a consequence of an elementary argument with Taylor's theorem.

%%-----------------------------------------------------------%%
%%-----------------------------------------------------------%%
\begin{lemma}\label{lem: cos}
Let $Q \ge 0$.
There exist $C_{1,Q}$, $C_{2,Q} > 0$ such that if $Z \in \C$ with $|Z| \le C_{1,Q}$ and $|\im(Z)| \le Q$, then
\[
|\cos(Z)-1|
\ge C_{2,Q}|Z|^2.
\]
In particular, if $Q = 0$, then $C_{1,0} > \pi$.
\end{lemma}

We use the reverse triangle inequality on $\tM_2^{(\nu)}$ from \eqref{eqn: varpi1 varpi2} to find
\begin{equation}\label{eqn: tvarpi2-nu lower}
\big|\tM_2^{(\nu)}(z)\big| 
\ge 2\tau_2|\cos(\nu{z})-1| - c_0\nu^2q - c_0\nu^2|x|.
\end{equation}
Take $0 < \nu < \nu_{\M}$, where
\begin{equation}\label{eqn: ep-varpi}
\nu_{\M}
<  \min\left\{\nu_3,\left(\frac{1}{q}\right)^{1/(1-m)}, \left(\frac{C_{1,q}}{2}\right)^{1/m}\right\},
\end{equation}
with $\nu_3$ defined in \eqref{eqn: ep3}, to find
\begin{align*}
|\nu z | 
\le \nu^{m} + \nu{z} 
\le 2 \nu^{m} 
< C_{1,q}.
\end{align*}
Lemma \ref{lem: cos} then guarantees
\[
|\cos(\nu{z})-1|
\ge C|\nu{z}|^2
\ge C\nu^{2-2p}.
\]
Finally, since $|x| \le \nu^{m-1}$ in this regime we use the bound \eqref{eqn: tvarpi2-nu lower} to conclude
\[
\big|\tM_2^{(\nu)}(z)\big| 
\ge C\big(\nu^{2-2p} - \nu^2 - \nu^{m+1}\big).
\]

We remark that the derivation of the estimate \eqref{eqn: num close to 0} only assumed $|\nu{x}-2\pi{n}| \le \nu^m$ and did not rely on having $n \ne 0$.
So it is still valid here, and we conclude
\begin{equation}\label{eqn: varpi-nu z large 2}
\big|\tM^{(\nu)}(z)\big|
\le C\nu\frac{\nu^m+\nu}{\nu^{2-2p}-\nu^2-\nu^{m+1}}
= C\frac{\nu^{m+2p-1}+\nu^{2p}}{1-\nu^{2p}-\nu^{m+2p-1}}
\le C\big(\nu^{m+2p-1}+\nu^{2p}\big).
\end{equation}
Here we are assuming 
\begin{equation}\label{eqn: p2 m2}
0
\le 1-2p 
< \min\{1,m\}.
\end{equation}

%%-----------------------------------------------------------%%
%%-----------------------------------------------------------%%
%%-----------------------------------------------------------%%
\subsubsection{Estimates for $\re(\nu{z})$ `far from' a nonzero integer multiple of $2\pi$} 
In this regime we assume $|\nu{x}-2\pi{n}| > \nu^m$.
We do not perform separate work on $n = 0$ and $n \ne 0$.

Via \eqref{eqn: cos cosh sinh} we find
\[
\re\big(\M_2^{(\nu)}(z)\big)
= -c_0\nu^2q + 2\tau_2(\cos(\nu{x}-2n\pi)-1) + 2\tau_2\cos(\nu{x})(\cosh(\nu{q})-1).
\]
We estimate
\[
\big|\re\big(\M_2^{(\nu)}(z)\big)\big|
\ge C\big(|\cos(\nu{x}-2n\pi)-1|-|\cos(\nu{x})||\cosh(\nu{q})-1|-\nu^2\big).
\]
Now we use Lemma \ref{lem: cos} with $Q = 0$ to bound
\[
|\cos(\nu{x}-2n\pi)-1|
\ge C|\nu{z}-2n\pi|^2
\ge C\nu^{2m}.
\]
Also, a routine Lipschitz estimate on the hyperbolic cosine gives
\[
|\cos(\nu{x})||\cosh(\nu{q})-1|
\le C\nu^2
\]
since $|\nu{q}| \le 1$.
We thus find
\[
\big|\re\big(\M_2^{(\nu)}(z)\big)\big|
\ge C(\nu^{2m}-\nu^2).
\]
As we are assuming $0 < m < 1$ from \eqref{eqn: m1}, this is a positive lower bound.

Finally, we bound the numerator $\tM_1^{\nu}(z)$ crudely as $\big|\tM_1^{(\nu)}(z)\big| \le C$ for all $z \in \C$ with $|\im(z)| = q$.
This follows from the boundedness of $Z \mapsto e^{iZ}$ on strips.
We conclude
\begin{equation}\label{eqn: varpi-nu z large 3}
\big|\tM^{(\nu)}(z)\big|
< C\frac{\nu}{\nu^{2m}-\nu^2}
\le C\nu^{1-2m}.
\end{equation}
This is a positive bound if we now require
\begin{equation}\label{eqn: m3}
0
< m
< \frac{1}{2}.
\end{equation}

%%-----------------------------------------------------------%%
%%-----------------------------------------------------------%%
%%-----------------------------------------------------------%%
%%-----------------------------------------------------------%%
\subsection{Overall estimates}
Suppose $0 < \nu < \nu_{\M}$, where $\nu_{\M}$ was specified in \eqref{eqn: ep-varpi}.
We conclude from \eqref{eqn: z small} that
\begin{equation}\label{eqn: final z small}
\sup_{\substack{|\im(z)| = q \\ |z| \le \nu^{-p}}} \big|\tM^{(\nu)}(z) - \tM^{(0)}(z)\big|
\le C\big(\nu^{1-p} + \nu^{2-3p}\big)
\end{equation}
and, by combining \eqref{eqn: varpi0 z large}, \eqref{eqn: varpi-nu z large 1}, \eqref{eqn: varpi-nu z large 2}, and \eqref{eqn: varpi-nu z large 3}, that
\begin{equation}\label{eqn: final z large}
\sup_{\substack{|\im(z)| = q \\ |z| > \nu^{-p}}} \big|\tM^{(\nu)}(z) - \tM^{(0)}(z)\big|
\le C\nu^p + C\max\big\{\nu^m, \nu^{m+2p-1}+\nu^{2p}, \nu^{1-2m}\big\}.
\end{equation}
Additionally, we need, per \eqref{eqn: p1}, and \eqref{eqn: p2 m2}, and \eqref{eqn: m3}, the exponents $p$ and $m$ to satisfy
\begin{equation}\label{eqn: p m constraints}
0 < p < \frac{2}{3},
\qquad
0 < m < \frac{1}{2},
\quadword{and}
0 < 1-2p < \min\{1,m\}.
\end{equation}

There are many possible choices of $p$ and $m$ that will satisfy \eqref{eqn: p m constraints}.
Purely for convenience, we elect to take $m = 1/3$ and then $p = 1/2$.
We combine \eqref{eqn: final z small} and \eqref{eqn: final z large} to conclude the estimate \eqref{eqn: varpi symbol conv}.

%%-----------------------------------------------------------%%
%%-----------------------------------------------------------%%
%%-----------------------------------------------------------%%
%%-----------------------------------------------------------%%
%%-----------------------------------------------------------%%
\section{The Proof of Proposition \ref{prop: T invert}}\label{app: proof of prop T invert}

Throughout this appendix we assume $q \in (0,c_0/\tau_2)$.

%%-----------------------------------------------------------%%
%%-----------------------------------------------------------%%
%%-----------------------------------------------------------%%
%%-----------------------------------------------------------%%
\subsection{A formula for $\T^{-1}$}
Fix $g \in H_q^1$.
We want to find $f \in H_q^1$ such that $\T{f} = g$, and we want to do so in such a way that the mapping of $g$ to this $f$ is a bounded operator on $H_q^1$.
We recall that $\T$ was defined in \eqref{eqn: T}.
The following steps are quite similar to the derivation of the Bernoulli solution $\sigma$ in Section \ref{sec: formal leading order1}.

We have $\T{f} = g$ if and only if 
\begin{equation}\label{eqn: Tf=g1}
(c_0+\tau_2\partial_X)f - \Rcal_1^0(\sigma)f + \big(D\Rcal_1^0(\sigma)f\big)\sigma
= g ,
\end{equation}
where $\Rcal_1^0$ was defined in \eqref{eqn: Rcal1-0}.
From this definition, we find
\begin{equation}\label{eqn: DRcal0(sigma)}
\big(D\Rcal^0(\sigma)f\big)(X)
= \int_X^{\infty} \left(\int_W^{\infty} \sigma(V) \dV\right) f(W) \dW
+ \int_X^{\infty} \left(\int_W^{\infty} f(V) \dV\right) \sigma(W) \dW.
\end{equation}
Since $g$, $\sigma \in H_q^1$, and since we seek $f \in H_q^1$, we may define the antiderivatives
\begin{equation}\label{eqn: F G Sigma}
F(X) := \int_X^{\infty} f(W) \dW,
\qquad
G(X) := \int_X^{\infty} g(W) \dW
\quadword{and}
\Sigma(X) := \int_X^{\infty} \sigma(W) \dW.
\end{equation}
Then \eqref{eqn: Tf=g1} is equivalent to
\begin{multline}\label{eqn: Tf = g all caps}
\tau_2F''(X) -c_0F'(X) 
- \left(\frac{\alpha\kappa\tau_1}{c_0^2}\right)F'(X)\int_X^{\infty} \Sigma(W)\Sigma'(W) \dW \\
- \left(\frac{\alpha\kappa\tau_1}{c_0^2}\right)\Sigma'(X)\int_X^{\infty} \big(\Sigma(W)F'(W) + F(W)\Sigma'(W)\big) \dW
= -c_0G'(X) -\tau_2G''(X).
\end{multline}

Although it may not be apparent at first glance, every term in this equation is a perfect derivative.
First, since $\Sigma$ and $F$ must vanish at $+\infty$, we have
\[
\int_X^{\infty} \Sigma(W)\Sigma'(W) \dW
= -\frac{\Sigma(X)^2}{2}
\]
and
\[
\int_X^{\infty} \big(\Sigma(W)F'(W) + F(W)\Sigma'(W)\big) \dW
= -\Sigma(X)F(X).
\]
Hence \eqref{eqn: Tf = g all caps} really is
\begin{equation}\label{eqn: Tf = g all caps2}
\tau_2F''-c_0F'
+\left(\frac{\alpha\kappa\tau_1}{c_0^2}\right)\left(\frac{F'\Sigma^2}{2}
+ \Sigma'\Sigma{F}\right)
= -c_0G' - \tau_2G'',
\end{equation}
where 
\[
\frac{F'\Sigma^2}{2}
+ \Sigma'\Sigma{F}
= \frac{1}{2}\partial_X[\Sigma^2F].
\]
So, we deduce that $F$ and $G$ must satisfy
\begin{equation}\label{eqn: Tf = g all caps3}
- \tau_2F'' -c_0F' 
+ \left(\frac{\alpha\kappa\tau_1}{c_0^2}\right)\partial_X\left[\frac{\Sigma^2F}{2}\right]
= -c_0G' - \tau_2G''.
\end{equation}

Since both $F$ and $G$ must vanish at $+\infty$ (though not necessarily at $-\infty$), we may integrate \eqref{eqn: Tf = g all caps3} to find
\begin{equation}\label{eqn: linearized Bernoulli}
\tau_2F' + c_0F 
- \left(\frac{\alpha\kappa\tau_1}{c_0^2}\right)\frac{\Sigma^2F}{2}
= c_0G + \tau_2G'.
\end{equation}
This is, of course, the linearization of the Bernoulli equation \eqref{eqn: Bernoulli} at its solution $\Sigma$.

Moreover, \eqref{eqn: linearized Bernoulli} is really just a linear first-order ordinary differential equation, which we know how to solve with an integrating factor.
Put
\begin{equation}\label{eqn: rho}
\rho(X) 
:= \frac{1}{\tau_2}\left(c_0-\left(\frac{\alpha\kappa\tau_1}{c_0^2}\right)\frac{\Sigma(X)^2}{2}\right),
\end{equation}
so that \eqref{eqn: linearized Bernoulli} is equivalent to
\[
F' + \rho{F}
= rG + G', 
\qquad r:= \frac{c_0}{\tau_2}.
\]

Define an antiderivative $\Rho$ of $\rho$ by
\begin{equation}\label{eqn: Rho}
\Rho(X)
:= \ln\left(\frac{[\alpha\kappa\tau_1+6c_0^2e^{2rX}]^{3/2}}{e^{2rX}}\right).
\end{equation}
Then a solution to \eqref{eqn: linearized Bernoulli} is
\begin{equation}\label{eqn: F soln lin B}
F(X)
:= e^{-\Rho(X)}\int_0^X e^{\Rho(W)}\left(rG(W) + G'(W)\right) \dW.
\end{equation}
This is, of course, not the most general solution to \eqref{eqn: linearized Bernoulli}, as we have set the constant of integration that arises from the usual integrating factor method equal to 0.
But we know that if $F$ is defined by \eqref{eqn: F soln lin B}, then we can put $f := -F'$ to find that $f$ solves \eqref{eqn: Tf=g1}.
After all, we just need to develop {\it{one}} solution to this problem.
It remains for us to check that we really do have $-F' \in H_q^1$, to which we now turn.

We differentiate \eqref{eqn: F soln lin B} to find
\[
-F'(X)
= \rho(X)e^{-\Rho(X)}\int_0^X e^{\Rho(W)}\left(rG(W) + G'(W)\right) \dW
- \big(rG(X) + G'(X)\big).
\]
We recall the definition of $G$ from \eqref{eqn: F G Sigma} to put
\begin{equation}\label{eqn: Scal}
(\Scal{g})(X)
:= \rho(X)e^{-\Rho(X)}\int_0^X e^{\Rho(W)}\left(r\int_W^{\infty} g(V) \dV -g(W)\right) \dW
-\left( r\int_X^{\infty} g(W) \dW - g(X)\right).
\end{equation}
Repackaging our work above, if we know that $\Scal{g} \in H_q^1$ when $g \in H_q^1$, then we will have $\T(\Scal{g}) = g$.
This is indeed the case, and in the next section we will prove the following formal encapsulation of this result.

%%-----------------------------------------------------------%%
%%-----------------------------------------------------------%%
\begin{proposition}\label{prop: Scal}
For $g \in H_q^1$, define $\Scal{g}$ by \eqref{eqn: Scal}.
Then $\Scal \in \b(H_q^1)$.
\end{proposition}

%%-----------------------------------------------------------%%
%%-----------------------------------------------------------%%
%%-----------------------------------------------------------%%
%%-----------------------------------------------------------%%
\subsection{The proof of Proposition \ref{prop: Scal}}
For $g \in H_q^1$ put
\begin{equation}\label{eqn: H}
(\H{g})(X)
:= r\int_X^{\infty} g(W) \dW - g(X),
\end{equation}
so that from \eqref{eqn: Scal} we have
\begin{equation}\label{eqn: Scal with H}
(\Scal{g})(X)
= \rho(X)e^{-\Rho(X)}\int_0^X e^{\Rho(W)}(\H{g})(W) \dW 
- (\H{g})(X)
\end{equation}
and
\begin{equation}\label{eqn: Scal'}
(\Scal{g})'(X)
= \rho'(X)e^{-\Rho(X)}\int_0^X e^{\Rho(W)}(\H{g})(W) \dW
- \rho(X)(\Scal{g})(X)
- (\H{g})'(X).
\end{equation}
We will show that $\Scal{g}$, $(\Scal{g})' \in L_q^2$ and that there exists $C > 0$ such that for any $g \in H_q^1$ we have
\[
\norm{\Scal{g}}_{L_q^2} + \norm{(\Scal{g})'}_{L_q^2}
\le C\norm{g}_{H_q^1}.
\]
From this it will follow that $\Scal$ is a bounded operator on $H_q^1$.

Before proceeding, we record some convenient properties of $\rho$ and $\Rho$ that follow from their formulas in \eqref{eqn: rho} and \eqref{eqn: Rho}.

%%-----------------------------------------------------------%%
%%-----------------------------------------------------------%%
\begin{lemma}\label{lem: rhos}
There exist $A$, $B$, $C_{\rho} > 0$ such that the following hold.

%%-----------------------------------------------------------%%
\begin{enumerate}[label={\bf(\roman*)}, ref={(\roman*)}]

%%-----------------------------------------------------------%%
\item\label{part: rho+}
$|\rho(X) - r| \le C_{\rho}e^{-rX}$ for $X > 0$.

%%-----------------------------------------------------------%%
\item\label{part: rho-}
$|\rho(X) - (-2r)| \le C_{\rho}e^{rX}$ for $X < 0$.

%%-----------------------------------------------------------%%
\item\label{part: rho'}
$|\rho'(X)| \le C_{\rho}e^{-r|X|}$.

%%-----------------------------------------------------------%%
\item\label{part: exp Rho-}
$e^{-\Rho(X)}
= \frac{e^{-rX}}{(Ae^{-2rX}+B)^{3/2}} = \frac{e^{2rX}}{(A+Be^{2rX})^{3/2}}$.

%%-----------------------------------------------------------%%
\item\label{part: exp Rho+}
$e^{\Rho(X)} = e^{rX}(Ae^{-2rX}+B)^{3/2} = e^{-2rX}(A+Be^{2rX})^{3/2}$.
\end{enumerate}

\end{lemma}

Now we estimate with gusto.

%%-----------------------------------------------------------%%
%%-----------------------------------------------------------%%
%%-----------------------------------------------------------%%
\subsubsection{Estimates on $(\Scal{g})'$}
The second and third terms of $(\Scal{g})'$ from \eqref{eqn: Scal'} are easy to control if we have bounds on $\Scal{g}$.
We assume an estimate of the form $\norm{\Scal{g}}_{L_q^2} \le C\norm{g}_{H_q^1}$ with $C$ independent of $g$; we prove this in Section \ref{sec: ests on Scal} below.
Since $\rho \in L^{\infty}$, we obtain at once
\[
\norm{\rho(\Scal{g})}_{L_q^2}
\le C\norm{\Scal{g}}_{L_q^2}
\le C\norm{g}_{H_q^1}.
\]
Next, since $(\H{g})' = -rg - g'$ by \eqref{eqn: H}, we have
\[
\norm{(\H{g})'}_{L_q^2}
\le C\norm{g}_{L_q^2} + C\norm{g'}_{L_q^2}
= C\norm{g}_{H_q^1}.
\]

Estimating the first term in \eqref{eqn: Scal'},
\[
(\I{g})(X)
:= \rho'(X)e^{-\Rho(X)}\int_0^X e^{\Rho(W)}(\H{g})(W) \dW,
\]
requires slightly more work. 
We first return to \eqref{eqn: H} to bound
\begin{equation}\label{eqn: Hg Linfty}
\norm{\H{g}}_{L^{\infty}}
\le C\norm{g}_{L_q^1} + \norm{g}_{L^{\infty}}
\le C\norm{g}_{H_q^1}
\end{equation}
by the embedding of $L_q^2$ into $L^1$, the Sobolev embedding of $H^1$ into $L^{\infty}$, and the embedding of $H_q^1$ into $H^1$.
Thus for any $X > 0$ we have
\[
|(\I{g})(X)|
\le C\norm{g}_{H_q^1}e^{-r|X|}e^{-\Rho(X)}\int_0^X e^{\Rho(W)} \dW,
\]
where we have used part \ref{part: rho'} of Lemma \ref{lem: rhos} to estimate $\rho'$.
Our estimates on $e^{\pm\Rho}$ are slightly different depending on whether $X$ is positive or negative.

First suppose $X > 0$.
Then part \ref{part: exp Rho-} of Lemma \ref{lem: rhos} shows $e^{-\Rho(X)} \le Ce^{-rX}$ and $e^{\Rho(W)} \le Ce^{rW}$ for $0 \le W \le X$.
Hence
\[
|(\I{g})(X)|
\le C\norm{g}_{H_q^1}e^{-2rX}\int_0^X e^{rW} \dW
\le C\norm{g}_{H_q^1}e^{-rX},
\]
and so
\[
|e^{qX}(\I{g})(X)|
\le C\norm{g}_{H_q^1}e^{(q-r)X}, \ X > 0.
\]
The analysis when $X > 0$ is entirely similar, except we use the estimates $e^{-\Rho(X)} \le e^{2rX}$ and $e^{\Rho(W)} \le e^{-2rW}$ for $X \le W \le 0$.
We conclude
\[
|e^{q|X|}(\I{g})(X)|
\le C\norm{g}_{H_q^1}e^{-(r-q)|X|},
\]
which implies $\norm{\I{g}}_{L_q^2} \le C\norm{g}_{H_q^1}$.

%%-----------------------------------------------------------%%
%%-----------------------------------------------------------%%
%%-----------------------------------------------------------%%
\subsubsection{Estimates on $\Scal{g}$}\label{sec: ests on Scal}
First suppose $X > 0$ and, using the definition of $\Scal$ in \eqref{eqn: Scal with H} and the formulas for $e^{\pm\Rho}$ from parts \ref{part: exp Rho-} and \ref{part: exp Rho+} of Lemma \ref{lem: rhos}, write
\[
(\Scal{g})(X)
= \sum_{k=1}^4 (\Scal_k^+g)(X),
\]
where
\[
(\Scal_1^+g)(X)
:= (\H{g})(X) -re^{-rX}\int_0^X e^{rW}(\H{g})(W) \dW,
\]
\[
(\Scal_2^+g)(X)
:= \frac{(r-\rho(X))e^{-rX}}{(Ae^{-2rX}+B)^{3/2}}\int_0^X e^{rW}(Ae^{-2rW}+B)^{3/2}(\H{g})(W) \dW,
\]
\[
(\Scal_3^+g)(X)
:= re^{-rX}\left(\frac{1}{(Ae^{-2rX}+B)^{3/2}}-\frac{1}{B^{3/2}}\right)\int_0^X e^{rW}(Ae^{-2rW}+B)^{3/2}(\H{g})(W)\dW,
\]
and
\[
(\Scal_4^+g)(X)
:= \frac{re^{-rX}}{B^{3/2}}\int_0^X e^{rW}\big[(Ae^{-2rW}+B)^{3/2}-B^{3/2}\big](\H{g})(W) \dW.
\]
We claim
\begin{equation}\label{eqn: Scal+ 2-4 est}
\sum_{k=2}^4 |(\Scal_k^+g)(X)|
\le Ce^{-rX}\norm{g}_{H_q^1},
\end{equation}
from which it follows that 
\begin{equation}\label{eqn: Scal+ 2-4 L2}
\sum_{k=2}^4 \norm{\Scal_k^+g}_{L_q^2((0,\infty))}
\le C\norm{g}_{H_q^1}.
\end{equation}
We achieve \eqref{eqn: Scal+ 2-4 est} using the $L^{\infty}$-estimate \eqref{eqn: Hg Linfty} on $\H{g}$, the estimate on $|\rho(X)-r|$ from part \ref{part: rho+} of Lemma \ref{lem: rhos}, and (local) Lipschitz estimates on the two differences in $\Scal_3^+$ and $\Scal_4^+$.

To control $\Scal_1^+$ we integrate by parts:
\[
\int_0^X e^{rW}(\H{g})(W) \dW
= \frac{e^{rX}(\H{g})(X)-(\H{g})(0)}{r} -\frac{1}{r}\int_0^X e^{rW}(\H{g})'(W) \dW.
\]
By the definition of $\H$ in \eqref{eqn: H} we have
\begin{multline}\label{eqn: nice product rule}
\int_0^X e^{rW}(\H{g})'(W) \dW
= -\int_0^X e^{rW}(rg(W)+g'(W)) \dW
= -\int_0^X \partial_W[e^{rW}g(W)] \dW \\
= g(0)-e^{rX}g(X).
\end{multline}
We conclude
\[
(\Scal_1^+g)(X)
= e^{-rX}(\H{g})(0) + e^{-rX}g(0) - g(X).
\]
Use the $L^{\infty}$-estimate \eqref{eqn: Hg Linfty} on $\H{g}$ and the Sobolev embedding on $g(0)$ to conclude
\begin{equation}\label{eqn: Scal+ 1 est}
|(\Scal_1^+g)(X)| 
\le Ce^{-rX}\norm{g}_{H_q^1},
\end{equation}
and so 
\begin{equation}\label{eqn: Scal+ 1 L2}
\norm{\Scal_1^+g}_{L_q^2((0,\infty))}
\le C\norm{g}_{H_q^1}.
\end{equation}

Our analysis for $X < 0$ starts out similarly.
Rewrite
\[
(\Scal{g})(X)
= \sum_{k=1}^4 (\Scal_k^-g)(X),
\]
where now
\[
(\Scal_1^-g)(X)
:= (\H{g})(X) - 2re^{2rX}\int_X^0 e^{-2rW}(\H{g})(W) \dW,
\]
\[
(\Scal_2^-g)(X)
:= (\rho(X)-(-2r))\frac{e^{2rX}}{(A+Be^{2rX})^{3/2}}\int_X^0 e^{-2rW}(A+Be^{2rW})^{3/2}(\H{g})(W) \dW,
\]
\[
(\Scal_3^-g)(X)
:= 2r\left(\frac{1}{A^{3/2}}-\frac{1}{(A+Be^{2rX})^{3/2}}\right)e^{2rX}\int_X^0 e^{-2rW}(A+Be^{2rW})^{3/2}(\H{g})(W) \dW,
\]
and
\[
(\Scal_4^+g)(X)
:= \frac{2r}{A^{3/2}}e^{2rX}\int_X^0 e^{-2rW}[(A+Be^{2rW})^{3/2}-A^{3/2}](\H{g})(W) \dW.
\]
As before, we obtain
\begin{equation}\label{eqn: Scal- 2-4 est}
\sum_{k=2}^4 |e^{-qX}(\Scal_k^-g)(X)|
\le Ce^{(r-q)X}\norm{g}_{H_q^1},
\end{equation}
and so 
\begin{equation}\label{eqn: Scal- 2-4 L2}
\sum_{k=2}^4 \norm{\Scal_k^-g}_{L_q^2((-\infty,0))}
\le C\norm{g}_{H_q^1}.
\end{equation}

We integrate by parts within $\Scal_1^-g$ to find
\[
\int_X^0 e^{-2rW}(\H{g})(W) \dW
= \frac{e^{-2rX}(\H{g})(X)-(\H{g})(0)}{2r} - \frac{1}{2r}\int_X^0 e^{-2rW}(rg(W)+g'(W)) \dW.
\]
The difference compared to \eqref{eqn: nice product rule} in our treatment of $\Scal_1^+g$ is that we no longer have a perfect derivative as the integrand on the right; this is an artifact of the different asymptotic behavior of $\rho$ and $e^{\Rho}$ at $-\infty$ compared to $+\infty$, as specified in Lemma \ref{lem: rhos}.
And so, at first glance, the best that we have is 
\[
(\Scal_1^-g)(X)
= (\H{g})(0)e^{2rX}+re^{2rX}\int_X^0 e^{-2rW}g(W) \dW + e^{2rX}\int_X^0 e^{-2rW}g'(W) \dW.
\]
It suffices to show, of course, that each of the three terms above is a function in $L_q^2((-\infty,0))$ with norm bounded by a constant multiple of $\norm{g}_{H_q^1}$.
This is easy for the first term, since we can use the familiar $L^{\infty}$-estimate \eqref{eqn: Hg Linfty} on $(\H{g})(0)$.
For the integral terms, we want an estimate of the form
\[
\int_{-\infty}^0 e^{-2qX}\left|e^{2rX}\int_X^0 e^{-2rW}g(W) \dW\right|^2 \dX
\le C\norm{g}_{H_q^1}
\]
and similarly for $g'$.
To obtain these estimates, we use the following lemma, whose proof we defer to Section \ref{sec: proof of int lemma}.

%%-----------------------------------------------------------%%
%%-----------------------------------------------------------%%
\begin{lemma}\label{lem: int lemma}
There exists $C > 0$ such that 
\begin{equation}\label{eqn: funky int}
\int_{-\infty}^0 e^{2X}\left|\int_X^0 e^{-W}h(W) \dW\right|^2 \dX
\le C\norm{h}_{L^2}
\end{equation}
for all $h \in L^2$.
\end{lemma}

We work out the estimate just for the integral term involving $g$.
Since $g \in L_q^2$, we can write $g(X) = e^{2q|X|}\tilde{g}(X)$ for some $\tilde{g} \in L^2$.
Then, changing variables, we find
\begin{multline*}
\int_{-\infty}^0 e^{-2qX}\left|e^{2rX}\int_X^0 e^{-2rW}g(W) \dW\right|^2 \dX
= \int_{-\infty}^0 e^{2(2r-q)X}\left|\int_X^0 e^{-(2r-q)W}\tilde{g}(W) \dW\right|^2 \dX \\
= \frac{1}{(2r-q)^2}\int_{-\infty}^0 e^{2U}\left|\int_U^0 e^{-V}g\left(\frac{V}{2r-q}\right)\dV\right|^2 \dU.
\end{multline*}
Applying Lemma \ref{lem: int lemma}, we obtain
\[
\int_{-\infty}^0 e^{-2qX}\left|e^{2rX}\int_X^0 e^{-2rW}g(W) \dW\right|^2 \dX
\le C\left\|\tilde{g}\left(\frac{\cdot}{2r-q}\right)\right\|_{L^2}
\le C\norm{e^{-q|\cdot|}(e^{q|\cdot|}\tilde{g})}_{L^2}
\le C\norm{g}_{H_q^1.}
\]
After an identical analysis with $g'$, we find
\begin{equation}\label{eqn: Scal- 1 L2}
\norm{\Scal_1^-g}_{L_q^2((-\infty,0))}
\le C\norm{g}_{H_q^1}.
\end{equation}

Combine \eqref{eqn: Scal+ 2-4 L2}, \eqref{eqn: Scal+ 1 L2}, \eqref{eqn: Scal- 2-4 L2}, and \eqref{eqn: Scal- 1 L2} to conclude
\[
\norm{\Scal{g}}_{L_q^2}
\le C\norm{g}_{H_q^1},
\]
as desired.

%%-----------------------------------------------------------%%
%%-----------------------------------------------------------%%
%%-----------------------------------------------------------%%
\subsubsection{The proof of Lemma \ref{lem: int lemma}}\label{sec: proof of int lemma}
Put
\[
\W 
:= \set{(X,W,Y) \in \R^3}{-\infty < X \le 0, \ X \le W \le 0, \ X \le Y \le 0},
\]
so that, after using the triangle inequality, the integral in \eqref{eqn: funky int} is bounded by
\[
\I
:= \int_{-\infty}^0 e^{2X}\left(\int_X^0 e^{-W}|h(W)| \dW\right)^2 \dX
= \iiint_{\W} e^{2X}e^{-W}e^{-Y}|h(W)h(Y)|\dY\dW\dX .
\]
Next, put
\[
\W_1 
:= \set{(X,W,Y) \in \R^3}{-\infty < X \le W, \ W \le Y \le 0, \ -\infty < W \le 0}
\]
and
\[
\W_2
:= \set{(X,W,Y) \in \R^3}{-\infty < X \le Y, \ Y \le W \le 0, \ -\infty < Y \le 0},
\]
so $\W = \W_1 \cup \W_2$ and $\W_1 \cap \W_2$ has measure zero.
Then 
\[
\I
= \I_1 + \I_2,
\]
where
\[
\I_1
:= \iiint_{\W_1} e^{2X}e^{-W}e^{-Y}|h(W)h(Y)|\dY\dW\dX
\]
and
\[
\I_2
:= \iiint_{\W_2} e^{2X}e^{-W}e^{-Y}|h(W)h(Y)|\dY\dW\dX.
\]
Since the integrands are symmetric in $W$ and $Y$, it suffices to show
\[
\I_1
\le C\int_{-\infty}^0 |h(X)|^2 \dX.
\]

Change variables to obtain
\[
\I_1
= \int_{-\infty}^0\int_W^0\left(\int_{-\infty}^W e^{2X} \dX\right) e^{-W}e^{-Y} |h(W)h(Y)| \dY\dW
= \frac{1}{2}\int_{-\infty}^0\int_W^0 e^We^{-Y} |h(W)h(Y)| \dY\dW.
\]
Now we estimate
\begin{equation}\label{eqn: I1 decomp}
4|\I_1|
\le \I_{12} + \I_{13},
\end{equation}
where
\[
\I_{12}
:= \int_{-\infty}^0\int_W^0 e^{W}e^{-Y}|h(W)|^2 \dY\dW 
\quadword{and}
\I_{13}
:= \int_{-\infty}^0\int_W^0 e^{W}e^{-Y}|h(Y)|^2 \dY\dW.
\]

We first evaluate
\[
\I_{12}
= \int_{-\infty}^0 \left(\int_W^0 e^{-Y} \dY\right)e^W|h(W)|^2 \dW
= \int_{-\infty}^0 (1-e^W)|h(W)|^2 \dW.
\]
Since $W \le 0$ we have $|1-e^W| \le 2$, and so
\begin{equation}\label{eqn: I12 est}
\I_{12}
\le 2\int_{-\infty}^0 |h(W)|^2 \dW
\le C\norm{h}_{L^2}^2.
\end{equation}

Next, we change variables in $\I_{13}$ to find
\begin{equation}\label{eqn: I13 est}
\I_{13}
= \int_{-\infty}^0 \left(\int_{-\infty}^{Y} e^{W} \dW\right) e^{-Y}|h(Y)|^2 \dY
= \int_{-\infty}^0 |h(Y)|^2 \dY
\le \norm{h}_{L^2}^2.
\end{equation}
Combining the decomposition \eqref{eqn: I1 decomp} and the estimates \eqref{eqn: I12 est} and \eqref{eqn: I13 est} gives
\[
|\I_1|
\le C\norm{h}_{L^2}^2,
\]
as desired.

%%-----------------------------------------------------------%%
%%-----------------------------------------------------------%%
%%-----------------------------------------------------------%%
%%-----------------------------------------------------------%%
%%-----------------------------------------------------------%%
\section{The Proof of Proposition \ref{prop: main contraction ests}}\label{app: proof of main contraction ests}
Our proof depends on the following lemma, which we prove in the subsequent parts of this appendix.

%%-----------------------------------------------------------%%
%%-----------------------------------------------------------%%
\begin{lemma}\label{lem: workhorse}
Let $\nu_{\M} > 0$ be as in Proposition \ref{prop: varpi-nu conv}.
There exist $C_{\Nfrakb}$, $\rho_{\Nfrakb} > 0$ such that if $0 < \nu < \nu_{\M}$ and $\norm{\eta_1}_{H_q^1}$, $\norm{\grave{\eta}_1}_{H_q^1}$, $\norm{\eta_2}_{H_q^1}$, $\norm{\grave{\eta}_2}_{H_q^1} \le \rho_{\Nfrakb}$, then the following hold.

%%-----------------------------------------------------------%%
\begin{enumerate}[label={\bf(\roman*)}, ref={(\roman*)}]

%%-----------------------------------------------------------%%
\item\label{part: mapping workhorse}
$\norm{\Nfrakb^{\nu}(\etab)}_{\X}
\le C_{\Nfrakb}\big(\nu^{1/3} + \norm{\etab}_{\X}^2\big)$.

%%-----------------------------------------------------------%%
\item\label{part: Lipschitz workhorse}
$\norm{\Nfrakb^{\nu}(\etab)-\Nfrakb^{\nu}(\grave{\etab})}_{\X}
\le C_{\Nfrakb}\big(\nu^{1/3} + \norm{\etab}_{\X} + \norm{\grave{\etab}}_{\X}\big)\norm{\etab-\grave{\etab}}_{\X}$.
\end{enumerate}
\end{lemma}

Define
\[
C_{\star} := C_{\Nfrakb}
\quadword{and}
\nu_{\star} := \frac{1}{2}\min\left\{\nu_{\M},1,\frac{1}{(1+C_{\star}^2)^6}, \frac{1}{64C_{\star}^6(1+2C_{\star})^6},\rho_{\Nfrakb}\right\} .
\]
Take $0 < \nu < \nu_{\star}$ and $\etab$, $\grave{\etab} \in \Bfrak(C_{\star}\nu)$.
Then by part \ref{part: mapping workhorse} of Lemma \ref{lem: workhorse} we have
\[
\norm{\Nfrakb^{\nu}(\etab)}_{\X}
\le C_{\Nfrakb}\big(\nu^{1/3} + \norm{\etab}_{\X}^2\big)
\le C_{\star} \big[C_{\star}(1+C_{\star}^2)\nu^{1/6}\big]\nu^{1/3}
\le C_{\star}\nu^{1/3}.
\]
This proves part \ref{part: main mapping} of Proposition \ref{prop: main contraction ests}.
Next, part \ref{part: Lipschitz workhorse} of that lemma gives
\[
\norm{\Nfrakb^{\nu}(\etab)-\Nfrakb^{\nu}(\grave{\etab})}_{\X}
\le C_{\Nfrakb}\big(\nu^{1/3}+2C_{\Nfrakb}\nu^{1/6}\big)\norm{\etab-\grave{\etab}}_{\X}
\le C_{\star}\big(1+2C_{\star})\nu^{1/6}\norm{\etab-\grave{\etab}}_{\X}
\le \frac{1}{2}\norm{\etab-\grave{\etab}}_{\X}.
\]
This proves part \ref{part: main Lipschitz} of Proposition \ref{prop: main contraction ests}.

%%-----------------------------------------------------------%%
%%-----------------------------------------------------------%%
%%-----------------------------------------------------------%%
%%-----------------------------------------------------------%%
\subsection{Auxiliary estimates}\label{app: aux ests}
Throughout this appendix we will frequently obtain estimates in terms of the $L^1$- or $L^{\infty}$-norms of a function $f \in H_q^1$.
Afterwards we can use the embedding of $L_q^2$ into $L^1$ and the corresponding inequalities
\[
\norm{f}_{L^1}
\le C\norm{f}_{L_q^2}
\le C\norm{f}_{H_q^1}
\]
for $f \in H_q^1$, as well as the Sobolev embedding, to turn these $L^1$- and $L^{\infty}$-estimates into $H_q^1$ estimates.
For brevity, we will omit those details.

It will be convenient to define the antidifferentiation operator 
\begin{equation}\label{eqn: A}
(\A{f})(X) 
:= \int_X^{\infty} f(W) \dW
\end{equation}
for $f \in L^1$.
Of course, we have
\[
\norm{\A{f}}_{L^{\infty}}
\le \norm{f}_{L^1},
\]
and we shall use this inequality frequently.
Also, if $f$ is continuous, then $\A{f}$ is differentiable and
\[
\partial_X[\A{f}]
= -f.
\]

We can use the operator $\A$ and the definitions of $\P_1^{\nu}$ in \eqref{eqn: P1-nu} and $\Rcal_1^{\nu}$ in \eqref{eqn: Rcal1-nu} to recast
\[
\P_1^{\nu}(f)(X)
= \frac{\alpha}{c_0}\A\big[\E(\nu^{1/2}S^{\nu}f)(\cdot,X)f\big](X)
\]
and
\[
\Rcal_1^{\nu}(f)(X)
= \frac{\alpha}{c_0^2\tau_1}\int_X^{\infty} \A\big[\E(\nu^{1/2}S^{\nu}f)(\cdot,V)f\big](V)(S^{\nu}f)(V) \dV.
\]

Now we begin our estimates on $\E$, $\P_1^{\nu}$, and $\Rcal_1^{\nu}$ in earnest.

%%-----------------------------------------------------------%%
%%-----------------------------------------------------------%%
\begin{lemma}\label{lem: E ests}
There exists $C > 0$ such that if $\norm{f}_{L^1}$, $\norm{\grave{f}}_{L^1} \le 1$, then the following hold.

%%-----------------------------------------------------------%%
\begin{enumerate}[label={\bf(\roman*)}, ref={(\roman*)}]

%%-----------------------------------------------------------%%
\item\label{part: E Lip}
$|\E(f)(V,X)-\E(\grave{f})(V,X)| \le C\norm{f-\grave{f}}_{L^1}$ for all $V$, $X \in \R$.

%%-----------------------------------------------------------%%
\item\label{part: E mapping}
$|\E(f)(V,X)| \le C$ for all $V$, $X \in \R$.
\end{enumerate}
\end{lemma}

\begin{proof}

%%-----------------------------------------------------------%%
\begin{enumerate}[label={\bf(\roman*)}]

%%-----------------------------------------------------------%%
\item
Since
\[
\left|\frac{\kappa}{c_0}\int_V^X f(U) \dU\right|
\le \frac{\kappa}{c_0}\norm{f}_{L^1}
\]
for all $V$, $X \in \R$, a local Lipschitz estimate on the exponential yields $C > 0$ such that if $\norm{f}_{L^1}$, $\norm{\grave{f}}_{L^1} \le 1$, then
\[
|\E(f)(V,X)-\E(\grave{f})(V,X)|
\le C\left|\int_V^X f(U) \dU - \int_V^X \grave{f}(U) \dU\right|
\le C\norm{f-\grave{f}}_{L^1}.
\]

%%-----------------------------------------------------------%%
\item
Since $\E(0) = 0$, this follows from part \ref{part: E Lip} by taking $\grave{f} = 0$.
\qedhere
\end{enumerate}
\end{proof}

The following lemma guarantees that $\P_1^{\nu}$ maps $H_q^1$ to $W^{1,\infty}$, among other results.

%%-----------------------------------------------------------%%
%%-----------------------------------------------------------%%
\begin{lemma}\label{lem: P ests}
There exists $C> 0$ such that if $0 \le \nu < 1$ and $f$, $\grave{f} \in H_q^1$ with $\norm{f}_{H_q^1}$, $\norm{\grave{f}}_{H_q^1} \le 1$, then the following hold.

%%-----------------------------------------------------------%%
\begin{enumerate}[label={\bf(\roman*)}, ref={(\roman*)}]

%%-----------------------------------------------------------%%
\item\label{part: P1-nu Lip}
$\norm{\P_1^{\nu}(f)-\P_1^{\nu}(\grave{f})}_{L^{\infty}} \le C\big(\nu^{1/2}\norm{f}_{H_q^1}+1\big)\norm{f-\grave{f}}_{H_q^1}$.

%%-----------------------------------------------------------%%
\item\label{part: P1-nu bound}
$\norm{\P_1^{\nu}(f)}_{L^{\infty}} \le C\norm{f}_{H_q^1}$.

%%-----------------------------------------------------------%%
\item\label{part: P1-nu - P0 bound}
$\norm{\P_1^{\nu}(f)-\P_1^0(f)}_{L^{\infty}} \le C\nu^{1/2}\norm{f}_{H_q^1}^2$.

%%-----------------------------------------------------------%%
\item\label{part: P1-nu - P0 deriv bound} 
$\norm{\partial_X[\P_1^{\nu}(f)-\P_1^0(f)]}_{L^{\infty}} \le C\nu^{1/2}\norm{f}_{H_q^1}^2$.

%%-----------------------------------------------------------%%
\item\label{part: P1-nu - P0 Lip}
$\norm{\big(\P_1^{\nu}(f)-\P_1^0(f)\big)-\big(\P_1^{\nu}(\grave{f})-\P_1^0(\grave{f})\big)}_{L^{\infty}}
\le C\nu^{1/2}\big(\norm{f}_{H_q^1}+\norm{\grave{f}}_{H_q^1}\big)\norm{f-\grave{f}}_{H_q^1}$.

%%-----------------------------------------------------------%%
\item\label{part: P1-nu - P0 Lip deriv}
$\norm{\partial_X[\P_1^{\nu}(f)-\P_1^0(f)]\big)-\big(\partial_X[\P_1^{\nu}(\grave{f})-\P_1^0(\grave{f})]}_{L^{\infty}}
\le C\nu^{1/2}\big(\norm{f}_{H_q^1}+\norm{\grave{f}}_{H_q^1}\big)\norm{f-\grave{f}}_{H_q^1}$.

\end{enumerate}
\end{lemma}

\begin{proof}
As we mentioned earlier, in most cases we will conclude bounds in terms of $L^{\infty}$- or $L^1$-norms, which then immediately yield the $H_q^1$-bounds stated above.

%%-----------------------------------------------------------%%
\begin{enumerate}[label={\bf(\roman*)}]

%%-----------------------------------------------------------%%
\item
We have
\[
\P_1^{\nu}(f)(X) - \P_1^{\nu}(\grave{f})(X)
= \I_1^{\nu}(f,\grave{f})(X)
+ \I_2^{\nu}(f,\grave{f})(X),
\]
where
\[
\I_1^{\nu}(f,\grave{f})(X)
:= \frac{\alpha}{c_0}\int_X^{\infty} \big(\E(\nu^{1/2}S^{\nu}f)(V,X)-\E(\nu^{1/2}S^{\nu}\grave{f})(V,X)\big)f(V) \dV
\]
and
\[
\I_2^{\nu}(f,\grave{f})(X)
:= \frac{\alpha}{c_0}\int_X^{\infty} \E(\nu^{1/2}S^{\nu}\grave{f})(V,X)\big(f(V)-\grave{f}(V)\big) \dV.
\]

We use part \ref{part: E Lip} of Lemma \ref{lem: E ests} to bound
\begin{equation}\label{eqn: E Lip2}
\big|\E(\nu^{1/2}S^{\nu}f)(V,X)-\E(\nu^{1/2}S^{\nu}\grave{f})(V,X)\big|
\le C\nu^{1/2}\norm{S^{\nu}f-S^{\nu}\grave{f}}_{L^1}
= C\nu^{1/2}\norm{f-\grave{f}}_{L^1}
\end{equation}
for all $V$, $X \in \R$.
Thus
\[
|\I_1^{\nu}(f,\grave{f})(X)|
\le C\nu^{1/2}\norm{f-\grave{f}}_{L^1}\int_X^{\infty}|f(V)| \dV
\le C\nu^{1/2}\norm{f-\grave{f}}_{L^1}\norm{f}_{L^1}
\]
for all $X \in \R$.

Next, we use part \ref{part: E mapping} of Lemma \ref{lem: E ests} to bound
\[
|\I_2^{\nu}(f,\grave{f})(X)|
\le C\int_X^{\infty} |f(V)-\grave{f}(V)| \dV
\le C\norm{f-\grave{f}}_{L^1}.
\]

%%-----------------------------------------------------------%%
\item
Since $\P^{\nu}(0) = 0$, this follows from part \ref{part: P1-nu Lip} by taking $\grave{f} = 0$.

%%-----------------------------------------------------------%%
\item
We have
\begin{equation}\label{eqn: P1-nu - P1-0}
\P_1^{\nu}(f)(X) - \P_1^0(f)(X)
= \frac{\alpha}{c_0}\int_X^{\infty} \big(\E(\nu^{1/2}S^{\nu}f)(V,X)-1\big)f(V) \dV.
\end{equation}
Since 
\[
\E(\nu^{1/2}S^{\nu}f)(V,X)-1
= \E(\nu^{1/2}S^{\nu}f)(V,X) - \E(0)(V,X),
\]
we may use part \ref{part: E Lip} of Lemma \ref{lem: E ests} to bound
\begin{equation}\label{eqn: E Lip3}
|\E(\nu^{1/2}S^{\nu}f)(V,X)-1|
\le C\nu^{1/2}\norm{S^{\nu}f}_{L^1}
= C\nu^{1/2}\norm{f}_{L^1}.
\end{equation}
Thus 
\[
|\P_1^{\nu}(f)(X) - \P_1^0(f)(X)|
\le C\nu^{1/2}\norm{f}_{L^1}\int_X^{\infty} |f(V)| \dV
\le C\nu^{1/2}\norm{f}_{L^1}^2.
\]

%%-----------------------------------------------------------%%
\item
We first differentiate under the integral and use the condition $\E(g)(X,X) = 1$, apparent from the definition of $\E$ in \eqref{eqn: E} and valid for all integrable $g$ and $X \in \R$, to calculate
\[
\partial_X[\P_1^{\nu}(f)](X)
= -\frac{\alpha}{c_0}f(X)
+ \left(\frac{\alpha}{c_0}\right)^2\nu^{1/2}\int_X^{\infty} \E(\nu^{1/2}S^{\nu}f)(V,X)f(V+\nu)f(V) \dV.
\]
Then
\begin{equation}\label{eqn: dX P1-nu - P1-0}
\partial_X[\P_1^{\nu}(f)-\P_1^0(f)](X)
= \left(\frac{\alpha}{c_0}\right)^2\nu^{1/2}\int_X^{\infty} \E(\nu^{1/2}S^{\nu}f)(V,X)f(V+\nu)f(V) \dV.
\end{equation}
Part \ref{part: E mapping} of Lemma \ref{lem: E ests} then guarantees
\[
\norm{\partial_X[\P_1^{\nu}(f)-\P_1^0(f)]}_{L^{\infty}}
\le C\nu^{1/2}\norm{f}_{L^{\infty}}\norm{f}_{L^1}.
\]

%%-----------------------------------------------------------%%
\item
We use \eqref{eqn: P1-nu - P1-0} to write
\[
\big(\P_1^{\nu}(f)-\P_1^0(f)\big)-\big(\P_1^{\nu}(\grave{f})-\P_1^0(\grave{f})\big)
= \I_3^{\nu}(f,\grave{f}) + \I_4^{\nu}(f,\grave{f}),
\]
where
\[
\I_3^{\nu}(f,\grave{f})(X)
:=  \frac{\alpha}{c_0}\int_X^{\infty} \big(\E(\nu^{1/2}S^{\nu}f)(V,X)-\E(\nu^{1/2}S^{\nu}\grave{f})(V,X)\big)f(V) \dV
\]
and
\[
\I_4^{\nu}(f,\grave{f})(X)
:=  \frac{\alpha}{c_0}\int_X^{\infty} \big(\E(\nu^{1/2}S^{\nu}\grave{f})(V,X)-1\big)\big(f(V)-\grave{f}(V)\big) \dV.
\]
We use \eqref{eqn: E Lip2} to estimate
\[
|\I_3^{\nu}(f,\grave{f})(X)|
\le C\nu^{1/2}\norm{f-\grave{f}}_{L^1}\int_X^{\infty} |f(V)| \dV
\le C\nu^{1/2}\norm{f}_{L^1}\norm{f-\grave{f}}_{L^1}.
\]
We use \eqref{eqn: E Lip3} to estimate
\[
|\I_3^{\nu}(f,\grave{f})(X)|
\le C\nu^{1/2}\norm{\grave{f}}_{L^1}\int_X^{\infty} |f(V)-\grave{f}(V)| \dV
\le C\nu^{1/2}\norm{\grave{f}}_{L^1}\norm{f-\grave{f}}_{L^1}.
\]

%%-----------------------------------------------------------%%
\item
Using \eqref{eqn: dX P1-nu - P1-0}, we have
\begin{multline*}
\partial_X[\P_1^{\nu}(f)-\P_1^0(f)](X)-\partial_X[\P_1^{\nu}(\grave{f})-\P_1^0(\grave{f})](X) \\
= \left(\frac{\alpha}{c_0}\right)^2\nu^{1/2}\int_X^{\infty} \big[\E(\nu^{1/2}S^{\nu}f)(V,X)f(V+\nu)f(V)-\E(\nu^{1/2}S^{\nu}\grave{f})(V,X)\grave{f}(V+\nu)\grave{f}(V)\big] \dV.
\end{multline*}
The estimate follows in a manner analogous to the proof of part \ref{part: P1-nu - P0 Lip} above, so we omit the details.
\qedhere
\end{enumerate}
\end{proof}

The next lemma guarantees that $\Rcal_1^{\nu}$ maps $H_q^1$ to $W^{1,\infty}$.

%%-----------------------------------------------------------%%
%%-----------------------------------------------------------%%
\begin{lemma}\label{lem: R-nu ests}
There exists $C> 0$ such that if $0 \le \nu < 1$ and $\norm{f}_{H_q^1}$, $\norm{\grave{f}}_{H_q^1} \le 1$, then the following hold.

%%-----------------------------------------------------------%%
\begin{enumerate}[label={\bf(\roman*)}, ref={(\roman*)}]

%%-----------------------------------------------------------%%
\item\label{part: R-nu bound}
$\norm{\Rcal^{\nu}(f)}_{L^{\infty}} \le C\norm{f}_{H_q^1}^2$.

%%-----------------------------------------------------------%%
\item\label{part: R-nu bound deriv}
$\norm{\partial_X[\Rcal^{\nu}(f)]}_{L^{\infty}} \le C\norm{f}_{H_q^1}^2$.

%%-----------------------------------------------------------%%
\item\label{part: R-nu Lip}
$\norm{\Rcal^{\nu}(f) - \Rcal^{\nu}(\grave{f})}_{L^{\infty}} \le C\big(\nu^{1/2}+\norm{f}_{H_q^1} + \norm{\grave{f}}_{H_q^1}\big)\norm{f-\grave{f}}_{H_q^1}$.

%%-----------------------------------------------------------%%
\item\label{part: R-nu Lip deriv}
$\norm{\partial_X[\Rcal^{\nu}(f) - \Rcal^{\nu}(\grave{f})]}_{L^{\infty}} \le C\big(\nu^{1/2}+\norm{f}_{H_q^1} + \norm{\grave{f}}_{H_q^1}\big)\norm{f-\grave{f}}_{H_q^1}$

%%-----------------------------------------------------------%%
\item
$\norm{\Rcal^{\nu}(f) - \Rcal_1^0(f)}_{L^{\infty}} \le C\nu^{1/2}\norm{f}_{H_q^1}^2$.

%%-----------------------------------------------------------%%
\item\label{part: R-nu - R-0 Lip}
$\norm{\big(\Rcal^{\nu}(f)-\Rcal_1^0(f)\big)-\big(\Rcal^{\nu}(\grave{f})-\Rcal_1^0(\grave{f})\big)}_{L^{\infty}} \le C\big(\nu^{1/2} + \norm{f}_{H_q^1} + \norm{\grave{f}}_{H_q^1}\big)\norm{f-\grave{f}}_{H_q^1}$.

%%-----------------------------------------------------------%%
\item\label{part: R-nu - R-0 bound}
$\norm{\Rcal^{\nu}(f)-\Rcal_1^0(f)}_{L^{\infty}} \le C\nu^{1/2}\norm{f}_{H_q^1}^2$.
\end{enumerate}

\end{lemma}

\begin{proof}
Throughout we will use the inequality
\[
\norm{\Rcal^{\nu}(f)}_{L^{\infty}}
\le \norm{\P_1^{\nu}(f)(S^{\nu}f)}_{L^1}.
\]
As before, we stop when we have bounds in terms of $L^1$- or $L^{\infty}$-norms.

%%-----------------------------------------------------------%%
\begin{enumerate}[label={\bf(\roman*)}]

%%-----------------------------------------------------------%%
\item
We use part \ref{part: P1-nu bound} of Lemma \ref{lem: P ests} to bound
\[
\norm{\Rcal^{\nu}(f)}_{L^{\infty}}
= C\norm{\P_1^{\nu}(f)(S^{\nu}f)}_{L^1}
\le C\norm{\P_1^{\nu}(f)}_{L^{\infty}}\norm{S^{\nu}f}_{L^1}
\le C\norm{f}_{L^1}^2.
\]

%%-----------------------------------------------------------%%
\item
We have
\[
\partial_X[\Rcal^{\nu}(f)]
= -\frac{\alpha}{c_0^2\tau_1}\P_1^{\nu}(f)(S^{\nu}f),
\]
thus
\[
\norm{\partial_X[\Rcal^{\nu}(f)]}_{L^{\infty}}
\le C\norm{\P_1^{\nu}(f)(S^{\nu}f)}_{L^{\infty}}
\le C\norm{\P_1^{\nu}(f)}_{L^{\infty}}\norm{f}_{H_q^1}
\le C\norm{f}_{H_q^1}^2
\]
by the Sobolev embedding and part \ref{part: P1-nu bound} of Lemma \ref{lem: P ests}.

%%-----------------------------------------------------------%%
\item
We use parts \ref{part: P1-nu Lip} and \ref{part: P1-nu bound} of Lemma \ref{lem: P ests} to bound
\begin{align*}
\norm{\Rcal^{\nu}(f)-\Rcal^{\nu}(\grave{f})}_{L^{\infty}}
&\le C\norm{\big(\P_1^{\nu}(f)-\P_1^{\nu}(f)\big)f}_{L^1} 
+ C\norm{\P_1^{\nu}(\grave{f})\big(S^{\nu}(f-\grave{f})\big)}_{L^1} \\
&\le C\norm{\P_1^{\nu}(f)-\P_1^{\nu}(\grave{f})}_{L^{\infty}}\norm{f}_{L^1}
+ C\norm{\P_1^{\nu}(\grave{f})}_{L^{\infty}}\norm{S^{\nu}(f-\grave{f})}_{L^1} \\
&\le C\big(\nu^{1/2}\norm{f}_{L^1} + 1\big)\norm{f}_{L^1}\norm{f-\grave{f}}_{L^1}
+ C\norm{\grave{f}}_{L^1}\norm{f-\grave{f}}_{L^1}.
\end{align*}

%%-----------------------------------------------------------%%
\item
We have
\[
\partial_X[\Rcal^{\nu}(f)-\Rcal^{\nu}(\grave{f})]
= \frac{\alpha}{c_0^2\tau_1}\P_1^{\nu}(\grave{f})(S^{\nu}\grave{f}) - \frac{\alpha}{c_0^2\tau_1}\P_1^{\nu}(f)(S^{\nu}f),
\]
thus
\[
\norm{\partial_X[\Rcal^{\nu}(f)-\Rcal^{\nu}(\grave{f})]}_{L^{\infty}}
\le C\norm{\big(\P_1^{\nu}(f)-\P_1^{\nu}(\grave{f})\big)\grave{f}}_{L^{\infty}}
+ C\norm{\P_1^{\nu}(\grave{f})\big(S^{\nu}(f-\grave{f})\big)}_{L^{\infty}}.
\]
We use part \ref{part: P1-nu Lip} of Lemma \ref{lem: P ests} and the Sobolev embedding to estimate
\[
\norm{\big(\P_1^{\nu}(f)-\P_1^{\nu}(\grave{f})\big)\grave{f}}_{L^{\infty}}
\le \norm{\P_1^{\nu}(f)-\P_1^{\nu}(\grave{f})}_{L^{\infty}}\norm{f}_{H_q^1}
\le C\big(\nu^{1/2}\norm{f}_{L^1}+1\big)\norm{f}_{H_q^1}\norm{f-\grave{f}}_{L^1}
\]
and part \ref{part: P1-nu bound} of Lemma \ref{lem: P ests} and the Sobolev embedding to estimate
\[
\norm{\P_1^{\nu}(\grave{f})\big(S^{\nu}(f-\grave{f})\big)}_{L^{\infty}}
\le C\norm{\P_1^{\nu}(\grave{f})}_{L^{\infty}}\norm{f-\grave{f}}_{L^{\infty}}
\le C\norm{f}_{L^1}\norm{f-\grave{f}}_{H_q^1}.
\]

%%-----------------------------------------------------------%%
\item
We first estimate
\begin{multline*}
\norm{\Rcal^{\nu}(f)-\Rcal_1^0(f)}_{L^{\infty}}
\le C\norm{\P_1^{\nu}(f)(S^{\nu}f) - \P_1^0(f)f}_{L^1} 
\le C\norm{\big(\P_1^{\nu}(f)-\P_1^0(f)\big)(S^{\nu}f)}_{L^1} \\
+ C\norm{\P_1^0(f)(S^{\nu}f-f)}_{L^1}.
\end{multline*}
Then part \ref{part: P1-nu - P0 bound} of Lemma \ref{lem: P ests} gives
\[
\norm{\big(\P_1^{\nu}(f)-\P_1^0(f)\big)(S^{\nu}f)}_{L^1}
\le \norm{\P_1^{\nu}(f)-\P_1^0(f)}_{L^{\infty}}\norm{S^{\nu}f}_{L^1}
\le C\nu^{1/2}\norm{f}_{L^1}^3.
\]
Next, part \ref{part: P1-nu bound} of Lemma \ref{lem: P ests} implies
\[
\norm{\P_1^0(f)(S^{\nu}f-f)}_{L^1}
\le \norm{\P_1^0(f)}_{L^{\infty}}\norm{S^{\nu}f-f}_{L^1}
\le C\norm{f}_{L^1}\norm{(S^{\nu}-1)f}_{L^1}.
\]

Since $f \in H_q^1$, we have
\[
\norm{(S^{\nu}-1)f}_{L^1}
\le C_q\norm{(S^{\nu}-1)f}_{L_q^2}.
\]
It follows from \cite[Lem.\@ A.11]{faver-wright} that
\[
\norm{(S^{\nu}-1)f}_{L_q^2}
\le C\nu\norm{f}_{H_q^1}.
\]

%%-----------------------------------------------------------%%
\item
We estimate
\begin{multline*}
\norm{\big(\Rcal^{\nu}(f)-\Rcal_1^0(f)\big)-\big(\Rcal^{\nu}(\grave{f})-\Rcal_1^0(\grave{f})\big)}_{L^{\infty}}
\le C\norm{\big(\P_1^{\nu}(f)(S^{\nu}f) - \P_1^0(f)f\big) - \big(\P^{\nu}(\grave{f})(S^{\nu}\grave{f}) - \P_1^0(\grave{f})\grave{f}\big)}_{L^1} \\
\le C\norm{\big(\P_1^{\nu}(f)-\P_1^{\nu}(\grave{f})\big)(S^{\nu}f)}_{L^1}
+ C\norm{\P_1^{\nu}(\grave{f})\big(S^{\nu}(f-\grave{f})\big)}_{L^1}
+ C\norm{\big(P_1^0(f)-\P_1^0(\grave{f})\big)\grave{f}}_{L^1}
+ C\norm{\P_1^0(f)(f-\grave{f})}_{L^1}.
\end{multline*}
We use part \ref{part: P1-nu bound} of Lemma \ref{lem: P ests} to bound
\begin{multline*}
\norm{\P_1^{\nu}(\grave{f})\big(S^{\nu}(f-\grave{f})\big)}_{L^1} + \norm{\P_1^0(f)(f-\grave{f})}_{L^1} 
\le \norm{\P_1^{\nu}(\grave{f}}_{L^{\infty}}\norm{S^{\nu}(f-\grave{f})}_{L^1}
+ \norm{\P_1^0(f)}_{L^{\infty}}\norm{f-\grave{f}}_{L^1} \\
\le C\norm{f}_{L^1}\norm{f-\grave{f}}_{L^1}.
\end{multline*}
We use part \ref{part: P1-nu Lip} of Lemma \ref{lem: P ests} to bound
\begin{multline*}
\norm{\big(\P_1^{\nu}(f)-\P_1^{\nu}(\grave{f})\big)(S^{\nu}f)}_{L^1}
+ \norm{\big(P_1^0(f)-\P_1^0(\grave{f})\big)\grave{f}}_{L^1} \\
\le \norm{\P_1^{\nu}(f)-\P_1^{\nu}(\grave{f})}_{L^{\infty}}\norm{S^{\nu}f}_{L^1}
+ \norm{\P_1^0(f)-\P_1^0(\grave{f})}_{L^{\infty}}\norm{\grave{f}}_{L^1} \\
\le C\big(\nu^{1/2}\norm{f}_{L^1}+1\big)\norm{f}_{L^1}\norm{f-\grave{f}}_{L^1}
+ C\norm{\grave{f}}_{L^1}\norm{f-\grave{f}}_{L^1}.
\end{multline*}

%%-----------------------------------------------------------%%
\item
We use part \ref{part: R-nu - R-0 Lip} with $\grave{f} = 0$.
\qedhere
\end{enumerate}
\end{proof}

Finally, we present estimates on the operators $\Ncal^{\nu}$ defined in \eqref{eqn: Ncal-nu} and $\P_2^{\nu}$ from \eqref{eqn: P2-nu}.

%%-----------------------------------------------------------%%
%%-----------------------------------------------------------%%
\begin{lemma}\label{lem: N R higher order}
There exist $C$, $\rho_0 > 0$ such that if $0 \le \nu < 1$, then the following hold.

%%-----------------------------------------------------------%%
\begin{enumerate}[label={\bf(\roman*)}, ref={(\roman*)}]

%%-----------------------------------------------------------%%
\item\label{part: N R 1}
If $f$, $\grave{f} \in H_q^1$ and $g$, $\grave{g} \in W^{1,\infty}$ with $\norm{f}_{H_q^1} + \norm{g}_{W^{1,\infty}} \le \rho_0$ and $\norm{\grave{f}}_{H_q^1} + \norm{\grave{g}}_{W^{1,\infty}} \le \rho_0$, then
\begin{multline*}
\norm{\Ncal^{\nu}(f,g) - \Ncal^{\nu}(\grave{f},\grave{g})}_{L_q^2}
+ \norm{\P_2^{\nu}(f,g) - \P_2^{\nu}(\grave{f},\grave{g})}_{H_q^1} \\
\le
C\big(\nu^{1/2} + \norm{f}_{H_q^1} + \norm{\grave{f}}_{H_q^1} + \norm{g}_{W^{1,\infty}} + \norm{\grave{g}}_{W^{1,\infty}}\big)\big(\norm{f-\grave{f}}_{H_q^1} + \norm{g-\grave{g}}_{W^{1,\infty}}\big).
\end{multline*}

%%-----------------------------------------------------------%%
\item\label{part: N R 2}
If $f \in H_q^1$ and $g \in W^{1,\infty}$ with $\norm{f}_{H_q^1} + \norm{g}_{W^{1,\infty}} \le \rho_0$, then 
\[
\norm{\Ncal^{\nu}(f,g)}_{L_q^2} 
+ \norm{\P_2^{\nu}(f,g)}_{W^{1,\infty}} \le C.
\]
\end{enumerate}
\end{lemma}

\begin{proof}
Part \ref{part: N R 2} follows from part \ref{part: N R 1} since $\Ncal^{\nu}(0,0) = \P_2^{\nu}(0,0) = 0$.
The proof of the Lipschitz estimates in part \ref{part: N R 1} follows exactly the strategies deployed above, and we would learn almost nothing new from seeing its argument, so we omit that.
The one difference here is that $\Ncal^{\nu}$ and $\P_2^{\nu}$ incorporate the maps $\nl_1^{\nu}$ and $\nl_2^{\nu}$, which were defined in \eqref{eqn: nl-nu} and which are really rational functions from $\R^2$ to $\R$.
A glance at the formulas for $\nl_1^{\nu}$ and $\nl_2^{\nu}$ provides $\rho_{\nl} > 0$ such that if $0 < \nu < 1$, then $\nl_1^{\nu}$ and $\nl_2^{\nu}$ are defined and smooth on the ball $\set{(X,Y) \in \R^2}{|X| + |Y| \le \rho_{\nl}}$.
By taking $\norm{f}_{H_q^1} + \norm{g}_{W^{1,\infty}} \le \rho_0$ for some small $\rho_0 > 0$, we can guarantee that the compositions involving $\nl_1^{\nu}$ and $\nl_2^{\nu}$ with $f$, $g$, and other operators acting on $f$ and $g$ are all defined and satisfy tame Lipschitz estimates.
\end{proof}

%%-----------------------------------------------------------%%
%%-----------------------------------------------------------%%
%%-----------------------------------------------------------%%
%%-----------------------------------------------------------%%
\subsection{Lipschitz estimates}
We first prove the Lipschitz estimates undergirding part \ref{part: Lipschitz workhorse} of Lemma \ref{lem: workhorse}, which we then use to prove the mapping estimates in part \ref{part: mapping workhorse}.
From \eqref{eqn: Nfrakb-nu}, we have $\Nfrakb^{\nu} = (\Nfrak_1^{\nu},\Nfrak_2^{\nu})$, where $\Nfrak_1^{\nu}$ was defined in \eqref{eqn: Nfrak1-nu} and $\Nfrak_2^{\nu}$ in \eqref{eqn: Nfrak2-nu}.
Using these definitions and the boundedness of the operator $\Scal$ from Proposition \ref{prop: T invert}, we can prove part \ref{part: Lipschitz workhorse} of Lemma \ref{lem: workhorse} if we show
\[
\sum_{k=1}^5 \big(\norm{\V_{1k}^{\nu}(\etab)-\V_{1k}^{\nu}(\grave{\etab})}_{H_q^1}\big) 
+ \big(\norm{\V_{21}^{\nu}(\etab)-\V_{21}^{\nu}(\grave{\etab})}_{W^{1,\infty}}\big) 
+ \norm{\V_{23}^{\nu}(\etab)-\V_{23}^{\nu}(\grave{\etab})}_{W^{1,\infty}}\big) 
\le C\Rfrak_{\star}^{\nu}(\etab,\grave{\etab}),
\]
where
\[
\Rfrak_{\star}^{\nu}(\etab,\grave{\etab})
:= \big(\nu^{1/3} + \norm{\eta_1}_{H_q^1} + \norm{\grave{\eta}_1}_{H_q^1} + \norm{\eta_2}_{W^{1,\infty}} + \norm{\grave{\eta}_2}_{W^{1,\infty}}\big)\big(\norm{\eta_1-\grave{\eta_1}}_{H_q^1} + \norm{\eta_2-\grave{\eta}_2}_{W^{1,\infty}}\big).
\]
The terms $\V_{1k}^{\nu}$ were defined in \eqref{eqn: V1k} and $\V_{2k}^{\nu}$ in \eqref{eqn: V2k}.

%%-----------------------------------------------------------%%
%%-----------------------------------------------------------%%
%%-----------------------------------------------------------%%
\subsubsection{Lipschitz estimates on $\V_{11}^{\nu}$}\label{app: V11-nu Lip}
We use the estimate on $\M^{(\nu)}-\M^{(0)}$ from Proposition \ref{prop: varpi-nu conv} to obtain
\begin{multline*}
\norm{\V_{11}^{\nu}(\etab)-\V_{11}^{\nu}(\grave{\etab})}_{H_q^1}
\le C\nu^{1/3}\norm{\big(\Rcal^{\nu}(\sigma+\eta_1)-\Rcal^{\nu}(\sigma+\grave{\eta}_1)\big)(\sigma+\grave{\eta}_1)}_{H_q^1}
+ C\nu^{1/3}\norm{\Rcal^{\nu}(\sigma+\eta_1)(\eta_1-\grave{\eta}_1)}_{H_q^1}.
\end{multline*}
We first estimate
\begin{multline*}
\norm{\big(\Rcal^{\nu}(\sigma+\eta_1)-\Rcal^{\nu}(\sigma+\grave{\eta}_1)\big)(\sigma+\grave{\eta}_1)}_{H_q^1}
\le \norm{\partial_X\big[\Rcal^{\nu}(\sigma+\eta_1)-\Rcal^{\nu}(\sigma+\grave{\eta}_1)\big](\sigma+\grave{\eta}_1)}_{L_q^2} \\
+ \norm{\big(\Rcal^{\nu}(\sigma+\eta_1)-\Rcal^{\nu}(\sigma+\grave{\eta}_1)\big)\partial_X[\sigma+\grave{\eta}_1]}_{L_q^2},
\end{multline*}
where
\begin{multline*}
\norm{\partial_X\big[\Rcal^{\nu}(\sigma+\eta_1)-\Rcal^{\nu}(\sigma+\grave{\eta}_1)\big](\sigma+\grave{\eta}_1)}_{L_q^2}
\le \norm{\partial_X\big[\Rcal^{\nu}(\sigma+\eta_1)-\Rcal_1^{\nu}(\sigma+\grave{\eta}_1)\big]}_{L^{\infty}}\norm{\sigma+\grave{\eta}_1}_{L_q^2} \\
\le C\big(\nu^{1/2}+\norm{\eta_1}_{H_q^1}+\norm{\grave{\eta}_1}_{H_q^1}\big)\norm{\eta_1-\grave{\eta}_1}_{H_q^1}
\end{multline*}
by part \ref{part: R-nu Lip deriv} of Lemma \ref{lem: R-nu ests} and
\begin{multline*}
\norm{\big(\Rcal_1^{\nu}(\sigma+\eta_1)-\Rcal_1^{\nu}(\sigma+\grave{\eta}_1)\big)\partial_X[\sigma+\grave{\eta}_1]}_{L_q^2},
\le \norm{\Rcal_1^{\nu}(\sigma+\eta_1)-\Rcal_1^{\nu}(\sigma+\grave{\eta}_1)}_{L^{\infty}}\norm{\partial_X[\sigma+\grave{\eta}_1]}_{L_q^2} \\
\le C\big(\nu^{1/2}+\norm{\eta_1}_{H_q^1}+\norm{\grave{\eta}_1}_{H_q^1}\big)\norm{\eta_1-\grave{\eta}_1}_{H_q^1}
\end{multline*}
by part \ref{part: R-nu Lip} of Lemma \ref{lem: R-nu ests}.

Next we estimate
\[
\norm{\Rcal_1^{\nu}(\sigma+\eta_1)(\eta_1-\grave{\eta}_1)}_{H_q^1}
\le \norm{\partial_X\big[\Rcal_1^{\nu}(\sigma+\eta_1)\big](\eta_1-\grave{\eta}_1)}_{L_q^2}
+ \norm{\Rcal_1^{\nu}(\sigma+\eta_1)\partial_X[\eta_1-\grave{\eta}_1]}_{L_q^2},
\]
where
\begin{multline*}
\norm{\partial_X\big[\Rcal_1^{\nu}(\sigma+\eta_1)\big](\eta_1-\grave{\eta}_1)}_{L_q^2}
\le \norm{\partial_X\big[\Rcal_1^{\nu}(\sigma+\eta_1)\big]}_{L^{\infty}}\norm{\eta_1-\grave{\eta}_1}_{L_q^2}
\le C\norm{\sigma+\eta_1}_{H_q^1}^2\norm{\eta_1-\grave{\eta}_1}_{H_q^1} \\
\le C\norm{\eta_1-\grave{\eta}_1}_{H_q^1}
\end{multline*}
by part \ref{part: R-nu bound deriv} of Lemma \ref{lem: R-nu ests} and
\begin{multline*}
\norm{\Rcal_1^{\nu}(\sigma+\eta_1)\partial_X[\eta_1-\grave{\eta}_1]}_{L_q^2}
\le \norm{\Rcal_1^{\nu}(\sigma+\eta_1)}_{L^{\infty}}\norm{\partial_X[\eta_1-\grave{\eta}_1]}_{L_q^2}
\le C\norm{\sigma+\eta_1}_{H_q^1}^2\norm{\eta_1-\grave{\eta}_1}_{H_q^1} \\
\le C\norm{\eta_1-\grave{\eta}_1}_{H_q^1}^2.
\end{multline*}
by part \ref{part: R-nu bound} of Lemma \ref{lem: R-nu ests}.

%%-----------------------------------------------------------%%
%%-----------------------------------------------------------%%
%%-----------------------------------------------------------%%
\subsubsection{Lipschitz estimates on $\V_{12}^{\nu}$}
We use the smoothing property of $\M^{(0)}$ from Lemma \ref{lem: varpi0} to bound 
\begin{multline*}
\norm{\V_{12}^{\nu}(\etab)-\V_{12}^{\nu}(\grave{\etab})}_{H_q^1}
\le C\norm{\big[\big(\Rcal_1^{\nu}(\eta_1)-\Rcal_1^0(\eta_1)\big)-\big(\Rcal_1^{\nu}(\grave{\eta}_1)-\Rcal_1^0(\grave{\eta}_1)\big)\big](\sigma+\eta_1)}_{L_q^2} \\
+ C\norm{\big(\Rcal_1^{\nu}(\grave{\eta}_1)-\Rcal_1^0(\grave{\eta}_1)\big)(\eta_1-\grave{\eta}_1)}_{L_q^2}.
\end{multline*}
Call the two $L_q^2$-norm terms above $I$ and $\II$. 
We estimate
\begin{multline*}
I
\le \norm{\big(\Rcal_1^{\nu}(\eta_1)-\Rcal_1^0(\eta_1)\big)-\big(\Rcal_1^{\nu}(\grave{\eta}_1)-\Rcal_1^0(\grave{\eta}_1)\big)}_{L^{\infty}}\norm{\sigma+\eta_1}_{L_q^2} \\
\le C\big(\nu^{1/2} + \norm{\eta_1}_{H_q^1} + \norm{\grave{\eta}_1}_{H_q^1}\big)\norm{\eta_1-\grave{\eta}_1}_{H_q^1}.
\end{multline*}
by part \ref{part: R-nu - R-0 Lip} of Lemma \ref{lem: R-nu ests} and
\[
\II
\le \norm{\Rcal_1^{\nu}(\grave{\eta}_1)-\Rcal_1^0(\grave{\eta}_1)}_{L^{\infty}}\norm{\eta_1-\grave{\eta}_1}_{L_q^2} 
\le C\nu^{1/2}\norm{\grave{\eta}_1}_{H_q^1}\norm{\eta_1-\grave{\eta}_1}_{H_q^1}
\]
by part \ref{part: R-nu - R-0 bound} of Lemma \ref{lem: R-nu ests}.

%%-----------------------------------------------------------%%
%%-----------------------------------------------------------%%
%%-----------------------------------------------------------%%
\subsubsection{Lipschitz estimates on $\V_{13}^{\nu}$}\label{app: V13-nu Lip}
We use again the smoothing property of $\M^{(0)}$ to bound 
\begin{multline*}
\norm{\V_{13}^{\nu}(\etab)-\V_{13}^{\nu}(\grave{\etab})}_{H_q^1}
\le C\norm{\big(\Rcal_1^0(\sigma+\eta_1)-\Rcal_1^0(\sigma)-D\Rcal_1^0(\sigma)\eta_1\big)\sigma}_{L_q^2} \\
\le C\norm{\Rcal_1^0(\sigma+\eta_1)-\Rcal_1^0(\sigma)-D\Rcal_1^0(\sigma)\eta_1}_{L^{\infty}}\norm{\sigma}_{L_q^2} .
\end{multline*}

Next we will use the following `difference of squares' estimate, which is proved using the fundamental theorem of calculus.
We thank J.\@ Douglas Wright for pointing out this lemma to us.

%%-----------------------------------------------------------%%
%%-----------------------------------------------------------%%
\begin{lemma}
Let $\X$ and $\Y$ be Banach spaces with $\Zcal \subseteq \X$ open and convex and with $0 \in \Zcal$.
Let $f \in \Cal^1(\Zcal,\Y)$ with $Df(0) = 0$, and suppose
\[
\Lip_{\Zcal}(Df)
:= \sup_{\substack{x,\grave{x} \in \Zcal \\ x \ne \grave{x}}} \frac{\norm{Df(x)-Df(\grave{x})}_{\b(\X,\Y)}}{\norm{x-\grave{x}}_{\X}}
< \infty.
\]
Then
\[
\norm{f(x)-f(\grave{x})}_{\Y}
\le \frac{1}{2}\Lip_{\Zcal}(Df)\big(\norm{x}_{\X} + \norm{\grave{x}}_{\X}\big)\norm{x-\grave{x}}_{\X}.
\]
\end{lemma}

We apply this lemma to $f(\eta_1) := \Rcal_1^0(\sigma)\eta_1)-\Rcal_1^0(\sigma)-D\Rcal_1^0(\sigma)\eta_1$, which is infinitely differentiable as a map from $H_q^1$ to $W^{1,\infty}$ by Remark \ref{rem: ops rem},  to conclude
\[
\norm{\Rcal_1^0(\sigma+\eta_1)-\Rcal_1^0(\sigma)-D\Rcal_1^0(\sigma)\eta_1}_{L^{\infty}}
\le C\big(\norm{\eta_1}_{H_q^1} + \norm{\grave{\eta}_1}_{H_q^1}\big)\norm{\eta_1-\grave{\eta}_1}_{H_q^1}.
\]

%%-----------------------------------------------------------%%
%%-----------------------------------------------------------%%
%%-----------------------------------------------------------%%
\subsubsection{Lipschitz estimates on $\V_{14}^{\nu}$}\label{app: V14-nu Lip}
We smooth with $\M^{(0)}$ once more, and then we use the fundamental theorem of calculus and the smoothness of $\Rcal_1^0$ to rewrite
\[
\norm{\V_{14}^{\nu}(\etab)-\V_{14}^{\nu}(\grave{\etab})}_{H_q^1}
\le C\norm{I}_{L_q^2} + C\norm{\II}_{L_q^2},
\]
where
\[
I
:= \left(\int_0^1 \big(D\Rcal_1^0(\sigma+s\eta_1)-D\Rcal_1^0(\sigma+s\grave{\eta}_1)\big) \ds\right)\eta_1^2
\]
and
\[
\II
:= \left(\int_0^1 D\Rcal_1^0(\sigma+s\grave{\eta}_1) \ds\right)(\eta_1+\grave{\eta}_1)(\eta_1-\grave{\eta}_1).
\]
Then
\[
\norm{I}_{L_q^2}
\le \bignorm{\int_0^1 \big(D\Rcal_1^0(\sigma+s\eta_1)-D\Rcal_1^0(\sigma+s\grave{\eta}_1)\big) \ds}_{L^{\infty}}\norm{\eta_1}_{L_q^2}
\]
and 
\[
\norm{\II}_{L_q^2}
\le \bignorm{\int_0^1 D\Rcal_1^0(\sigma+s\grave{\eta}_1) \ds}_{L^{\infty}}\norm{\eta_1+\grave{\eta}_1}_{L^{\infty}}\norm{\eta_1-\grave{\eta}_1}_{L_q^2}.
\]
We conclude
\[
\norm{I}_{L_q^2}
\le C\norm{\eta_1}_{H_q^1}\norm{\eta_1-\grave{\eta}_1}_{H_q^1}
\]
via a Lipschitz estimate on $D\Rcal_1^0$ and
\[
\norm{\II}_{L_q^2}
\le C\big(\norm{\eta_1}_{H_q^1}+\norm{\grave{\eta}_1}_{H_q^1}\big)\norm{\eta_1-\grave{\eta}_1}_{H_q^1}
\]
via the boundedness of $D\Rcal_1^0$.

%%-----------------------------------------------------------%%
%%-----------------------------------------------------------%%
%%-----------------------------------------------------------%%
\subsubsection{Lipschitz estimates on $\V_{15}^{\nu}$}\label{app: V15-nu Lip}
We smooth with $\M^{(0)}$ to estimate
\[
\norm{\V_{15}^{\nu}(\etab)-\V_{15}^{\nu}(\grave{\etab})}_{H_q^1}
\le C\nu^{1/2}\norm{\Ncal^{\nu}(\sigma+\eta_1,\zeta+\eta_2)-\Ncal^{\nu}(\sigma+\grave{\eta}_1,\zeta+\grave{\eta}_2)}_{L_q^2}.
\]
The desired estimate then follows from part \ref{part: N R 1} of Lemma \ref{lem: N R higher order}.

%%-----------------------------------------------------------%%
%%-----------------------------------------------------------%%
%%-----------------------------------------------------------%%
\subsubsection{Lipschitz estimates on $\V_{21}^{\nu}$}
This is a direct application of parts \ref{part: P1-nu - P0 Lip} and \ref{part: P1-nu - P0 Lip deriv} of Lemma \ref{lem: P ests}.

%%-----------------------------------------------------------%%
%%-----------------------------------------------------------%%
%%-----------------------------------------------------------%%
\subsubsection{Lipschitz estimates on $\V_{22}^{\nu}\circ\Nfrak_1^{\nu}$}\label{app: V22-nu Lip}
We have
\begin{equation}\label{eqn: V22-nu circ Nfrak1-nu}
\V_{22}^{\nu}(\Nfrak_1^{\nu}(\etab))(X)
= \P_1^0(\sigma+\Nfrak_1^{\nu}(\etab))(X)-\P_1^0(\sigma)(X)
= \frac{\alpha}{c_0}\int_X^{\infty} \Nfrak_1^{\nu}(\etab)(V) \dV.
\end{equation}
The desired Lipschitz estimate on $\V_{22}^{\nu}$ then follows at once from the Lipschitz estimate 
\[
\norm{\Nfrak_1^{\nu}(\etab)-\Nfrak_1^{\nu}(\grave{\etab})}_{H_q^1}
\le C\Rfrak_{\star}^{\nu}(\etab,\grave{\etab}),
\]
which we proved in Appendix \ref{app: V11-nu Lip} through \ref{app: V15-nu Lip}.
Without having substituted $\Nfrak_1^{\nu}(\etab)$ for $\eta_1$ in the process of defining $\Nfrak_2^{\nu}$ in \eqref{eqn: Nfrak2-nu}, we would have only a useless $\O(1)$ estimate here.

%%-----------------------------------------------------------%%
%%-----------------------------------------------------------%%
%%-----------------------------------------------------------%%
\subsubsection{Lipschitz estimates on $\V_{23}^{\nu}$}
This is a direct application of part \ref{part: N R 1} of Lemma \ref{lem: N R higher order}.

%%-----------------------------------------------------------%%
%%-----------------------------------------------------------%%
%%-----------------------------------------------------------%%
%%-----------------------------------------------------------%%
\subsection{Mapping estimates}\label{app: map}
We prove the mapping estimates that deliver part \ref{part: mapping workhorse} of Lemma \ref{lem: workhorse} and rely mostly on the preceding Lipschitz estimates.
Due to the boundedness of $\Scal$, it suffices to show
\[
\sum_{k=1}^5 \norm{\V_{1k}^{\nu}(\etab)}_{H_q^1} 
+ \norm{\V_{21}^{\nu}(\etab)}_{W^{1,\infty}} 
+ \norm{\V_{23}^{\nu}(\etab)}_{W^{1,\infty}}
\le C\big(\nu^{1/3} + \norm{\eta_1}_{H_q^1}^2 + \norm{\eta_2}_{W^{1,\infty}}^2\big).
\]

%%-----------------------------------------------------------%%
%%-----------------------------------------------------------%%
%%-----------------------------------------------------------%%
\subsubsection{Mapping estimates on $\V_{11}^{\nu}$}\label{app: V11-nu map}
We estimate
\[
\norm{\V_{11}^{\nu}(\etab)}_{H_q^1}
\le \norm{\V_{11}^{\nu}(\etab) - \V_{11}^{\nu}(0)}_{H_q^1} + \norm{\V_{11}^{\nu}(0)}_{H_q^1},
\]
where 
\[
\norm{\V_{11}^{\nu}(\etab) - \V_{11}^{\nu}(0)}_{H_q^1}
\le C\nu^{1/3}\norm{\eta_1}_{H_q^1}^2
\]
by the Lipschitz estimates in Appendix \ref{app: V11-nu Lip} and
\[
\norm{\V_{11}^{\nu}(0)}_{H_q^1}
= \norm{\big(\M^{(\nu)}-\M^{(0)}\big)\big[\Rcal_1^{\nu}(\sigma)\sigma\big]}_{H_q^1}
\le C\nu^{1/3}
\]
by Proposition \ref{prop: varpi-nu conv}.

%%-----------------------------------------------------------%%
%%-----------------------------------------------------------%%
%%-----------------------------------------------------------%%
\subsubsection{Mapping estimates on $\V_{12}^{\nu}$}
We estimate
\[
\norm{\V_{12}^{\nu}(\etab)}_{L_q^2}
\le \norm{\V_{12}^{\nu}(\etab) - \V_{12}^{\nu}(0)}_{L_q^2} + \norm{\V_{12}^{\nu}(0)}_{L_q^2},
\]
where 
\[
\norm{\V_{12}^{\nu}(\etab) - \V_{12}^{\nu}(0)}_{L_q^2}
\le C
\]
by the Lipschitz estimates in Appendix \ref{app: V11-nu Lip} and
\[
\norm{\V_{12}^{\nu}(0)}_{L_q^2}
= \norm{\big(\Rcal_1^{\nu}(\sigma)-\Rcal_1^0(\sigma)\big)\sigma}_{L_q^2}
\le \norm{\Rcal_1^{\nu}(\sigma)-\Rcal_1^0(\sigma)}_{L^{\infty}}\norm{\sigma}_{L_q^2}
\le C\nu^{1/2}
\]
by part \ref{part: R-nu - R-0 bound} of Lemma \ref{lem: R-nu ests}.

%%-----------------------------------------------------------%%
%%-----------------------------------------------------------%%
%%-----------------------------------------------------------%%
\subsubsection{Mapping estimates on $\V_{13}^{\nu}$}
Because $\V_{13}^{\nu}(0) = 0$, these follow from the Lipschitz estimates for $\V_{13}^{\nu}$ that we developed above in Appendix \ref{app: V13-nu Lip}.

%%-----------------------------------------------------------%%
%%-----------------------------------------------------------%%
%%-----------------------------------------------------------%%
\subsubsection{Mapping estimates on $\V_{14}^{\nu}$}
Because $\V_{14}^{\nu}(0) = 0$, these follow from the Lipschitz estimates for $\V_{14}^{\nu}$ that we developed above in Appendix \ref{app: V13-nu Lip}.

%%-----------------------------------------------------------%%
%%-----------------------------------------------------------%%
%%-----------------------------------------------------------%%
\subsubsection{Mapping estimates on $\V_{15}^{\nu}$}\label{app: V15-nu map}
The estimates are analogous to those in Appendix \ref{app: V11-nu map}, except now we use Lemma \ref{lem: N R higher order} instead of the Lipschitz estimates in Appendix \ref{app: V11-nu Lip}.

%%-----------------------------------------------------------%%
%%-----------------------------------------------------------%%
%%-----------------------------------------------------------%%
\subsubsection{Mapping estimates on $\V_{21}^{\nu}$}
These estimates follow directly from parts \ref{part: P1-nu - P0 bound} and \ref{part: P1-nu - P0 deriv bound} of Lemma \ref{lem: P ests}.

%%-----------------------------------------------------------%%
%%-----------------------------------------------------------%%
%%-----------------------------------------------------------%%
\subsubsection{Mapping estimates on $\V_{22}^{\nu} \circ \Nfrak_1^{\nu}$}
We obtain these estimates by first rewriting $\V_{22}^{\nu} \circ \Nfrak_1^{\nu}$ via the identity \eqref{eqn: V22-nu circ Nfrak1-nu} and then using the mapping estimates on $\Nfrak_1^{\nu}$ developed in Appendices \ref{app: V11-nu map} through \ref{app: V15-nu map}.

%%-----------------------------------------------------------%%
%%-----------------------------------------------------------%%
%%-----------------------------------------------------------%%
\subsubsection{Mapping estimates on $\V_{23}^{\nu}$}
This estimate follows from part \ref{part: N R 2} of Lemma \ref{lem: N R higher order}.

%%-----------------------------------------------------------%%
%%-----------------------------------------------------------%%
%%-----------------------------------------------------------%%
%%-----------------------------------------------------------%%
%%-----------------------------------------------------------%%

\bibliographystyle{siam}
\bibliography{auxin_bib}

\subsection*{Supplementary Video S1} Wavetrain simulation for the expanded system
\eqref{eqn:int:main:sys:expanded}, corresponding to Fig~\ref{fig:int:wave:trains}. Higher amplitude pulses travel faster than lower amplitude pulses, in correspondence with the scaling relations \eqref{eq:int:scaling:relations}. These speed differences lead to merge events where even higher pulses are formed, which detach from the bulk.
We used the procedure described in {\S}\ref{sec:sub:int:mr}, taking $A_1(0) = A_{\diamond} = 0.0$ but adding $0.025$ to $\dot{A}_1(t)$ to simulate a constant auxin influx at the left boundary. 
We picked $\delta = 0.1$ and $k_2 = 0.2$,
leaving the remaining parameters
from Fig. \ref{fig:int:profile} unchanged.

\end{document}